\documentclass{amsart}
\address{Department of Mathematics, Kyoto
University, Kitashirakawa, Kyoto, 606-8502, Japan } \email{fukaya@math.kyoto-u.ac.jp}
\address{Department of Mathematics, University of
Wisconsin, Madison, WI, USA and
National Institute for Mathematical Sciences, Daeduk
Boulevard 628, Yuseong-gu, Daejeon, 305-340, Korea.} \email{oh@math.wisc.edu}
\address{Graduate School of Mathematics,
Nagoya University, Chikusa, Nagoya, 464-8602, Japan \& 
Korea Institute for Advanced Study, Seoul, Korea,} \email{ohta@math.nagoya-u.ac.jp}
\address{Department of Mathematics,
Hokkaido University, Sapporo, 060-0810, Japan \& 
Korea Institute for Advanced Study, Seoul, Korea,}
\email{ono@math.sci.hokudai.ac.jp}

\usepackage[all]{xy}

\def\E{\ifmmode{\mathbb E}\else{$\mathbb E$}\fi} 
\def\N{\ifmmode{\mathbb N}\else{$\mathbb N$}\fi} 
\def\R{\ifmmode{\mathbb R}\else{$\mathbb R$}\fi} 
\def\Q{\ifmmode{\mathbb Q}\else{$\mathbb Q$}\fi} 
\def\C{\ifmmode{\mathbb C}\else{$\mathbb C$}\fi} 
\def\H{\ifmmode{\mathbb H}\else{$\mathbb H$}\fi} 
\def\Z{\ifmmode{\mathbb Z}\else{$\mathbb Z$}\fi} 
\def\P{\ifmmode{\mathbb P}\else{$\mathbb P$}\fi} 
\def\T{\ifmmode{\mathbb T}\else{$\mathbb T$}\fi} 
\def\SS{\ifmmode{\mathbb S}\else{$\mathbb S$}\fi} 
\def\DD{\ifmmode{\mathbb D}\else{$\mathbb D$}\fi} 
\def\K{\ifmmode{\mathbb K}\else{$\mathbb K$}\fi}

\newcommand{\e}{\varepsilon}

\newcommand{\del}{\partial}

\newcommand{\ben}{\begin{enumerate}}
\newcommand{\een}{\end{enumerate}}
\newcommand{\be}{\begin{equation}}
\newcommand{\ee}{\end{equation}}
\newcommand{\bea}{\begin{eqnarray}}
\newcommand{\eea}{\end{eqnarray}}
\newcommand{\beastar}{\begin{eqnarray*}}
\newcommand{\eeastar}{\end{eqnarray*}}
\newcommand{\bc}{\begin{center}}
\newcommand{\ec}{\end{center}}

\theoremstyle{theorem}
\newtheorem{thm}{Theorem}[section]
\newtheorem{cor}[thm]{Corollary}
\newtheorem{lem}[thm]{Lemma}
\newtheorem{sublem}[thm]{Sublemma}
\newtheorem{prop}[thm]{Proposition}

\theoremstyle{definition}
\newtheorem{defn}[thm]{Definition}
\newtheorem{rem}[thm]{Remark}

\newtheorem{exm}[thm]{Example}
\newtheorem{conds}[thm]{Condition}
\newtheorem{prob}[thm]{Problem}

\newtheorem*{thm*}{Theorem}

\numberwithin{equation}{section}

\def\R{{\mathbb R}}

\def\E{{\mathbb E}}
\def\Z{{\mathbb Z}}
\def\C{{\mathbb C}}
\def\R{{\mathbb R}}
\def\P{{\mathbb P}}

\def\N{{\mathbb N}}

\def\11{{\mathbb I}}

\def\C{\mathbb{C}}
\def\Z{\mathbb{Z}}

\def\T{\mathbb{T}}

\def\Q{\mathbb{Q}}

\def\E{\ifmmode{\mathbb E}\else{$\mathbb E$}\fi} 
\def\N{\ifmmode{\mathbb N}\else{$\mathbb N$}\fi} 
\def\R{\ifmmode{\mathbb R}\else{$\mathbb R$}\fi} 
\def\Q{\ifmmode{\mathbb Q}\else{$\mathbb Q$}\fi} 
\def\C{\ifmmode{\mathbb C}\else{$\mathbb C$}\fi} 
\def\H{\ifmmode{\mathbb H}\else{$\mathbb H$}\fi} 
\def\Z{\ifmmode{\mathbb Z}\else{$\mathbb Z$}\fi} 
\def\P{\ifmmode{\mathbb P}\else{$\mathbb P$}\fi} 
\def\SS{\ifmmode{\mathbb S}\else{$\mathbb S$}\fi} 
\def\DD{\ifmmode{\mathbb D}\else{$\mathbb D$}\fi} 

\def\R{{\mathbb R}}

\def\E{{\mathbb E}}
\def\Z{{\mathbb Z}}
\def\C{{\mathbb C}}
\def\R{{\mathbb R}}

\def\N{{\mathbb N}}

\def\e{\varepsilon}

\def\CA{{\mathcal A}}

\def\CK{{\mathcal K}}

\def\CM{{\mathcal M}}

\def\darr#1{\raise1.5ex\hbox{$\leftrightarrow$}
\mkern-16.5mu #1}

\def\roughly#1{\raise.3ex\hbox{$#1$\kern-.75em
\lower1ex\hbox{$\sim$}}}

\def\opname#1{\mathop{\kern0pt{\rm #1}}\nolimits}

\def\dim{\opname{dim}}

\begin{document}

\def\mq{\mathfrak{q}}
\def\mp{\mathfrak{p}}
\def\mH{\mathfrak{H}}
\def\mh{\mathfrak{h}}
\def\ma{\mathfrak{a}}
\def\ms{\mathfrak{s}}
\def\mm{\mathfrak{m}}
\def\mn{\mathfrak{n}}
\def\mz{\mathfrak{z}}
\def\mw{\mathfrak{w}}
\def\Hoch{{\tt Hoch}}
\def\mt{\mathfrak{t}}
\def\ml{\mathfrak{l}}
\def\mT{\mathfrak{T}}
\def\mL{\mathfrak{L}}
\def\mg{\mathfrak{g}}
\def\md{\mathfrak{d}}
\def\mr{\mathfrak{r}}

\title[Lagrangian Floer theory on compact toric manifolds II]
{Lagrangian Floer theory on compact toric manifolds II: Bulk
deformations.}

\author[K. Fukaya, Y.-G. Oh, H. Ohta, K.
Ono]{Kenji Fukaya, Yong-Geun Oh, Hiroshi Ohta, Kaoru Ono}
\thanks{KF is supported partially by JSPS Grant-in-Aid for Scientific Research
No.18104001 and by Global COE Program G08, YO by US NSF grant \# 0904197, HO by JSPS Grant-in-Aid
for Scientific Research No.19340017, and KO by JSPS Grant-in-Aid for
Scientific Research, Nos. 18340014 and 21244002}

\begin{abstract} This is a continuation of part I in the series of
the papers on Lagrangian Floer theory on toric manifolds. Using the
deformations of Floer cohomology by the ambient cycles, which we
call \emph{bulk deformations}, we find a continuum of
non-displaceable Lagrangian fibers on some compact toric manifolds.
We also provide a method of finding all fibers 
with non-vanishing Floer cohomology with bulk deformations in arbitrary
compact toric manifolds, which we call \emph{bulk-balanced}
Lagrangian fibers.
\end{abstract}

\date{Feb. 25, 2011}



\maketitle \tableofcontents
\section{Introduction}
\label{sec:introduction}

This is the second of series of papers to study Lagrangian Floer
theory on toric manifolds. The main purpose of this paper is to
explore bulk deformations of Lagrangian Floer theory, which we
introduced in Section 3.8 \cite{fooo06} (=Section 13 \cite{fooo06pre}) and draw its applications. In
particular, we prove the following Theorems \ref{thm:bupP2},
\ref{thm:leading}. We call a Lagrangian submanifold $L$ of a
symplectic manifold $X$ {\it non-displaceable} if $\psi(L) \cap L
\ne \emptyset$ for any Hamiltonian diffeomorphism $\psi: X \to X$.

\begin{thm}\label{thm:bupP2}
Let $X_k$ be the $k$-points blow up of $\C P^2$ with $k \ge 2$. Then
there exists a toric K\"ahler structure on $X_k$ such that there
exist a continuum of non-displaceable Lagrangian fibers $L(u)$.
\par
Moreover they have the following property:
If $\psi: X \to X$ is a Hamiltonian isotopy such that $\psi(L(u))$ is
transversal to $L(u)$ in addition, then
$$
\# (\psi(L(u)) \cap L(u)) \ge 4.
$$
\end{thm}
\begin{rem}\label{rem;thm1}
\begin{enumerate}
\item We state Theorem \ref{thm:bupP2} in the case of
the blow up of $\C P^2$. We can construct many similar examples by the same method.
\item We will prove Theorem \ref{thm:bupP2} by proving the
existence of $\mathfrak b \in H^2(X_k;\Lambda_+)$ and
$\mathfrak x \in H(L(u);\Lambda_0)$ such that
\begin{equation}\label{bulkHF;form}
HF((L(u),(\mathfrak b,\mathfrak x)),(L(u),(\mathfrak b,\mathfrak
x));\Lambda_0) \cong H(T^2;\Lambda_0).
\end{equation}
Here
\begin{equation}\label{lambda0}
\Lambda_0 =\left\{\sum_{i=1}^{\infty} a_i T^{\lambda_i}
\in \Lambda \, \Big\vert \,
\lambda_i \ge 0, \lim_{i\to\infty}\lambda_i = \infty, a_i \in R\right\},
\end{equation}
($R$ is a field of characteristic $0$) and
\begin{equation}\label{lambda+}
\Lambda_+ = \left\{ \sum_{i=1}^{\infty} a_i T^{\lambda_i}
\in \Lambda \, \Big\vert \, \lambda_i > 0\right\}
\end{equation}
are the
universal Novikov ring and its maximal ideal. The left hand side
of (\ref{bulkHF;form}) is the Floer cohomology with bulk
deformation. See Section 3.8 \cite{fooo06} (= Section 13 \cite{fooo06pre}) and Section \ref{sec:bulk}
of this paper for its definition.
\item In a sequel of this series of papers \cite{fooospectral}, we will study this
example further and prove that the universal cover
$\widetilde{Ham}(X_k)$ of the group of Hamiltonian diffeomorphisms
allows infinitely many continuous and homogeneous Calabi
quasi-morphisms $\varphi_u: \widetilde{Ham}(X_k) \to \R$ (see
\cite{entov-pol03}) such that for any finitely many $u_1,\ldots,u_N$
there exists a subgroup $\cong \Z^N \subset \widetilde{Ham}(X_k)$ on
which $(\varphi_{u_1},\ldots,\varphi_{u_N}): \Z^N \to \R^N$ is
injective.
\end{enumerate}
\end{rem}
In Sections 9 and 10 \cite{fooo08}, we introduced the notion of
leading term equation for each Lagrangian fiber $L(u)$ of a toric
manifold $X$. See also Section \ref{sec:Ellim} of this paper. The
leading term equation is a system consisting of $n$-elements of the
Laurent polynomial ring
$\C[y_1,\ldots,y_n,y_1^{-1},\ldots,y_n^{-1}]$ of $n$ variables.
(Here $n = \dim L(u)$.) In Section 10 \cite{fooo08}, we proved that
if the leading term equation has a solution in $(\C \setminus
\{0\})^n$ then $L(u)$ has a nontrivial Floer cohomology for some
bounding cochain $\mathfrak x$ in $H^1(L(u);\Lambda_0)$ {\it under
certain nondegeneracy conditions}. The next theorem says that if we
consider more general class of Floer cohomology integrating bulk
deformations into its construction, we can remove this nondegeneracy
condition.
\begin{thm}\label{thm:leading}
Let $X$ be a compact toric manifold and $L(u)$ its Lagrangian
fiber. Suppose that the leading term equation of $L(u)$ has a
solution in $(\C \setminus \{0\})^n$.
\par
Then there exists $\mathfrak b \in H^2(X;\Lambda_+)$ and
$\mathfrak x \in H(L(u);\Lambda_0)$ satisfying
\begin{equation}\label{bulkHF;form2}
HF((L(u),(\mathfrak b,\mathfrak x)),(L(u),(\mathfrak b,\mathfrak
x));\Lambda_0) \cong H(T^n;\Lambda_0).
\end{equation}
\end{thm}
We remark that it is proved in \cite{fooo08} that there always exists
$u$ such that the leading term equation of $L(u)$ has a
solution in $(\C \setminus \{0\})^n$.
\begin{cor}\label{Cormainapapl}
Let $X$ be a compact toric manifold and $L(u)$ its Lagrangian
fiber. Suppose that the leading term equation of $L(u)$ has a
solution in $(\C \setminus \{0\})^n$. Then $L(u)$ is
non-displaceable.
\par
Moreover $L(u)$ has  the following property.
If $\psi: X \to X$ is a Hamiltonian isotopy such that $\psi(L(u))$ is transversal to $L(u)$, then
\begin{equation}\label{2topowern}
\# (\psi(L(u)) \cap L(u)) \ge 2^n,
\end{equation}
where $n = \dim L(u)$.
\end{cor}
\par
The converse to Theorem \ref{thm:leading} also holds. (See Theorem
\ref{ellim}.)
\par
The leading term equation can be easily solved in practice for most of the
compact toric manifolds, which are not necessarily Fano.
Theorem $\ref{thm:leading}$ enables us to reduce the problem to
locate all $L(u)$ such that there exists a pair $(\mathfrak b,\mathfrak x)
\in H^2(X;\Lambda_+)\times H(L(u);\Lambda_0)$ satisfying
(\ref{bulkHF;form2}) to the problem to decide existence of
nonzero solution of explicitly calculable system of polynomial equations.
In \cite{fooo08} we provided such a reduction
for the case $\mathfrak b =0$. If all the solutions of the leading
term equation are weakly nondegenerate (see Definition 10.2
\cite{fooo08}), Floer cohomology with $\mathfrak b=0$ seems to
provide enough information for the general study of non-displacement of Lagrangian fibers.
The method employed in this paper works for arbitrary compact
toric manifolds without nondegeneracy assumption, and the calculation
is actually simpler. We hope that this method might provide an
optimal result on the non-displacement of Lagrangian fibers. (See
Problem \ref{conjdisp}.)
\begin{rem}
In \cite{cho07}, Cho used Floer cohomology with `$B$-field' to study
non-displacement of Lagrangian fibers in toric manifolds.
`$B$-field' which Cho used is parameterized by
$H^2(X;\sqrt{-1}\,\R)$. The bulk deformation we use in this paper is
parameterized by $\mathfrak b \in H^*(X;\Lambda_0)$. If we restrict to $\mathfrak
b \in H^2(X;\sqrt{-1}\,\R)$ our bulk deformation by $\mathfrak b$ in
this paper coincides with the deformation by a `$B$-field' in
\cite{cho07}.
\end{rem}
A brief outline of each section of the paper is now in order. In
Section \ref{sec:bulk}, we review construction of the operator
$\mathfrak q$ given in Section 3.8 \cite{fooo06} (=Section 13 \cite{fooo06pre}) and explain how we use
$\mathfrak q$ to deform Floer cohomology. In Section \ref{sec:Potbul} we
provide a more explicit description thereof for the case of compact
toric manifolds and study its relation to the potential function
with bulk, which is the generating function defined by the structure
constants of $\mathfrak q$. This section also contains various results
on the operator $\mathfrak q$ and on the potential function with bulk.
These results will be used also in sequels of this series
of papers \cite{fooo10} \cite{fooospectral}. In Section \ref{sec:Ellim}, we explain how we use the
results of Section \ref{sec:Potbul} to study Floer cohomology of
Lagrangian fibers of compact toric manifolds. Especially we prove
Theorem \ref{thm:leading} there. Section \ref{sec:Exam} is devoted
to the proof of Theorem \ref{thm:bupP2}.  In this section we discuss
the case of two points blow up of $\C P^2$ in detail. The calculation we
perform in this section can be generalized to arbitrary compact
toric manifolds. In Section \ref{sec:operatorq} we describe the
results on the moduli space of pseudo-holomorphic discs with
boundary on a Lagrangian fiber of a general toric manifold, which are
basically due to \cite{cho-oh}. We use these results in the study of
the operator $\mathfrak q$. In Section \ref{sec:CalPot} we carry out
some calculation of the potential function with bulk deformation
which is strong enough to locate all the Lagrangian fibers with
nontrivial Floer cohomology (after bulk deformation).
\par
In Section \ref{sec:HFbulk} we use the Floer cohomology with bulk
deformation in the study of non-displacement of Lagrangian
submanifolds. For this purpose we define the cohomology between a
pair of Lagrangian submanifolds $L$ and $\psi(L)$ for a Hamiltonian
diffeomorphism $\psi$. We also show that this Floer cohomology of
the pair is isomorphic to the Floer cohomology of $L$ itself. This
is a standard process one takes to use Floer cohomology for the
non-displacement problem dating back to Floer \cite{floer:morse}. We
include bulk deformations and deformations by bounding cochain
there. These results were previously obtained in \cite{fooo06}.
However we give rather detailed account of these constructions here
in order to make this paper as self-contained as possible. To avoid
too much overlap with that of \cite{fooo06}, in this paper we give a
proof using the de Rham cohomology version here which is different
from that of \cite{fooo06} in which we used the singular cohomology
version. In Section \ref{sec:domain} we study the convergence
property of potential functions. Namely we prove that the potential
function is contained in the completion of the ring of Laurent
polynomials over a Novikov ring with respect to an appropriate
non-Archimedean norm. This choice of the norm depends on the
K\"ahler structure (or equivalently on the moment polytope). We discuss
the natural way to take completion (in Section \ref{sec:Potbul}) and
show that our potential function actually converges in that sense
(in Section \ref{sec:domain}).  In Section \ref{sec:euler}, we
discuss the relation of the Euler vector field and the potential
function. In Section \ref{sec:byLambda0}, we slightly enlarge the
parameter space of bulk deformations including $\mathfrak b$ from
$H(X;\Lambda_0)$ not just from $H(X;\Lambda_+)$. In Section
\ref{sec:integration}, we review the construction of smooth
correspondence in de Rham cohomology using continuous family of
multisections and integration along fibers via its zero sets.
\par\smallskip
The authors thank the anonymous referees for 
their careful and thorough reading and 
for their suggestions which help to 
improve presentation of the paper. 
\par\bigskip
\centerline{\bf Notations and conventions}
\par\medskip
We take any field $R$ containing $\Q$.
The universal Novikov ring $\Lambda_0$ is defined as (\ref{lambda0}),
where $a_i \in R$.
Its ideal $\Lambda_+$ is defined as (\ref{lambda+}).
$$
\Lambda = \left\{ \sum_{i=1}^{\infty} a_i T^{\lambda_i}
\, \Big\vert \,  a_i \in R, \lambda_i \in \R, \,
\lambda_i < \lambda_{i+1},\,\lim_{i\to\infty}\lambda_i = \infty \right\}
$$
is the field of fraction of $\Lambda_0$.
\par
In case we need to specify $R$ we write
$\Lambda_0(R)$, $\Lambda_+(R)$, $\Lambda(R)$.
The (non-Archimedean) valuation $\mathfrak v_T$ on them are defined by
$$
\mathfrak v_T\left(\sum_{i=1}^{\infty} a_i T^{\lambda_i}\right)
= \inf \{ \lambda_i \mid a_i \ne 0\}, \quad 
\mathfrak v_T(0)=\infty.
$$
It induces a non-Archimedean norm
$\Vert x\Vert = e^{-\mathfrak v_T(x)}$ and defines a topology on them.
Those rings are complete with respect to this norm.
\par
If $C$ is an $R$ vector space, we denote by $C(\Lambda_0)$ the completion
of $C\otimes \Lambda_0$ with respect to the non-Archimedean topology of $\Lambda_0$.
In other words its elements are of the form
$$
\sum a_i T^{\lambda_i}
$$
such that $a_i \in C$, $\lambda_i < \lambda_{i+1}$, $\lambda_i \ge 0$, $\lim_{i\to\infty}\lambda_i = \infty$.
$C(\Lambda_+)$, $C(\Lambda)$, $C(\Lambda_0(R))$, $C(\Lambda_+(R))$, $C(\Lambda(R))$
are defined in the same way.
\par
Let $N$ be an $n$-dimensional lattice $\Z^n$ and $N_\R = N
\otimes_\Z \R \cong \R^n$. We denote by $X = X_\Sigma$ the compact
toric manifold associated to a complete $n$-dimensional fan of
regular cones in $N_\R$. We denote by $G(\Sigma) = \{v_1,\dots,
v_m\}$ the set of 1-dimensional generators of $\Sigma$. Let $\omega$
be any $T^n \subset (\C^*)^n$-invariant a K\"ahler form $\omega$ on
$X$. We denote its moment map by $\pi: X \to M_\R \cong \R^n$ and
the corresponding moment polytope by $P \subset M_\R \cong \R^n$. The
boundary $\partial P$ is divided into $m$ codimension $1$ faces,
which we denote by $\partial_i P$ for $i=1,\ldots,\, m$. We denote by $D_i = \pi^{-1}(\partial_i P)$
the associated toric divisor and by $\CA^2(\Z)$ the free abelian group
generated by $D_i$'s.
\par
For any given fiber $\pi^{-1}(u) =: L(u)$ for $u \in \text{Int} \,
P$, we have an $\R$ linear isomorphism
\begin{equation}\label{H1toN}
H_1(L(u);\R) \to N_{\R}
\end{equation}
which is defined by tensoring $\R$ with the farthest right column of the
following commutative diagram
\be
\xymatrix{
0 \ar[r] & \ker \pi \ar[r]^i\ar[d]_\cong & \Z^m \ar[r]^\pi\ar[d]_{\cong} & N \ar[r]\ar[d]_{\cong} & 0 \\
0 \ar[r] & H_2(X,\Z) \ar[r] & H_2(X, X \setminus D;\Z) \ar[r]^{\del}
&
H_1(X \setminus D;\Z) \ar[r] & 0 \\
0 \ar[r] & H_2(X,\Z) \ar[r]\ar[u]^= & H_2(X, L(u);\Z)
\ar[r]\ar[u]^\cong & H_1(L(u);\Z) \ar[r]\ar[u]^\cong & 0}
\label{eq:diagram}
\ee
where $\pi$ is defined by $\pi(e_i) = v_i$ for a $\Z$-basis
$\{e_1,\dots, e_m\}$ of $\Z^m$. Here the two isomorphisms in the
lower row are induced by the canonical inclusion map $L(u)
\hookrightarrow X \setminus D$ and the central isomorphism in the top
row is induced by the composition of the isomorphism
$$
\Z^m \cong \CA^2(\Z) \cong H^2(X,X\setminus D;\Z)
$$
in which the first isomorphism is induced by the assignment $e_i \mapsto D_i$.
We denote by $\beta_i \in  H_2(X,L(u)) \cong H_2(X, X \setminus D;\Z)$ the
image of $e_i \mapsto [D_i] \in H^2(X,\Z)$ with $D_i = \pi^{-1}(\partial_i P)$.

In \cite{cho-oh}, \cite{fooo08}, we used the affine functions
$\ell_i: M_{\R} \to \R$ and representation of the moment
polytope and its boundary
\be\label{def;polytope}
P = \{u  \in M_{\R} \mid \ell_i(u) \ge 0,\,\,\, i=1,\ldots,m \},
\qquad \partial_i P = \{ u \in M_{\R} \mid \ell_i(u) = 0\}
\ee
which will also play an important role in the present paper. We have
canonical identification 
\begin{equation}\label{eq:v}
v_i = (v_{i,1},\dots v_{i,n}) = d\ell_i \in N_{\R} \cong
H_1(L(u);\R), \quad i=1,\dots , m.
\end{equation} 
Then $v_i$ becomes an integral vector i.e.,
$v_i\in H_1(L(u);\Z)$ and $v_i = \partial \beta_i$. Furthermore we have the
identity $\omega(\beta_i) = 2\pi \ell_i(u)$ (see Theorem 8.1 \cite{cho-oh}.)
\par

\section{Bulk deformations of Floer cohomology }
\label{sec:bulk}

In this section, we review the results of Section 3.8 \cite{fooo06} (=Section 13 of \cite{fooo06pre}).
\par
Let $(X,\omega)$ be a compact symplectic manifold and $L$ its
Lagrangian submanifold. We take a finite dimensional graded
$R$-vector space $H$ of smooth singular \emph{cycles} of $X$.
(Actually we may consider a subcomplex of the smooth singular chain
complex of $X$ and consider smooth singular {\it chains}. Since
consideration of chain level arguments is not needed in this paper,
we restrict ourselves to the case of {\it cycles}. See \cite{fooo06}
and \cite{fooo10} for relevant explanations.)
\par
We regard an element of $H$ as a cochain (cocycle) by identifying a
$k$-chain with a $(2n-k)$-cochain where $n = \dim L$.
\par
In Section 3.8 \cite{fooo06} (=Section 13 of \cite{fooo06pre}) we introduced a family of operators
denoted by
\begin{equation}\label{qeqation}
\mathfrak q_{\beta;\ell,k}:
E_{\ell} (H[2]) \otimes B_k(H^*(L;R)[1]) \to  H^*(L;R)[1].
\end{equation}
Explanation of the various notations appearing in (\ref{qeqation})
is in order. $\beta$ is an element of the image of $\pi_2(X,L) \to
H_2(X,L;\Z)$. $H[2]$ is the degree shift of $H$ by 2 defined by
$(H[2])^d = H^{d+2}$. $H^*(L;R)[1]$ is the degree shift of the
cohomology group with $R$ coefficient. The notations $E_{\ell}$ and
$B_k$ are defined as follows. Let $C$ be a graded vector space. We
put
$$
B_kC = \underbrace{C\otimes \cdots \otimes C}_{\text{$k$ times}}.
$$
The symmetric group $\mathfrak S_k$ of order $k!$ acts on $B_kC$ by
$$
\sigma \cdot (x_1 \otimes \cdots \otimes x_k)
= (-1)^* x_{\sigma(1)} \otimes \cdots \otimes x_{\sigma(k)},
$$
where
$$
* = \sum_{i<j; \sigma(i)>\sigma(j)} \deg x_i \deg x_j.
$$
$E_kC$ is the set of $\mathfrak S_k$-invariant elements of $B_kC$.
The map (\ref{qeqation}) is a $\Q$-linear map of degree $1
-\mu(\beta)$ here $\mu$ is the Maslov index.
\par
We next describe the main properties of $\mathfrak q_{\beta;\ell,k}$.
Let $B_kC$ be as above and put
$$
BC = \bigoplus_{k=0}^{\infty} B_kC.
$$
(We remark $B_0C = R$.) $BC$ has the structure of coassociative
coalgebra with its coproduct $\Delta: BC \to BC \otimes BC$ defined
by
$$
\Delta(x_1 \otimes \cdots \otimes x_k)
= \sum_{i=0}^k (x_1 \otimes \cdots \otimes x_i)
\otimes (x_{i+1} \otimes \cdots\otimes x_k).
$$
This induces a coproduct $\Delta: EC \to EC \otimes EC$ with
respect to which $EC$ becomes a coassociative and graded
cocommutative.
\par
We also consider a map $\Delta^{n-1}: BC \to (BC)^{\otimes n}$ or $EC \to
(EC)^{\otimes n}$ defined by
$$
\Delta^{n-1} = (\Delta \otimes  \underbrace{id \otimes \cdots
\otimes id}_{n-2}) \circ (\Delta \otimes  \underbrace{id \otimes
\cdots \otimes id}_{n-3}) \circ \cdots \circ \Delta.
$$
For an element $\text{\bf x} \in BC$, it can be
expressed as
$$
\Delta^{n-1}(\text{\bf x}) = \sum_c \text{\bf x}^{n;1}_c \otimes
\cdots \otimes \text{\bf x}^{n;n}_c
$$
where $c$ runs over some index set depending on $\text{\bf x}$. For an element
$$
\text{\bf x} = x_1 \otimes \cdots \otimes x_k \in B_k(H(L;R)[1])
$$
we put the shifted degree $\deg'x_i = \deg x_i - 1$ and
$$
\deg' \text{\bf x} = \sum \deg'x_i = \deg \text{\bf x} - k.
$$
(Recall $\deg x_i$ is the cohomological degree of $x_i$ before
shifted.)
\begin{thm}\label{qproperties}
{\rm (Theorem 3.8.32 \cite{fooo06} = Theorem 13.32 \cite{fooo06pre})} The operators $\mathfrak
q_{\beta;\ell,k}$ have the following properties.
\begin{enumerate}
\item
For each $\beta$ and $\text{\bf x} \in B_k(H(L;R)[1])$,
$\text{\bf y} \in E_{\ell}(H[2])$, we have the following:
\begin{equation}\label{qmaineq}
0 =
\sum_{\beta_1+\beta_2=\beta}\sum_{c_1,c_2}
(-1)^*
\mathfrak q_{\beta_1}(\text{\bf y}^{2;1}_{c_1};
\text{\bf x}^{3;1}_{c_2} \otimes
\mathfrak q_{\beta_2}(\text{\bf y}^{2;2}_{c_1};\text{\bf x}^{3;2}_{c_2})
\otimes \text{\bf x}^{3;3}_{c_2})
\end{equation}
where
$$
* = \deg'\text{\bf x}^{3;1}_{c_2} +
\deg'\text{\bf x}^{3;1}_{c_2} \deg \text{\bf y}^{2;2}_{c_1}
+\deg \text{\bf y}^{2;1}_{c_1}.
$$
In $(\ref{qmaineq})$ and hereafter, we write $\mathfrak q_{\beta}(\text{\bf y};\text{\bf x})$ in place of
$\mathfrak q_{\beta;\ell,k}(\text{\bf y};\text{\bf x})$ if
$\text{\bf y} \in E_{\ell}(H[2])$, $\text{\bf x} \in B_{k}(H(L;R)[1])$.
\item If $1 \in E_0(H[2])$ and $\text{\bf x} \in B_k(H(L;R)[1])$, then
\begin{equation}\label{qism}
\mathfrak q_{\beta;0,k}(1;\text{\bf x}) = \mathfrak m_{\beta;k}(\text{\bf x}).
\end{equation}
Here $\mathfrak m_{\beta;k}$ is the filtered $A_{\infty}$ structure on
$H(L;R)$.
\item Let $\text{\bf e} = PD([L])$ be the Poincar\'e dual to the fundamental
class of $L$. Let $\text{\bf x}_i \in B(H(L;R)[1])$ and we put
$\text{\bf x} = \text{\bf x}_1 \otimes \text{\bf e} \otimes \text{\bf x}_2
\in B(H(L;R)[1])$. Then
\begin{equation}\label{unital}
\mathfrak q_{\beta}(\text{\bf y};\text{\bf x}) = 0
\end{equation}
except the following case.
\begin{equation}\label{unital2}
\mathfrak q_{\beta_0}(1;\text{\bf e} \otimes x) =
(-1)^{\deg x}\mathfrak q_{\beta_0}(1;x \otimes \text{\bf e}) = x,
\end{equation}
where $\beta_0 = 0 \in H_2(X,L;\Z)$ and $x \in H(L;R)[1]
= B_1(H(L;R)[1])$.
\end{enumerate}
\end{thm}
Theorem \ref{qproperties} is proved in  Sections 3.8 and 7.4 of
\cite{fooo06} (= Sections 13 and 32
of \cite{fooo06pre}). We will recall its proof in Section \ref{sec:CalPot}
in the case when $X$ is a toric manifold, $R=\R$ and $L$ is a
Lagrangian fiber of $X$.
\par
We next explain how we use the map $\mathfrak q$ to deform filtered
$A_{\infty}$ structure $\mathfrak m$ on $L$. In this section we use the
universal Novikov ring
$$
\Lambda_{0,{\text{\rm nov}}} = \left.\left\{ \sum c_i T^{\lambda_i}e^{n_i} \right\vert
c_i \in R, \lambda_i \ge 0, n_i \in \Z, \lim_{i\to\infty}\lambda_i = +\infty \right\}
$$
which was introduced in \cite{fooo00}.
We write $\Lambda_{0,{\text{\rm nov}}}(R)$ in case we need to specify $R$.
The ideal $\Lambda_{0,{\text{\rm nov}}}^+$ of $\Lambda_{0,{\text{\rm nov}}}$
is the set of all elements $\sum c_i T^{\lambda_i}e^{n_i}$ of
$\Lambda_{0,{\text{\rm nov}}}$ such that $\lambda_i > 0$. We put
$F^{\lambda}\Lambda_{0,{\text{\rm nov}}} = T^{\lambda}\Lambda_{0,{\text{\rm nov}}}$. It
defines a filtration on $\Lambda_{0,{\text{\rm nov}}}$, under which
$\Lambda_{0,{\text{\rm nov}}}$ is complete. $\Lambda_{0,{\text{\rm nov}}}$ becomes a graded
ring by putting $\deg e = 2$, $\deg T = 0$.
\par
We choose a basis $\text{\bf f}_a$ ($a=1,\ldots,B$) of $H$ and
consider an element
$$
\mathfrak b = \sum_{a} \mathfrak b_a \text{\bf f}_a \in H(X;\Lambda_{0,{\text{\rm nov}}}^+)
$$
such that $\deg \mathfrak b_a + \deg \text{\bf f}_a = 2$ for each $a$.
We then define
\begin{equation}\label{mkdefeq}
\mathfrak m_{k}^{\mathfrak b}(x_1,\ldots,x_k) = \sum_{\beta,\ell}
e^{\mu(\beta)/2}T^{\omega\cap \beta/2\pi} \mathfrak
q_{\beta;\ell,k}(\mathfrak b^{\otimes\ell};x_1,\ldots,x_k).
\end{equation}
Here $\mu: \pi_2(X,L) \to  \Z$ is the Maslov index.
\begin{lem}\label{bulkdef} The family
$\{\mathfrak m^{\mathfrak b}_{k}\}_{k=0}^{\infty}$ defines a filtered $A_{\infty}$
structure on $H(L;\Lambda_{0,{\text{\rm nov}}})$.
\end{lem}
\begin{proof}
We put
$$
e^{\mathfrak b} = \sum_{\ell=0}^{\infty}  \mathfrak b^{\otimes\ell}.
$$
Then we have
$$
\Delta(e^{\mathfrak b}) = e^{\mathfrak b} \otimes e^{\mathfrak b}.
$$
Lemma \ref{bulkdef} follows from this fact and Theorem
\ref{qproperties}. (See  Lemma 3.8.39 \cite{fooo06} = Lemma 13.39 \cite{fooo06pre} for detail.)
\end{proof}
Let $b \in H^1(L;\Lambda_{0,{\text{\rm nov}}}^+)$. We say
$b$ is a {\it weak bounding cochain} of the filtered $A_{\infty}$
algebra $(H(L;\Lambda_{0,{\text{\rm nov}}}),\{\mathfrak m_{k}^{\mathfrak b}\})$ if it satisfies
$$
\sum_{k=0}^{\infty} \mathfrak m_{k}^{\mathfrak b}(b,\ldots,b) = c PD([L])
$$
where $PD([L]) \in H^0(L;\Q)$ is the Poincar\'e dual to the
fundamental cycle and $c \in \Lambda_{0,{\text{\rm nov}}}^+$. By a degree counting, we find that $\deg c = 2$.
\par
We denote by $\widehat{\mathcal M}_{\text{\rm weak,def}}(L;\Lambda_{0,{\text{\rm nov}}}^+)$ the
set of the pairs $(\mathfrak b,b)$ of elements $\mathfrak b \in H\otimes \Lambda_{0,{\text{\rm nov}}}^+$
of degree $2$ and weak bounding cochain $b$ of $(H(L;\Lambda_{0,{\text{\rm nov}}}),\{\mathfrak m_{k}^{\mathfrak b}\})$.
\par
We define $\mathfrak{PO}(\mathfrak b,b)$ by the equation
$$
\mathfrak{PO}(\mathfrak b,b)e = c.
$$
By definition $\mathfrak{PO}(\mathfrak b,b)$ is an element of
$\Lambda_{0,{\text{\rm nov}}}^+$ of degree $0$ i.e.,
$$
\mathfrak{PO}(\mathfrak b,b) \in \Lambda_+
$$
where we recall \eqref{lambda+} for the definition of $\Lambda_+$.

We call the map
$$
\mathfrak{PO}: \widehat{\mathcal M}_{\text{\rm weak,def}}(L;\Lambda_{0,{\text{\rm nov}}}^+)
\to \Lambda_+
$$
{\it the potential function}. We also define the projection
$$
\pi: \widehat{\mathcal M}_{\text{\rm weak,def}}(L;\Lambda_{0,{\text{\rm nov}}}^+)
\to H\otimes \Lambda_{0,{\text{\rm nov}}}^+
$$
by
$$
\pi(\mathfrak b,b) = \mathfrak b.
$$
Let $\text{\bf b}_1 = (\mathfrak b,b_1), 
\text{\bf b}_0 = (\mathfrak b,b_0) \in 
\widehat{\mathcal M}_{\text{\rm weak,def}}(L;\Lambda_{0,{\text{\rm nov}}}^+)$
with
$$
\pi(\text{\bf b}_1) = \mathfrak b = \pi(\text{\bf b}_0).
$$
We define an operator
$$
\delta^{\text{\bf b}_1,\text{\bf b}_0}: H(L;\Lambda_{0,{\text{\rm nov}}}) \to H(L;\Lambda_{0,{\text{\rm nov}}})
$$
of degree $+1$ by
$$
\delta^{\text{\bf b}_1,\text{\bf b}_0}(x)
= \sum_{k_1,k_0} \mathfrak m^{\mathfrak b}_{k_1+k_0+1}(b_1^{\otimes k_1} \otimes x \otimes b_0^{\otimes k_0}).
$$
\begin{lem}\label{delta2}
$$
(\delta^{\text{\bf b}_1,\text{\bf b}_0} \circ \delta^{\text{\bf b}_1,\text{\bf b}_0})(x)
= (-\mathfrak{PO}(\text{\bf b}_1) + \mathfrak{PO}(\text{\bf b}_0))e x.
$$
\end{lem}
\begin{proof}
This is an easy consequence of Theorem \ref{qproperties}. See \cite{fooo06} Proposition 3.7.17 (=\cite{fooo06pre}
Proposition 12.17).
\end{proof}
\begin{defn}\label{FLoercohbulk}(\cite{fooo06} Definition 3.8.61 = \cite{fooo06pre} Definition 13.61.)
For a pair of elements $\text{\bf b}_1, \text{\bf b}_0 \in \widehat{\mathcal M}_{\text{\rm weak,def}}(L;\Lambda_{0,{\text{\rm nov}}}^+)$
with $\pi(\text{\bf b}_1) = \pi(\text{\bf b}_0)$,
$\mathfrak{PO}(\text{\bf b}_1) = \mathfrak{PO}(\text{\bf b}_0)$, we define
$$
HF((L,\text{\bf b}_1),(L,\text{\bf b}_0);\Lambda_{0,{\text{\rm nov}}})
= \frac{\text{\rm Ker}(\delta^{\text{\bf b}_1,\text{\bf b}_0})}
{\text{\rm Im}(\delta^{\text{\bf b}_1,\text{\bf b}_0})}.
$$
This is well defined by Lemma \ref{delta2}.
\end{defn}
By \cite{fooo06} Theorem 6.1.20 (= \cite{fooo06pre} Theorem 24.24), Floer cohomology is of the form
$$
HF((L,\text{\bf b}_1),(L,\text{\bf b}_0);\Lambda_{0,{\text{\rm nov}}})
\cong \Lambda_{0,{\text{\rm nov}}}^a \oplus \bigoplus_{i=1}^k \frac{\Lambda_{0,{\text{\rm nov}}}}{T^{\lambda_i}\Lambda_{0,{\text{\rm nov}}}}.
$$
We call $a$ the {\it Betti number} and
$\lambda_1, \ldots, \lambda_k$ the {\it torsion exponents} of the Floer cohomology.
\par
The following is a consequence of Theorems G \cite{fooo06} (=Theorems G \cite{fooo06pre})
combined. (See also Section \ref{sec:HFbulk}.)
\begin{thm}\label{HFdisplace}
Let $\text{\bf b}_1, \text{\bf b}_0 \in \widehat{\mathcal M}_{\text{\rm weak,def}}(L;\Lambda_{0,{\text{\rm nov}}}^+)$
be as in Definition $\ref{FLoercohbulk}$.
Let $\psi: X \to X$ be a Hamiltonian diffeomorphism.
We assume $\psi(L)$ is transversal to $L$.
Then 
the order of $\psi(L) \cap L$ is not smaller than the Betti number $a$ of
the Floer cohomology
$HF((L,\text{\bf b}_1),(L,\text{\bf b}_0);\Lambda_{0,{\text{\rm nov}}})$.
\end{thm}

\section{Potential function with bulk}
\label{sec:Potbul}
\par
In this section, we specialize the story of the last section to the
case of toric fibers, and make the construction of Section 3.8 \cite{fooo06} (=Section 13
\cite{fooo06pre}) explicit in this case. We also generalize the results
from Section 13 \cite{fooo08} and the story between Floer cohomology
and the potential function to the case with bulk deformations.
\par
Let $X$ be a compact toric manifold and $P$ its moment polytope.
Let $\pi: X \to P$ be the moment map. For each face (of arbitrary
codimension) $P_a$ of $P$ we have a complex submanifold $D_a =
\pi^{-1}(P_a)$ for $a=1,\ldots,B$. We enumerate $P_a$ so that the
first $m$ faces $P_a$ correspond to the $m$ codimension one faces of
$P$. Here we note that the complex codimension of $D_a$ is equal to
the real codimension of $P_a$. Let $\mathcal A=\mathcal A(\Z)$ be the free abelian group
generated by $D_a$. (In this paper we do not consider the case when
$P_a = P$.) It is a graded abelian group $\mathcal A = \oplus_{\ell}
\mathcal A_{\ell}$ with its grading given by the (real) dimension of $D_a$.
We put $D = \pi^{-1}(\partial P) = \cup_{a=1}^m D_a$, that is, the toric
divisor of $X$. We denote
$$
\mathcal A^k = \mathcal A^k(\Z):= \mathcal A_{2n-k}.
$$
We remark that $\mathcal A_{\ell}$ is nonzero only for even $\ell$ and so
$\mathcal A^k$ is nonzero for even $k$. The homomorphism
$: \mathcal A_{2n-k} \to H_{2n-k}(X;\Z)$ and the Poincar\'e duality
induce a surjective homomorphism
$$
i_!: \mathcal A^k(\Z) \to H^k(X;\Z)
$$
for $k \ne 0$. We remark that $i_!$ is not injective. For example
$\mathcal A^{2}(\Z) \cong \Z^m$ (where $m$ is the number of facets of $P$)
and $H^{2}(X;\Z) = \Z^{m-n}$. In fact, we have the exact sequence
$$
0 \to H_2(X;\Z) \to H_2(X,X\setminus D;\Z) \to H_1(X\setminus D;\Z) \cong \Z^n \to
0.
$$
On the other hand, since
$
H_2(N(D);\partial N(D)) \cong H^{2n-2}(N(D)) \cong H^{2n-2}(D)
$, where $N(D)$ is a regular neighborhood of $D$ in $X$, 
we have
$$
H_2(X,X\setminus D;\Q) \cong \Q^{m} \cong \mathcal A^2(\Q)^*.
$$
We also have the natural isomorphism
$$
(i_u)^*: H^1(X \setminus D;\Z) \to H^1(L(u);\Z)
$$
induced by the inclusion map $i_u: L(u) \hookrightarrow X \setminus D$.
Recall that $X \setminus D \cong (\C^*)^n$ and $L(u)$ is a deformation
retract of $X \setminus D$.

We put $\mathcal A^k(\Lambda_+) = \mathcal A^k
\otimes_{\Z} \Lambda_+$, and 
and $\mathcal A^k(\Lambda_0) = \mathcal A^k
\otimes_{\Z} \Lambda_0$. 
The following is a generalization of
Proposition 4.3 \cite{fooo08}. 
Here $\mathcal A(\Lambda_+)$ plays the role of 
$H$ (to be precise, $H \otimes \Lambda_+$) in 
Section \ref{sec:bulk}.
\begin{prop}\label{unobstruct} We have the canonical inclusion
$$
\mathcal A(\Lambda_+) \times  H^1(L(u);\Lambda_{+}) \hookrightarrow
\widehat{\CM}_{\text{\rm weak,def}}(L(u)).
$$
\end{prop}
Proposition \ref{unobstruct} will be proved in Section
\ref{sec:CalPot}.
We remark that the map $i_!: \mathcal A^k(\Z) \to H^k(X;\Z)$
is not injective. Therefore, the gauge equivalence relation
(See Definition 4.3.1 \cite{fooo06} = Definition 16.1 \cite{fooo06pre}.) on the left hand side is
nontrivial. Because of this the target of the inclusion cannot descend to ${\CM}_{\text{\rm weak,def}}(L(u))$,
the set of gauge equivalence classes of the elements of $\widehat{\CM}_{\text{\rm weak,def}}(L(u))$.
\par
For $\mathfrak b \in \bigoplus_k \mathcal A^k(\Lambda_+)$, 
$u \in \text{\rm Int}\, P$, we define
$$
\mathfrak {PO}^u_{\mathfrak b} = \mathfrak {PO}^u(\mathfrak b, \cdot): H^1(L(u);\Lambda_+) \to \Lambda_+
$$
by
\begin{equation}\label{PObuldef}
\mathfrak{PO}^u(\mathfrak b,b) = \sum_{\beta;\ell,k} T^{\omega\cap
\beta/2\pi}\mathfrak q_{\beta;\ell,k}(\mathfrak b^{\otimes\ell};
b^{\otimes k}) \cap [L(u)]
\end{equation}
for $b \in H^1(L(u);\Lambda_+)$. 
We remark that the summation on right hand side includes the term
where $\ell =0$. The term corresponding thereto is
$$
\sum_{k,\beta} T^{\omega\cap \beta/2\pi}\mathfrak
m_{\beta;k}(b^{\otimes k}) \cap [L(u)] = \mathfrak{PO}^u(b)
$$
which is nothing but the potential function in the sense of
Section 4 \cite{fooo08}. Namely we have the identity\begin{equation}
\mathfrak{PO}^u(0,b) = \mathfrak{PO}^u(b).
\end{equation}
This function (\ref{PObuldef}) is also a special case of the
potential function we discussed in Section \ref{sec:bulk}. (We will
not use the variable $e$ in this section.) 
We next discuss a generalization of Theorem 4.6 \cite{fooo08}.  

\subsection{Potential function $\mathfrak{PO}^u$}
\label{subsec:W}
Theorem \ref{weakpotential} in this subsection is a generalization of Theorem 4.6 \cite{fooo08}.  
To state it, 
we need some preparations.

\begin{defn}
A discrete submonoid of $\R_{\ge 0}$ is a subset $G \subset \R_{\ge
0}$ such that
\begin{enumerate}
\item $G$ is discrete.
\item If $g_1,g_2 \in G$, then $g_1 + g_2 \in G$. $0 \in G$.
\end{enumerate}
Hereafter we say {\it discrete submonoid} in place of discrete
submonoid of $\R_{\ge 0}$ for simplicity.
\par
For any discrete subset $\mathfrak X$ of $\R_{\ge 0}$ there exists a
discrete submonoid containing it. The discrete submonoid $G$
generated by $\mathfrak X$ is, by definition, the smallest one among
them. We write $G = \langle \mathfrak X\rangle$.
\end{defn}
Compare Condition 3.1.6 \cite{fooo06} (= Condition 6.11
\cite{fooo06pre}). In  \cite{fooo06} we considered $G \subset
\R_{\ge 0} \times 2\Z$. Since we do not use the grading parameter
$e$, we consider $G \subset \R_{\ge 0}$ in this paper.
\par
\begin{defn}
Let $C_i$ be an $R$ vector space. We denote by $C_i(\Lambda_0)$ the
completion of $C_i \otimes \Lambda_0$. Let $G$ be a discrete
submonoid.
\par
\begin{enumerate}
\item An element $x$ of $C_i(\Lambda_0)$ is said to be \emph{$G$-gapped} if
$$
x = \sum_{g\in G} x_{g} T^{g}
$$
where $x_{g} \in C_i$.
\item
A filtered $\Lambda_0$ module homomorphism $f: C_1(\Lambda_0) \to
C_2(\Lambda_0)$ is said to be {\it $G$-gapped} if there exist $R$ linear
maps $f_g: C_1 \to C_2$ for $g\in G$ such that
$$
f(x) = \sum_{{g}\in G} T^{g} f_{g}(x).
$$
Here we extend this map to one $f_{g}: C_1(\Lambda_0) \to
C_2(\Lambda_0)$ in an obvious way.
\end{enumerate}
\end{defn}
The $G$-gappedness of potential functions, of filtered $A_{\infty}$
structures, and etc. can be defined in a similar way.
\par
We define
\begin{equation}\label{gapX}
G(X) = \langle \{\omega \cap \beta / 2\pi \mid  \text{$\beta\in
\pi_2(X)$ is realized by a holomorphic sphere}\} \rangle.
\end{equation}
Denote by $G_{\text{\rm bulk}}$ the discrete submonoid that is
defined in Definition \ref{Gbulk}. It contains $G(X)$ as a subset.
\par
\medskip

By Proposition \ref{unobstruct}, the domain of the 
potential function $\mathfrak{PO}^u$ is 
$\mathcal A(\Lambda_+) \times H^1(L(u);\Lambda_+)$. 
We will extend it to 
$\mathcal A(\Lambda_+) \times H^1(L(u);\Lambda_0)$
as follows: 
For any $z \in \Lambda_+$ we have a function 
$
\exp : \Lambda_+ \to 1+ \Lambda_+
$ 
defined by $\exp z =e^z= \sum _{k=0}^{\infty}\frac {z^k}{k!}$. 
We remark that $1 + \Lambda_+$ is the set of elements $1 + x \in
\Lambda_0$ with $x \in \Lambda_+$. 
It coincides with the image of $\Lambda_+$ by  
the function $\exp$. 
\begin{lem}\label{lem:exp} 
The function 
$\exp : \Lambda_+ \to 1+ \Lambda_+$ 
extends to 
a function 
$$\exp ~:~ \Lambda_0 \to \C^{\ast}(1+ \Lambda_+)
\cong \Lambda_0 \setminus \Lambda_+ . 
$$
We denote 
the extended function by the same symbol: $\exp z=e^z$ for $z \in \Lambda_0$.
\end{lem}
\begin{proof} Take any $z \in \Lambda_0$. 
Note that
$\Lambda_0 = \C \oplus \Lambda_+$ and so we can decompose
$z = \bar z + z_+ \in \C \oplus \Lambda_+$ and define $\exp(z)$ by
\be\label{eq:exp}
\exp(z) : = e^{\bar z} e^{z_+}.
\ee
We note that $e^{\bar z} \in \C^*=\C \setminus 0$ and $e^{z_+} \in 1 + \Lambda_+$
which is well-defined.  
\end{proof}
Now we choose an integral basis 
$\{\text{\bf e}_i\}_{i=1,\dots , n}$ of 
$H^1(L(u);\Z)$. 
We denote by $x_i$ ($i=1,\ldots, n$) the coordinates of
$H^1(L(u);\Lambda_+)$ with respect to a basis 
$\{\text{\bf e}_i\}$.
We put
$$
\mathfrak b = \sum_{a=1}^B w_a [D_a] \in \mathcal{A}(\Lambda_+), \quad 
b = \sum_{i=1}^n x_i {\bf e}_i \in H^1(L(u);\Lambda_+).
$$
Here $B = \sum_k \text{rank} \, \mathcal A^k$ and  
we regard $x_i \in \Lambda_+$. 
However, by Lemma \ref{lem:exp} 
we can consider 
$$y_i=e^{x_i} \in \Lambda_0 \setminus \Lambda_+
$$ for 
$x_i \in \Lambda_0$. 
We remark that the coordinate variables $x_i, y_i$ 
introduced here depend on 
$u \in \text{Int}\,P$. 
(If we use the notation in the next subsection,  $y_i$ here is nothing but $y_i(u)$ in Subsection \ref{subsec:powerseriesring}. 
See Remark \ref{rem:yi} and 
Remark \ref{rem:notation}.) 
\par

We start with the {\it leading order potential function} defined by
\begin{equation}\label{PO0}
\mathfrak{PO}^u_0(b) = \sum_{i=1}^m T^{\ell_i(u)}y_1^{v_{i,1}}\cdots
y_n^{v_{i,n}}.
\end{equation}
Here $v_{i,j}$ is given by (\ref{eq:v}). 
Note that the leading order potential function 
$\mathfrak{PO}^u_0(b)$ is independent of $\mathfrak b \in \mathcal{A}(\Lambda_+)$ and 
a finite sum which can be read-off purely in terms of the
moment polytope of the toric manifold $X$.
\par
\begin{thm}\label{weakpotential}
Let $X$ be an arbitrary compact toric manifold and $L(u)$ as
above and $\mathfrak b \in \mathcal A(\Lambda_+)$ a $G_{\text{\rm bulk}}$-gapped element.
Then there exist $c_{\sigma} \in \Q$, $e_{\sigma}^i \in \Z_{\ge 0}$,
$\rho_{\sigma} \in G_{\text{\rm bulk}}$ and $\rho_{\sigma} > 0$, such that
$\sum_{i=1}^m e_{\sigma}^i > 0$ and
\begin{equation}\label{eq:weakPO}
\mathfrak{PO}^u(\mathfrak b;b) - \mathfrak{PO}_0^u(b)
=
\sum_{{\sigma}=1}^{\infty} c_{\sigma}
y_1^{v'_{{\sigma},1}}\cdots y_n^{v'_{{\sigma},n}}T^{\ell'_{\sigma}(u)+\rho_{\sigma}}, 
\end{equation}
where
\begin{equation}\label{edefform}
v'_{{\sigma},k} = \sum_{i=1}^m e_{\sigma}^iv_{i,k}, \quad \ell'_{\sigma} = \sum_{i=1}^m e_{\sigma}^i \ell_i.
\end{equation}
If there are infinitely many non-zero $c_{\sigma}$'s, we have
\begin{equation}\label{rhogoinfty}
\lim_{{\sigma}\to\infty}
\rho_{{\sigma}} = \infty.
\end{equation}
\end{thm}
Theorem \ref{weakpotential} is proved in Section \ref{sec:CalPot}.
The condition (\ref{rhogoinfty}) slightly improves the
corresponding statement in Theorem 4.6 \cite{fooo08}.
\par
Using the isomorphism $\CA(\Lambda_+) \cong (\Lambda_+)^B$,
we regard $\mathfrak{PO}^u$ as a function of $w_i$ and $y_i$ and
define a function
\be\label{defPObulk}
\mathfrak{PO}^u : (\Lambda_+)^B \times (\Lambda_0\setminus \Lambda_+)^n \to \Lambda_+
\ee
by
$$
\mathfrak{PO}^u(w_1,\ldots, w_B; y_1,\ldots,y_n) := \mathfrak{PO}^u(\mathfrak b;b)
$$
where $\mathfrak b =\sum_{a=1}^B w_a [D_a]$ and 
$b = \sum_{i=1}^n x_i {\bf e}_i$. 
(More formally speaking, we define
$$
\mathfrak{PO}^u(w_1,\ldots,w_B;y_1,\ldots, y_n) = \mathfrak{PO}^u(\mathfrak b; b(\text{Log}(y)))
$$
where $b(\text{Log}(y))$ is the weak bounding cochain corresponding to the coordinates
$x = (x_1,\ldots, x_n)$ with $y_i = e^{x_i}$. Namely, $\text{Log}$ is the inverse of
the exponential $\text{Exp}: (\Lambda_0)^n \to 
(\Lambda_0 \setminus \Lambda_+)^n$,
$\text{Exp}(x) = (e^{x_1},\ldots,e^{x_n})$.)

\begin{lem}\label{lem:extend} 
The potential function $\mathfrak {PO}^u$ 
in (\ref{PObuldef}) 
$$
\mathfrak {PO}^u ~:~ \mathcal A(\Lambda_+) \times 
H^1(L(u);\Lambda_+) \to \Lambda_+
$$
extends to a function
$$
\mathcal A(\Lambda_+) \times H^1(L(u);\Lambda_0)
\cong (\Lambda_+)^B \times (\Lambda_0)^n 
\to \Lambda_+
$$
with variable $w_1, \dots ,w_B$ and $x_1, \dots ,x_n$. 
If we regard it as a function of 
$w_1, \dots ,w_B$ and $y_1, \dots ,y_n$, 
its domain is 
$
\mathcal (\Lambda_+)^B \times (\Lambda_0 \setminus \Lambda_+)^n.
$
\end{lem}
\begin{proof}
By assumption $\mathfrak v_T(y_i) =0$. Therefore
$$
\mathfrak v_T
\left(
c_{\sigma}
y_1^{v'_{{\sigma},1}}\cdots y_n^{v'_{{\sigma},n}}T^{\ell'_{\sigma}(u)+\rho_{\sigma}}
\right)
= \ell_{\sigma}'(u) + \rho_{\sigma},
$$
which goes to infinity as $\sigma$ goes to infinity 
by  (\ref{rhogoinfty}).
\end{proof}
We denote the extension by the same symbol 
$\mathfrak {PO}^u$. 
\par
In order to
describe the function space which the potential function belongs to, 
we need some digression on
the notion of \emph{strictly convergent power series ring}. See
\cite{BGR} for detailed discussion on strictly
convergent power series ring over semi-normed ring. In our case, the
semi-normed ring is $\Lambda_0$ with its norm induced by the
valuation $\mathfrak v_T$.

\subsection{Strictly convergent power series}
\label{subsec:powerseriesring}

Here we will explain in what sense our potential function
$\mathfrak{PO}^u$ can be regarded as a convergent power series. 
The main result of this subsection is 
Theorem \ref{extendth}. 
\par
Let $w_1,\ldots,w_{m}$ be the parameters corresponding to $\mathcal A^2$. (Here $m$
is the number of facets of $P$.) We put $\mathfrak w_i =
e^{w_i}$ and consider the polynomial ring
\begin{equation}\label{stconvserring}
\Lambda_0[\mathfrak w_1,\ldots,\mathfrak w_m,
\mathfrak w_1^{-1},\ldots,\mathfrak w_m^{-1},w_{m+1},\ldots,w_{B},y_1,y_1^{-1},
\ldots,y_n,y_n^{-1}].
\end{equation}

The following definition is taken from Definition 1 in Section 1.4.1 \cite{BGR}.
\begin{defn}
Let $e_{k,i} \in \Z$ ($i\le m$), $e_{k,i} \in \Z_{\ge 0}$
($i> m$), $f_{k,i} \in \Z$, $a_k \in \Lambda_{0}$. An infinite power series
$$
\sum_k a_k \, \mathfrak w_1^{e_{k,1}}\cdots\mathfrak w_m^{e_{k,m}}w_{m+1}^{e_{k,m+1}}
\cdots w_{B}^{e_{k,B}} y_1^{f_{k,1}}
\cdots y_n^{f_{k,n}}
$$
is called \emph{strictly convergent} if
$
\lim_{k\to\infty} \mathfrak v_T(a_k) = \infty.
$
\end{defn}
We denote the set of strictly convergent  power series by
$$
\Lambda_0 \langle\!\langle \mathfrak w,
\mathfrak w^{-1},w,y,y^{-1}\rangle\!\rangle.
$$
By definition it provides a completion of $\Lambda_0[\mathfrak w,\mathfrak w^{-1},w,y,y^{-1}]$
with respect to the norm $e^{-\mathfrak v_T}$ of $\Lambda_0$.

We will show that the function $\mathfrak{PO}^u$ is a strictly convergent power series
contained in
\begin{equation}\label{strongconvincluded}
\Lambda_0\langle\!\langle\mathfrak w,
\mathfrak w^{-1},w,y,y^{-1}\rangle\!\rangle.
\end{equation}
\begin{rem}\label{rem:yi} Note that 
the variables $y_i$ used in this subsection 
(and only in this subsection and 
Remark \ref{anounce}) 
are 
different from $y_i$ used in the previous subsection 
(e.g. in the formula (\ref{eq:weakPO})) and other  
sections. 
The variables $y_i(u)$ in (\ref{yiu}) are exactly 
same as the variables $y_i$ in other sections. 
See also Remark \ref{rem:notation}.  
\end{rem}

In fact, we will prove a stronger statement in Theorem \ref{extendth}.
To make the precise statement on this we need some digression on the
coordinate changes associated to the moment polytope $P$.
\par
First recall that $P$ is convex and so $\text{Int}\,P$ is contractible.
Therefore we have a $T^n$-bundle isomorphism
$$
\Psi: \pi^{-1}(\text{Int}\,P) \cong T^n \times \text{Int}\,P.
$$
For example, we can construct such an isomorphism by first picking a
reference point $u_{\text{\rm ref}}$ and identifying a fiber
$\pi^{-1}(u_{\text{\rm ref}}) = L(u_{\text{\rm ref}})$ with $T^n$
and then using the parallel transport with respect to the natural
affine connection associated the Lagrangian smooth fibration
$\pi^{-1}(\text{Int}\,P) \to \text{Int}\,P$. (See \cite{weinst71},
\cite{duister80}.) Then $\Psi$ induces a natural isomorphism
$$
\psi_u: = (\Psi|_{\pi^{-1}(u)})^*:  H^1(T^n;\Z) \to H^1(L(u);\Z).
$$
Now we choose a basis $\{{\bf e}_i\}$ of $H^1(T^n;\Z)$ and $x_i$ for
$i = 1,\ldots, n$ the associated coordinates. We then denote $y_i =
e^{x_i}$. We note that $\{{\bf e}_i\}$ and $x_i$ (and so $y_i$)
depend only on $T^n$. Using the isomorphism $\psi_u$ we can
push-forward them to $H^1(L(u);\Z)$ which are nothing but the
coordinates associated to the basis
$$
\{\psi_u({\bf e}_i) \}_{1 \leq i \leq n}
$$
of $H^1(L(u);\Z)$ mentioned in the end of Section \ref{sec:introduction}.

We denote the variable
\be\label{yiu}
y_i(u) = T^{-u_i} y_i \circ \psi_u^{-1}, \quad i=1,\dots ,n
\ee
and consider the ring
$$
\Lambda[\mathfrak w_1,\ldots, \mathfrak w_m^{-1},w_{m+1},\ldots,w_{B},y_1(u),\ldots,y_n(u)^{-1}].
$$
By definition we have a ring isomorphism, again denoted by $\psi_u$,
$$
\psi_u: \Lambda[\mathfrak w,\mathfrak w^{-1}, w,y,y^{-1}] \to
\Lambda[\mathfrak w,\mathfrak w^{-1},w,y(u),y(u)^{-1}];
\quad y_i \mapsto T^{u_i}y_i(u).
$$
Furthermore by definition, we have a ring isomorphism
$$
\aligned
\psi_{u',u} &: \Lambda[\mathfrak w_1,\ldots,\mathfrak w_m^{-1},w_{m+1},\ldots,w_{B},y_1(u),\ldots,y_n(u)^{-1}] \\
&\to
\Lambda[\mathfrak w_1,\ldots,\mathfrak w_m^{-1},w_{m+1},\ldots,w_{B},y_1(u'),\ldots,y_n(u')^{-1}]
\endaligned
$$
given by $\psi_{u',u} = \psi_{u'} \circ \psi_u^{-1}$ or more explicitly by
$$
\psi_{u',u}(y_i(u)) =  T^{u'_i-u_i} y_i(u')
$$
for any two $u, \, u' \in \text{Int}\,P$. (Compare this discussion with
the one right after Remark 6.14 \cite{fooo08}.)
Clearly $\psi_{u'',u'}\circ \psi_{u',u} = \psi_{u'',u}$.
\par
Now we define a family of valuations $\mathfrak v_T^u$ parameterized by
$u \in \text{Int}\, P$ on the ring $\Lambda[\mathfrak w,\mathfrak w^{-1},w,y,y^{-1}]$ by
the formula
\begin{equation}
\aligned &\mathfrak v_T^{u}\left(\sum_k a_k \, \mathfrak
w_1^{e_{k,1}}\cdots \mathfrak w_m^{e_{k,m}}w_{m+1}^{e_{k,m+1}}
\cdots w_{B}^{e_{k,B}} y_1^{f_{k,1}}
\cdots y_n^{f_{k,n}}\right) \\
& = \inf_k \{\mathfrak v_T(a_k) + \langle f_k, u\rangle \mid
a_k \ne 0\}.
\endaligned
\end{equation}
This valuation can be regarded as a deformation of the valuation $\mathfrak v_T$
incorporating the `instanton correction' associated to the Lagrangian
fiber $L(u)$. By definition we have
\be\label{vTyi}
\mathfrak v_T^{u}(y_i) = u_i.
\ee
Then by (\ref{yiu}) the variable $y_i(u)$ satisfies 
\be\label{vtyiu}
\mathfrak v_T^u(y_i(u)) = 0.
\ee
\par
\begin{defn}\label{def:vTP}
We define a function
$$
\mathfrak v_T^P(x) = \inf\{ \mathfrak v_T^u(x) \mid u \in \text{\rm Int}\, P\}
$$
on the ring $\Lambda[\mathfrak w,\mathfrak w^{-1},w,y,y^{-1}]$. It does not
define a valuation but $e^{-\mathfrak v_T^P}$ defines a norm.
We denote its completion by
$
\Lambda^P\langle\!\langle\mathfrak w,\mathfrak w^{-1},w,y,y^{-1}\rangle\!\rangle.
$
We put
$$
\Lambda^P_0\langle\!\langle\mathfrak w,\mathfrak w^{-1},w,y,y^{-1} \rangle\!\rangle
= \{ x \in \Lambda^P \langle\!\langle\mathfrak w,\mathfrak w^{-1},w,y,y^{-1}\rangle\!\rangle
\mid \mathfrak v_T^P(x) \ge 0\}.
$$
\end{defn}

Next we define the variable
\be\label{def:zj}
z_j(u) = T^{\ell_j(u)} y_1(u)^{v_{j,1}}\cdots y_n(u)^{v_{j,n}}
\ee
for $j = 1, \ldots, m$.

The following lemma is a consequence of the definition \eqref{yiu} of
$y_i(u)$.

\begin{lem}\label{lem:zjupsiu} The expression
$$
z_j(u) \circ \psi_u \in \Lambda\langle\!\langle y,y^{-1} \rangle\!\rangle
$$
is independent of $u \in \text{\rm Int}\,P$. We denote this common
variable by $z_j$. Furthermore we have
\be\label{vTuzj}
\mathfrak v_T^u(z_j) = \ell_j(u).
\ee
In particular, $z_j$ lies in $\Lambda_0^P\langle\!\langle y,y^{-1}\rangle\!\rangle$
which is defined in a way similar to 
$$
\Lambda^P_0\langle\!\langle \mathfrak w,\mathfrak w^{-1},w,y,y^{-1} \rangle\!\rangle.
$$
\end{lem}
\begin{proof} From \eqref{yiu}, we have
$y_i(u) \circ \psi_u = T^{-u_i} y_i$.
Therefore we have
$$
(y_1(u)^{v_{j,1}}\cdots y_n(u)^{v_{j,n}}) \circ \psi_u = T^{-\langle
v_j,u \rangle} \prod_{i=1}^n y_i^{v_{j,i}}
$$
for $j = 1, \dots, m$. Recalling $\ell_j(u) = \langle v_j, u \rangle - \lambda_j$,
we obtain
$$
z_j(u) \circ \psi_u = T^{\ell_j(u)} (y_1(u)^{v_{j,1}}\cdots
y_n(u)^{v_{j,n}}) \circ \psi_u = T^{-\lambda_j} \prod_{i=1}^n
y_i^{v_{j,i}}
$$
which shows that that $z_j(u)\circ \psi_u$ is independent of
$u$.
\par
The formula (\ref{vTuzj}) immediately follows from 
(\ref{vtyiu}) and (\ref{def:zj}).
 Finally since $\ell_j(u) > 0$ for all $u \in
\text{Int}\,P$ and so $\mathfrak v_T^P(z_j) \geq 0$, $z_j \in \Lambda_0^P
\langle\!\langle y,y^{-1}\rangle\!\rangle$ by definition.
This finishes the proof.
\end{proof}
We remark that
\begin{equation}\label{vTzj}
\mathfrak v_T^P(z_j) = 0
\end{equation}
for $j=1,\dots,m$.
\begin{rem}
\begin{enumerate}
\item We note that the variables $z_j, \, j = 1, \ldots, m$ depend on
$\lambda_j$'s (i.e. on the polytope $P$) and the vectors
$\{v_j\}_{j=1,\ldots, m}$. Recall that the latter is the set of one
dimensional generators of the fan $\Sigma$ associated to the toric
manifold $X = X_\Sigma$ which is related to the \emph{complex
structure} on $X$. On the other hand,  
$\lambda_j$'s are 
related to the \emph{symplectic structure} of $X$. In other
words, the variables depend on both complex structure and symplectic
structure. 
Note that $\{v_{j}\}_{j=1,\dots ,m}$ are also uniquely determined by the structure of $X$ as a Hamiltonian $T^{n}$ space. 
\item These variables $z_j, \, j=1,\ldots,m$ correspond to the standard
coordinates of $H_2(X,X\setminus D;\Lambda_0) \cong H_2(X,L(u);\Lambda_0) \cong
(\Lambda_0)^m$. This is the standard homogeneous coordinates of the
toric varieties which also appear as the natural coordinates in the
linear sigma model \cite{hori-vafa}.
\end{enumerate}
\end{rem}
\par
Now we consider strictly convergent formal power series of the form
\begin{equation}\label{elecomp}
\sum_{k=1}^{\infty}b_k \mathfrak w_1^{e_{k,1}}\cdots \mathfrak
w_m^{e_{k,m}}w_{m+1}^{e_{k,m+1}} \cdots w_{B}^{e_{k,B}}
z_1^{h_{k,1}}\cdots z_m^{h_{k,m}},
\end{equation}
with the conditions
\begin{eqnarray*}
b_k & \in & \Lambda_0, \qquad \lim_{k \to \infty} \mathfrak v_T(b_k) = \infty , \\
e_{k,i} & \in & \begin{cases} \Z \quad & i\le m,\\
\Z_{\ge 0} \quad & i> m,
\end{cases} \\
h_{k,j} & \in & \Z_{\geq 0}.
\end{eqnarray*}
We denote by
\be\label{ring-z}
\Lambda_0\langle\!\langle\mathfrak w,\mathfrak w^{-1},w,z\rangle\!\rangle
\ee
the set of such power series. This is the completion of the polynomial ring
$$\Lambda_0[\mathfrak w,\mathfrak w^{-1}, w, z]$$ with respect to the
norm $\mathfrak v_T$.

The following lemma shows that the ring $
\Lambda_0^P\langle\!\langle\mathfrak w,\mathfrak w^{-1},w,y,y^{-1}\rangle\!\rangle$
we have introduced is obtained by a natural reduction of this canonical
ring $\Lambda_0\langle\!\langle\mathfrak w,\mathfrak w^{-1},w,z\rangle\!\rangle$.

\begin{lem}\label{characterizationlamda0} The relation
\be\label{zjformula}
z_j =  T^{\ell_j(u)}  (y_1(u))^{v_{j,1}} \cdots (y_n(u))^{v_{j,n}} = T^{-\lambda_j} \prod_{i=1}^n
y_i^{v_{j,i}}
\ee
defines a continuous surjective ring homomorphism
$$
\Lambda_0\langle\!\langle\mathfrak w,\mathfrak w^{-1},w,z\rangle\!\rangle
\to
\Lambda_0^P\langle\!\langle\mathfrak w,\mathfrak w^{-1},w,y,y^{-1}\rangle\!\rangle
$$
with respect to the the topology induced by $\mathfrak v_T^P$ and $\mathfrak v_T$,
respectively.
\end{lem}
\begin{proof} 
First of all, using the relation (\ref{zjformula}), we show that 
any element
$$
\sum_{k=1}^{\infty}b_k \mathfrak w_1^{e_{k,1}}\cdots \mathfrak
w_m^{e_{k,m}}w_{m+1}^{e_{k,m+1}} \cdots w_{B}^{e_{k,B}}
z_1^{h_{k,1}}\cdots z_m^{h_{k,m}} \in
\Lambda_0\langle\!\langle\mathfrak w,\mathfrak w^{-1},w,z\rangle\!\rangle
$$
as in (\ref{elecomp}) 
indeed defines an element of 
$\Lambda_0^P\langle\!\langle\mathfrak w,\mathfrak w^{-1},w,y,y^{-1}\rangle\!\rangle$.
\par
Let $\eta_k$ be the $k$-th monomial of the given power series above. 
Then by (\ref{vTzj}) we obtain
$$
\mathfrak v_T^P(\eta_k) = \inf_{u\in \text{Int}P} \mathfrak v_T^u(\eta_k) \geq \mathfrak v_T(b_k).
$$
Therefore we obtain
$$
\lim_{k \to \infty} \mathfrak v_T^P(\eta_k) \geq \lim_{k \to\infty} \mathfrak v_T(b_k).
$$
By the hypothesis $\lim_{k \to \infty} \mathfrak v_T(b_k) = \infty.$
This proves $\lim_{k \to \infty} \mathfrak v_T^P(\eta_k) =\infty$ and hence 
(\ref{elecomp}) defines an element of 
$\Lambda_0^P\langle\!\langle\mathfrak w,\mathfrak w^{-1},w,y,y^{-1}\rangle\!\rangle$. Thus   
we can define the homomorphism in 
Lemma \ref{characterizationlamda0}.
\par
Next we prove the surjectivity. 
Let
\be\label{eq:series-P}
\sum_k a_k \mathfrak w_1^{e_{k,1}}\cdots \mathfrak
w_m^{e_{k,m}}w_{m+1}^{e_{k,m+1}}
\cdots w_{B}^{e_{k,B}} y_1^{f_{k,1}} \cdots y_n^{f_{k,n}}
\in \Lambda_0^P\langle\!\langle\mathfrak w,\mathfrak w^{-1},w,y,y^{-1}\rangle\!\rangle.
\ee

We first show that each monomial $\xi_k$ in this term satisfies
$\mathfrak v_T^u(\xi_k) \geq 0$ for all $u \in \operatorname{Int}P$.

Consider any monomial
$$
\xi = a \, \mathfrak w_1^{e_{1}}\cdots \mathfrak
w_m^{e_{m}}w_{m+1}^{e_{m+1}}
\cdots w_{B}^{e_{B}} y_1^{f_{1}} \cdots y_n^{f_{n}}
$$
appearing in this series. Since the value of $\mathfrak v_T^P$ of this series
is nonnegative by definition of the ring
$\Lambda_0^P\langle\!\langle\mathfrak w,\mathfrak w^{-1},w,y,y^{-1}\rangle\!\rangle$,
$\xi$ satisfies $\mathfrak v_T^u(\xi) \geq 0$ for all $u \in \operatorname{Int}P$.
By definition of the valuation $\mathfrak v_T^u$, we have
$$
\mathfrak v_T^{u}(\xi) = \mathfrak v_T(a) + \langle f, u \rangle.
$$
Denote
\begin{equation}\label{cminimal}
c=\inf \left\{ \langle f, u \rangle \mid u \in \text{\rm Int}\, P\right\}.
\end{equation}
Since $P$ is a convex polytope, we can take a vertex $u^0$ of $P$ such that
$$
\langle f, u^0 \rangle = c.
$$
There exist $n$ faces $\partial_{j_i}P$, $i = 1,\ldots, n$ such that
\begin{equation}\label{u0isintersection}
\{u^0\} = \bigcap_{i =1}^{n} \partial_{j_i}P.
\end{equation}
Since $X$ is a smooth toric manifold, the corresponding fan is regular and so
$\vec v_{j_i}$ $i=1,\ldots,n$ forms a $\Z$-basis of $N$. (See
Section 2.1 \cite{fulton}, for example.) Therefore we have
\be\label{eq:f}
f = (f_1,\ldots,f_n) = \sum_{i=1}^n h_i \vec v_{j_i}
\ee
for some $h_i \in \Z$.
We choose $\vec v_{k}^* \in M = N^*$ such that
$$
\langle \vec v_{j_i}, \vec v_{k}^* \rangle
=
\begin{cases}
1 \quad &\text{if $i=k$}, \\
0  \quad &\text{otherwise.}
\end{cases}
$$
For $u\in M_{\R}=M\otimes \R$ we define $c_k(u)$ by 
$$
u - u^0 = \sum_{k=1}^n c_k(u) \vec v^*_{k}.
$$
If $u \in \operatorname{Int}P$, we have
$c_k(u) > 0$ for all $k$. And at least one $c_k(u) \to 0$ as $u$ approaches
$\bigcup_{i =1}^{n} \partial_{j_i}P$.
By definition of $c$ and $u^0$, we also have
$\langle f,u \rangle \geq \langle f, u^0 \rangle$.
In other words, we have
\begin{equation}\label{fu-u0}
\langle f, u-u^0 \rangle \geq 0
\end{equation}
for all $u \in \operatorname{Int}P$.
It follows that $h_i \ge 0$ for all $i$: Note that we have
$$
\langle f, u-u^0 \rangle \to c_k(u') h_k
$$
as we let $u \in \operatorname{Int}P$ approach to a point $u' \in
\bigcap_{i \neq k }^{n} \partial_{j_i}P\setminus \{u^0\}$.
In fact, 
$
u' \in \partial_{j_i}P$ is equivalent to $c_i(u') = 0$.
Therefore if $h_{k} < 0$ for some $k = 1,\ldots, n$, we would have 
$c_k(u') h_k  < 0$ and so $\langle f,
u-u^0 \rangle < 0$ for $u$ sufficiently close to $u'$, a contradiction to \eqref{fu-u0}.
\par
And we can express
$$
\xi = 
aT^{\langle f, u^0 \rangle}
\mathfrak w_1^{e_{1}}\cdots \mathfrak
w_m^{e_{m}}w_{m+1}^{e_{m+1}}
\cdots w_{B}^{e_{B}}z_{j_1}^{h_1}\cdots z_{j_n}^{h_n}.
$$
Here we use the identity $z_{j_k} = T^{-\lambda_{j_k}} \prod_{i=1}^n y_i^{v_{j_k,i}}$
and \eqref{eq:f}, and the facts that $z_{j_k}$
do not depend on $u$ and $\ell_{j_k}(u^0) = 0$ (by \eqref{u0isintersection}). Therefore we have
$$
\mathfrak v_T(aT^{\langle f, u^0 \rangle}) = \mathfrak v_T(a) + \langle f,
u^0 \rangle = \mathfrak v_T(a) + c.
$$
If $\mathfrak v_T^{u}(\xi) \ge 0$ for all 
$u \in \text{Int}\,P$, then
$$
\mathfrak v_T(aT^{\langle f, u^0  \rangle}) = \mathfrak v_T(a) + c =
\inf_{u \in \text{Int}\,P} \mathfrak v_T^u(\xi) \ge 0.
$$
This proves that $\xi $ is of the form (\ref{elecomp}).

Applying the above arguments to each monomial $\xi_k$ appearing in \eqref{eq:series-P}, we can express
$$
\xi_k = a_k T^{\langle f, u^0_k \rangle}
\mathfrak w_1^{e_{k,1}}\cdots \mathfrak
w_m^{e_{k,m}}w_{m+1}^{e_{k,m+1}}
\cdots w_{B}^{e_{k,B}}z_{j_1}^{h_{k,1}}\cdots z_{j_n}^{h_{k,n}}.
$$
Now it remains to show that
$$
\lim_{k \to \infty} \mathfrak v_T(\xi_k) = \lim_{k \to \infty} \mathfrak v_T(a_k T^{\langle f_k, u^0_k  \rangle}) = \infty.
$$
However we have
$$
\mathfrak v_T\left(a_k T^{\langle f_k, u^0_k  \rangle}\right) =
\inf_{u \in \text{Int}\,P} \mathfrak v_T^u(\xi_k)
$$
which converges to $\infty$ as $k \to \infty$ by the hypothesis \eqref{eq:series-P}.
Therefore we have finished the proof of 
the surjectivity of the homomorphism.
\end{proof}
\begin{rem}\label{remtocharacter}
\begin{enumerate}
\item
We remark that the representation $(\ref{elecomp})$ of an element
$x \in \Lambda_0^P\langle\!\langle\mathfrak w, \mathfrak w^{-1},w,y,y^{-1} \rangle\!\rangle$
is {\it not} unique. 
This non-uniqueness is related to the fact that $z_i$'s in
$\Lambda_0^P\langle\!\langle\mathfrak w,\mathfrak w^{-1},w,y,y^{-1}\rangle\!\rangle$
satisfy the quantum Stanley-Reisner relation. (See Proposition 6.7 \cite{fooo08}.)
\item The proof of Lemma  \ref{characterizationlamda0} implies the following:
A {\it monomial} of the form (\ref{elecomp}) is a {\it monomial} in
$\Lambda_0^P\langle\!\langle\mathfrak w,\mathfrak w^{-1},w,y,y^{-1}\rangle\!\rangle$ and vice versa.
\item The discussion above shows that the
moment polytope $P$ is closely related to the Berkovich spectrum
\cite{Ber}, \cite{KS} of $\Lambda_0^P\langle\!\langle\mathfrak w,\mathfrak w^{-1},w,y,y^{-1}\rangle\!\rangle$.
\end{enumerate}
\end{rem}

Now we can state the following theorem whose proof will be
postponed until Section \ref{sec:domain}.
\begin{thm}\label{extendth} Let $u \in \operatorname{Int}\,P$.
\begin{enumerate}
\item The function $\mathfrak{PO}^u \circ \psi_u$
does not depend on $u$. We denote the common function by 
$\mathfrak{PO}$.
\item
$\mathfrak{PO}$ lies in $\Lambda_0^P\langle\!\langle \mathfrak w,\mathfrak
w^{-1},w,y,y^{-1}\rangle\!\rangle $ as a function on $\CA(\Lambda_+)
\times H^1(T^n;\Lambda_0)$.
\end{enumerate}
\end{thm}
The $\C$-reduction of $\mathfrak{PO}$ corresponds to the precise
form of the physicists' Landau-Ginzburg potential function
associated to toric manifolds. (See \cite{hori-vafa}.)
\begin{rem}\label{rem:notation} 
This is a notational remark.
When we write the potential function as
$\mathfrak{PO}^u(y_1,\dots ,y_n)$ by specifying 
the dependence of $u$, 
the variables $y_i$ also depend on $u$ and 
stand for $y_i(u)$ in 
(\ref{yiu}). 
However, to simplify the notation, we will also 
write $y_i$ for the variables of 
$\mathfrak{PO}^u(y_1,\dots ,y_n)$, when 
no confusion can occur.   
On the other hand, as for the function  
$\mathfrak{PO}(y_1,\dots ,y_n)$ introduced 
in Theorem \ref{extendth},  
the variables $y_i$ do not depend on $u$. 
In this paper we do not use the potential function $\mathfrak{PO}$, except  
for Theorem \ref{extendth} and Remark 
\ref{anounce}. 
(In Section \ref{sec:bulk} we use 
the notaion $\mathfrak{PO}$ in general context.)   
The function $\mathfrak{PO}$ independent of $u$ 
is used in \cite{fooo10}.  
\end{rem}
\subsection{Application to Lagrangian intersections}

Now we can generalize the result of Section 4 \cite{fooo08} as
follows. Using Lemma \ref{lem:extend} and the idea of Cho (see
Section 12 \cite{fooo08}) this time applied to ambient
toric manifold $X$, instead of Lagrangian submanifolds, we can define Floer cohomology
$$
HF((L(u),\mathfrak b,\mathfrak x),(L(u),\mathfrak b,\mathfrak x);\Lambda_0)
$$
for any $(\mathfrak b,\mathfrak x) \in \mathcal A(\Lambda_0) \times
H^1(L(u);\Lambda_0) \cong (\Lambda_0)^{B} \times (\Lambda_0)^n$. See
Sections \ref{sec:HFbulk}
(the case 
$\mathfrak b \in \mathcal A(\Lambda_+)$)
and \ref{sec:byLambda0} (the case $\mathfrak b \in \mathcal A(\Lambda_0)$)). The following is
a generalization of Theorem 4.10 \cite{fooo08}.
\begin{thm}\label{homologynonzero}
Let $(\mathfrak b,\mathfrak x) \in 
\mathcal A(\Lambda_0) \times
H^1(L(u);\Lambda_0) \cong (\Lambda_0)^B \times (\Lambda_0)^n$ and
denote $\mathfrak y = \exp(\mathfrak x) = (e^{\mathfrak x_1},\ldots,e^{\mathfrak x_n})$.
If $(\mathfrak b,\mathfrak x)$ satisfies
\begin{equation}\label{crit}
\mathfrak y_i \frac{\partial \mathfrak{PO}^u}{\partial y_i}(\mathfrak b,\mathfrak y) = 0
\end{equation}
for $i=1,\ldots,n$, then we have
\begin{equation}\label{eq:homoiso}
HF((L(u_0),\mathfrak b,\mathfrak x),(L(u_0),\mathfrak b,\mathfrak x);\Lambda_0)
\cong H(T^n;\Lambda_{0}).
\end{equation}
\par
If $(\mathfrak b,\mathfrak x)$ satisfies
\begin{equation}\label{crit2}
\mathfrak y_i \frac{\partial \mathfrak{PO}^u}{\partial y_i}(\mathfrak b,\mathfrak y) \equiv 0 \mod T^{\mathcal N},
\end{equation}
then we have
\begin{equation}\label{eq:homoiso2}
HF((L(u_0),\mathfrak b,\mathfrak x),(L(u_0),\mathfrak b,\mathfrak x);\Lambda_0/T^{\mathcal N})
\cong H(T^n;\Lambda_0/T^{\mathcal N}).
\end{equation}
\end{thm}
Theorem \ref{homologynonzero} will be proved in Sections  \ref{sec:HFbulk} and  \ref{sec:byLambda0}.
\par
We next define:
\begin{defn}\label{balanced}
Let $L(u)$ be a Lagrangian fiber of a compact toric manifold
$(X,\omega)$. We say that $L(u)$ is {\it bulk-balanced} if there
exist sequences $\omega_i$, $P_i$, $u_i$, 
$\mathfrak b_i$, $\mathfrak x_i$ and
$\mathcal N_i$ with the following properties.
\begin{enumerate}
\item $(X,\omega_i)$ is a sequence of toric manifolds such that $\lim_{i\to\infty}\omega_i = \omega$.
\item
$P_i$ is a moment polytope of $(X,\omega_i)$. It converges to the
moment polytope $P$ of $(X,\omega)$.
\item $u_i \in P_i$ and $\lim_{i\to\infty} u_i =u$.
\item $\mathfrak b_i \in \mathcal A(\Lambda_+(\C))$, $\mathfrak x_i \in H^1(L(u_i);\Lambda_0(\C))$,
$\mathcal N_i \in \R_+$.
\item
$$
HF((L(u_i),\mathfrak b_i,\mathfrak x_i),((L(u_i),\mathfrak b_i,\mathfrak x_i);\Lambda_0(\C)/T^{\mathcal N_i})
\cong H(T^n;\Lambda_0(\C)/T^{\mathcal N_i}).
$$
\item
$\lim_{i\to\infty}\mathcal N_i = \infty$.
\end{enumerate}
\end{defn}
\begin{rem}
\begin{enumerate}
\item  Definition \ref{balanced} is related to Definitions 4.11 \cite{fooo08}.
Namely it is easy to see that
$$
\mbox{``Strongly balanced'' $\Rightarrow$ ``balanced''$\Rightarrow$
``bulk-balanced''.}
$$
On the other hand the three notions  are all different. (See Example 10.17 \cite{fooo08} and
Section \ref{sec:Exam} of the present paper.)
\item In Section \ref{sec:byLambda0}  
we generalize Theorem \ref{homologynonzero} to 
Proposition \ref{lambda0General} for the case 
$\mathfrak b \in \mathcal A(\Lambda_0(\C))$ in place of
$\mathfrak b \in \mathcal A(\Lambda_+)$. 
We note that we assume $R =\C$ in the generalization, while we do not assume it in Theorem \ref{homologynonzero}. 
See Remark \ref{Rrationality}.
\end{enumerate}
\end{rem}
The next result is a generalization of Proposition 4.12 \cite{fooo08}
which will be proved in Section \ref{sec:HFbulk}.
\begin{prop}\label{prof:bal2}
Suppose that $L(u) \subset X$ is bulk-balanced. Then $L(u)$ is
non-displaceable.
\par
Moreover if $\psi: X \to X$ is a Hamiltonian diffeomorphism such that
$\psi(L(u))$ is transversal to $L(u)$, then
\begin{equation}\label{estimateint}
\#(\psi(L(u))  \cap L(u)) \ge 2^n.
\end{equation}
\end{prop}
\par
It seems reasonable to expect the following converse to this proposition.
\par
\begin{prob}\label{conjdisp}
Are all the non-displaceable fibers  $L(u)$ of a compact toric manifold
bulk-balanced?
\end{prob}

\section{Elimination of higher order terms by bulk deformations}
\label{sec:Ellim}
\par
The purpose of this section is to apply the result of the last
section to locate bulk-balanced Lagrangian fibers. We first recall
the notion of leading term equation which was introduced in Section 9 \cite{fooo08} for the case 
$\mathfrak b=0$. We denote by
$\Lambda_0\langle\!\langle y,y^{-1}\rangle\!\rangle$ the completion of the Laurent polynomial
ring $\Lambda_0[y_1,y_1^{-1}, \ldots,y_n,y_n^{-1}]$ with respect to
the non-Archimedean norm. For each fixed $\mathfrak b \in
\CA(\Lambda_+)$ and $u$, we put
$$
\mathfrak{PO}^u_{\mathfrak b}(y_1,\ldots,y_n) = 
\mathfrak{PO}^u(\mathfrak b;y_1,\dots ,y_n) = 
\mathfrak{PO}^u(\mathfrak b;b)
$$
where $b = \sum_{i=1}^n x_i \text{\bf e}_i$ and $y_i = \exp(x_i)$. Then
we have
$$
\mathfrak{PO}^u_{\mathfrak b}(y_1,\dots,y_n) \in \Lambda_0\langle\!\langle y,y^{-1}\rangle\!\rangle
$$
and so regard $\mathfrak{PO}^u_{\mathfrak b}$ as an element of $\Lambda_0\langle\!\langle y,y^{-1}\rangle\!\rangle$.
\par

Henceforth we write $y^{\vec v}$ for $y_1^{v_1}\cdots y_n^{v_n}$
with $\vec v = (v_1,\ldots,v_n)$.
\par
Let $\vec v_i = d\ell_i = (v_{i,1},\ldots,v_{i,n}) \in H_1(L(u);\Z)
\cong \Z^n \cong N_{\Z}$ $(i=1,\ldots,m)$ as in 
(\ref{eq:v}). We define $S_{l} \in \R_{+}$ by $S_{l} <
S_{l+1}$ and
\begin{equation}\label{lamdai}
\{S_{l} \mid l = 1,2,\ldots, \mathcal L\} = \{\ell_i(u) \mid i=1,2,\ldots, m\}.
\end{equation}
We re-enumerate the set $\{\vec v_k \mid \lambda_k = S_l \}$ as
\begin{equation}\label{vljdef}
\{\vec v_{l,1},\ldots,\vec v_{l,a(l)}\}.
\end{equation}
\par
Let $A_{l}^\perp \subset N_{\R} \cong \R^n$ be the $\mathbb R$-vector space
generated by $\vec v_{l',r}$ for $l' \le l$, $r=1,\ldots,a(l')$. We
remark that $A_{l}^\perp$ is defined over $\Q$.
Namely $A_{l}^\perp \cap \Q^n$ generates $A_{l}^\perp$ as an $\R$ vector space.
Denote by $K$ the
smallest integer $l$ such that $A_l^\perp = N_{\R}$. We put
$d(l) = \dim A_{l}^\perp - \dim A_{l-1}^\perp$,
$d(1) = \dim A_1^\perp$. 
In particular, we have 
$n=\sum_{l=1}^K d(l).$
\par
We remark
$$
\{\vec v_{l,r} \mid l=1,\ldots,K, r=1,\ldots,a(l)\} \subset \{\vec v_i \mid i=1,\ldots,m\}.
$$
Henceforth we assume $l\le K$ whenever we write $\vec v_{l,r}$. For
each $(l,r)$ we define the integer $i(l,r) \in \{1,\ldots,m\}$ by
\begin{equation}\label{ilj}
\vec v_{l,r} = \vec v_{i(l,r)}.
\end{equation}
Renumbering $\vec v_i$, if necessary, we can enumerate them so that
\begin{equation}\label{vrest}
\aligned
&\{\vec v_i \mid i=1,\ldots,m\} \\
&= \{\vec v_{l,r}
\mid
l=1,\ldots,K, r=1,\ldots,a(l)\}
\cup \{ \vec v_{i} \mid i = \mathcal K+1 ,\ldots, m\}
\endaligned\end{equation}
for some $1 \leq \mathcal K \leq m$ with
$$
\mathcal K = \sum_{l =1}^K a(l).
$$
\par
Recall that we have chosen a basis $\text{\bf e}_i$ of $H^1(L(u);\Z)$ in
the end of Section \ref{sec:introduction}. It can be identified with
a basis of $M_{\R} \cong H^1(L(u);\R)$. Denote its dual basis on
$N_{\R}$ by $\text{\bf e}^*_i$.
\par
We choose a basis $\text{\bf e}^*_{l,s}$ of $N_{\R}$ such that
$\text{\bf e}^*_{1,1},\ldots,\text{\bf e}^*_{l,d(l)}$ forms a
$\Q$-basis of $A_{l}^\perp$ and that each of $\vec v_i$ lies in
$\bigoplus_{l,s} \Z \text{\bf e}^*_{l,s}$.
\par
We put
$$
\text{\bf e}_{i}^* = \sum_{l=1}^K\sum_{s=1}^{d(l)} a_{i;(l,s)} \text{\bf e}_{l,s}^*.
$$
(Here $a_{(l,s);i} \in \Z$ since $\vec v_i$ generates $H^2(L(u);\Z)$.)
Regarding $\text{\bf e}_i^*$ and $\text{\bf
e}_{l,s}^*$ as functions on $M_\R$, this equation can be written as
$$
x_i =\sum_{l=1}^K\sum_{s=1}^{d(l)} a_{i;(l,s)} x_{l,s}
$$
with $x_i =  \text{\bf e}_{i}^*$ and $x_{l,s} = \text{\bf
e}_{l,s}^*$. If we associate $y_{l,s} = e^{x_{l,s}}$, it is
contained in a finite field extension of
$\Q[y_1,y_1^{-1},\ldots,y_n,y_n^{-1}]$ and satisfies
\begin{equation}\label{yandy}
y_i= \prod_{l=1}^K\prod _{s=1}^{d(l)} y_{l,s}^{a_{i;(l,s)}}.
\end{equation}
Here we note that since 
$\sum_{l=1}^K d(l)=n$, the number of variables
$y_{l,s}$'s is equal to the number of variables 
$y_i$'s which is $n$. 
We put $\vec v_{l,r} = (v_{l,r;1},\ldots,v_{l,r;n}) \in \Z^n$.

\begin{lem}\label{stratifiedstr} The product
$$
y^{\vec v_{l,r}} = y_1^{v_{l,r;1}}\cdots y_n^{v_{l,r;n}}
$$
is a monomial of $y_{l',s}$ for $l' \le l$, $s\le d(l')$.
\end{lem}
\begin{proof} By the definition of $A_\ell^\perp$,
$\vec v_{l,r}$ is an element of $A_{l}^\perp$ and so
$$
\vec v_{l,r} = \sum_{l' \le l, s\le d(l')} c_{l,r;l',s} \text{\bf e}^*_{l',s}
$$
for some integers $c_{l,r;l',s}$. Therefore
$$
y^{\vec v_{l,r}} = \prod_{l' \le l, s\le d(l')}
y_{l',s}^{c_{l,r;l',s}}
$$
and the lemma follows.
\end{proof}

We note that the leading order potential function 
of $\mathfrak{PO}^u_{\mathfrak b}$ does not depend on $\mathfrak b \in \mathcal A(\Lambda_+)$.
We denote this leading order potential function 
 by $\mathfrak{PO}^u_0$. To distinguish the notation
for $\mathfrak{PO}^u_{\mathfrak b}$ for $\mathfrak b = 0$, we will denote the latter by
$\mathfrak{PO}^u_{\bf 0}$ when the latter appears.

We put
\begin{equation}\label{POsital}
(\mathfrak{PO}^u_0)_l = \sum_{r=1}^{a(l)} y^{\vec v_{l,r}}, 
\quad l=1,\dots ,K.
\end{equation}
By Lemma \ref{stratifiedstr}, 
$(\mathfrak{PO}^u_0)_l$ can be
written as a Laurent polynomial of $y_{l',s}$ for $l' \le l$, $s\le
d(l')$ with its coefficients are scalars i.e., \emph{elements of $R$}.
\par
Now we consider the system of equations
\begin{equation}\label{critPP}
y_k\frac{\partial \mathfrak{PO}^u_{\mathfrak b}}{\partial y_k} = 0
\end{equation}
with $k = 1,\ldots,n$ for $y_k$ \emph{from $\Lambda_0$}. By changing the
coordinates to $y_{l,s}$ $(l=1,\ldots,K, s=1,\ldots,d(l))$,
(\ref{critPP}) becomes
\begin{equation}\label{critPP2}
y_{l,s}\frac{\partial \mathfrak{PO}^u_{\mathfrak b}}{\partial y_{l,s}} = 0.
\end{equation}
(We recall that $y_k \frac{\del}{\del y_k}$ is the logarithmic derivative.)
Note $a_{(l,s);i} \in \Z$ implies
$\Q[y_1,y_1^{-1},\ldots,y_n,y_n^{-1}]
\subset \Q[y_{1,1},y_{1,1}^{-1},\ldots,y_{K,d(K)},y_{K,d(K)}^{-1}]$.
\begin{lem}\label{cofcorlem} 
Suppose that $R$ is an algebraically closed filed. 
The system of equations
$(\ref{critPP2})$ has a solution with $y_{l,s}$ 
$(l=1,\dots , K, s=1,\dots ,d(l))$ from $\Lambda_0(R) \setminus
\Lambda_+(R)$ if and only if 
the system of equations $(\ref{critPP})$ has a solution with $y_k \in
\Lambda_{0}(R) \setminus \Lambda_{+}(R)$, 
$k=1,\dots,n$.
\par
The ratio between the numbers
of solutions counted with multiplicity is equal to the degree of
field extension
$$
\left[\Q[y_{1,1},y_{1,1}^{-1},\ldots,y_{K,d(K)},y_{K,d(K)}^{-1}]:
\Q[y_1,y_1^{-1},\ldots,y_n,y_n^{-1}] \right].
$$
\end{lem}
\begin{proof}
This is obvious from the form of the change of coordinate
(\ref{yandy}).
\end{proof}
\begin{defn}\label{LOEdef}
The {\it leading term equation} of $(\ref{critPP})$ or of
$(\ref{critPP2})$ is the system of equations
\begin{equation}\label{LTE}
\frac{\partial (\mathfrak{PO}^u_0)_l}{\partial y_{l,s}} = 0, \qquad 
l=1,\ldots,K;\, s=1,\ldots,d(l)
\end{equation}
with $y_{l,s}$ from $R\setminus  \{0\}$.
\end{defn}
We emphasize that the leading term equation does not depend on $\mathfrak b$ but
depends only on the leading order potential $\mathfrak{PO}^u_0$.
We remark that (\ref{critPP}) or \eqref{critPP2} is an equation for $y_1,\ldots,y_n$
or $y_{l,s}$'s in $\Lambda_0$ respectively. On the other hand, the equation (\ref{LTE}) is one for
$y_{l,s} \in R\setminus \{0\}$.
\par
The following lemma describes the
relation between these two equations.
\begin{lem}\label{relLTE}
Let $y_{l,s} \in \Lambda_{0}(R) \setminus \Lambda_{+}(R)$, 
$l=1,\dots , K, s=1,\dots ,d(l)$ be a solution of $(\ref{critPP2})$.
We define $\overline y_{l,s} \in R\setminus\{0\}$ by $y_{l,s} \equiv \overline y_{l,s} \mod \Lambda
_+(R)$. Then $\overline y_{l,s}$ solves the leading term equation $(\ref{LTE})$.
\end{lem}
\begin{proof}
By Theorem \ref{weakpotential}, each exponent 
$\ell'_{\sigma}(u)+\rho_{\sigma}$ 
of $T$ in the right hand side of 
(\ref{eq:weakPO}) is strictly positive.
Therefore Lemma \ref{relLTE} follows directly form 
the definition of the leading term equation.
\end{proof}
\par
We remark that the discussion above applies to the leading order
potential function $\mathfrak {PO}_0^u$ (See (\ref{PO0})) without
changes. See Sections 9, 10 \cite{fooo08}.
We denote by $\Lambda_0\langle\!\langle y_{**},y_{**}^{-1}\rangle\!\rangle$ the completion of
$\Lambda_0[y_{1,1},y_{1,1}^{-1},\ldots,y_{K,d(K)},y_{K,d(K)}^{-1}]$
with respect to the non-Archimedean norm. It is a finite field
extension of $\Lambda_0\langle\!\langle y,y^{-1}\rangle\!\rangle$.
\par
Now we state the main result of this section.
\begin{thm}\label{ellim}
The following three conditions on $u$ are equivalent to each other:
\begin{enumerate}
\item The leading term equation of $\mathfrak{PO}^u$ has a solution
$y_{l,s} \in R \setminus \{0\}$ 
$(l=1,\dots , K, s=1,\dots ,d(l))$.
\item There exists $\mathfrak b \in \mathcal A(\Lambda_+(R))$ such that $\mathfrak{PO}^u_{\mathfrak b}$
has a critical point on $(\Lambda_{0}(R) \setminus \Lambda_+(R))^n$.
\item There exists $\mathfrak b \in \mathcal A(\Lambda_+(R))$ such that $y_{l,s} \in R \setminus \{0\}$ 
$(l=1,\dots , K, s=1,\dots ,d(l))$ in the item (1) above is a critical point of $\mathfrak{PO}^u_{\mathfrak b}$. 
\end{enumerate}
\end{thm}
By Theorem \ref{homologynonzero} we obtain 
the following corollary. 
\begin{cor}\label{weabalanced}
If the leading term equation of $\mathfrak{PO}_0^u$ has a solution, 
then $L(u)$ is bulk-balanced.
\end{cor}
\begin{proof}[Proof of Corollary  
\ref{weabalanced}] 
This follows from Theorem \ref{homologynonzero} and Theorem \ref{ellim}.
\end{proof}
\begin{proof}[Proof of Theorem \ref{ellim}]
Since $R \setminus \{ 0\} \subset \Lambda_0(R) \setminus \Lambda_+(R)$, (3) $\Rightarrow$ (2) 
is obvious. 
The proof of (2) $\Rightarrow$ (1) is a consequence of
Lemmata \ref{cofcorlem}, \ref{relLTE}. 
The rest of this section is devoted to the proof of (1) $\Rightarrow$ (3).

Let $\mathfrak y_{1,1},\ldots,\mathfrak y_{K,d(K)}$ be a solution of the
leading term equation. We remark $\mathfrak y_{l,s} \in R \setminus
\{0\} \subset \Lambda_{0}(R) \setminus \Lambda_+(R)$. 
We will fix 
$\mathfrak y_{l,s}$ during the proof of 
(1) $\Rightarrow$ (3). 
We will construct, by induction on energy, a suitable bulk deformation 
$\mathfrak b$ such that 
$\mathfrak y_{l,s}$ 
$(l=1,\dots , K, s=1,\dots ,d(l))$ 
is a critical point of $\mathfrak{PO}^u_{\mathfrak b}$. 
We also require $\mathfrak b$ to have the form
\begin{equation}\label{formfrakb}
\mathfrak b = \sum_{l=1}^{K}\sum_{r=1}^{a(l)} \mathfrak b_{l,r} D_{i(l,r)}
\end{equation}
where $\mathfrak b_{l,r} \in \Lambda_+$.
(Here and hereafter in this section we omit $R$ in $\Lambda_+(R)$ and etc.)
\par
Note $i(l,r) \le m$ and so $\deg D_{i(l,r)} = 2$. In other words, we
use $\mathfrak b$ only in the second cohomology $H^2(X;\Lambda_+)$ (more
precisely $\mathfrak b \in \mathcal A^2(\Lambda_+)$) to prove Theorem
\ref{ellim}.
\par
We divide the rest of the proof into the two cases, one the rational Fano
and the other the rest.

\subsection{The rational Fano case}

Assume $X$ is Fano. In this case we can
calculate  $\mathfrak{PO}^u_{\mathfrak b}(y)$ explicitly as follows.
\begin{prop}\label{POcalcFano}
Suppose $X$ is Fano and 
$\mathfrak b\in \mathcal A^2(\Lambda_+)$ is as in $(\ref{formfrakb})$.
Then
\begin{equation}\label{POFanoformula}
\mathfrak{PO}^u_{\mathfrak b}(y)= \sum_{l=1}^{K}\sum_{r=1}^{a(l)}
\exp(\mathfrak b_{l,r}) T^{S_l} y^{\vec v_{i(l,r)}}
+ \sum_{i= \mathcal K+1}^m T^{\ell_i(u)} y^{\vec v_i}.
\end{equation}
\end{prop}
We will prove Proposition \ref{POcalcFano} in Section
\ref{sec:CalPot}.
\par
We put
$$
\vec v_{i(l,r)} = \sum_{l' = 1}^l\sum_{s=1}^{d(l')} v_{l,r;l',s} \text{\bf e}_{l',s}^*,
\quad
\vec v_{i} = \sum_{l=1}^K\sum_{s=1}^{d(l)} v_{i;l,s} \text{\bf e}_{l,s}^*.
$$
\begin{lem}\label{derivforFano}
If $\{\mathfrak y_{l,r}\}_{1 \leq l \leq K,1 \leq r \leq d(l)} \subset R \setminus \{0\}$
is a solution of the system of leading term equations, then we have
\begin{equation}\label{critFano}
\aligned \mathfrak y_{l',s} \frac{\partial \mathfrak{PO}^u_{\mathfrak b}}{\partial
y_{l',s}} (\mathfrak y) = &\sum_{l=l'}^{K}\sum_{r=1}^{a(l)}
\left(\mathfrak b_{l,r} + \sum_{h=2}^{\infty} \frac{1}{h!} \mathfrak
b_{l,r}^h \right) T^{S_l} v_{l,r;l',s} \mathfrak y^{\vec v_{i(l,r)}} \\
& +\sum_{l=l'+1}^{K}\sum_{r=1}^{a(l)} T^{S_l}v_{l,r;l',s} \mathfrak y^{\vec v_{i(l,r)}}+
\sum_{i= \mathcal K+1}^m  v_{i;l',s} T^{\ell_i(u)} \mathfrak y^{\vec v_i}
\endaligned
\end{equation}
for all  $1 \leq l' \leq K$ and $1 \leq s \leq d(l')$. 
Here
$\mathfrak x = \text{\rm Log}(\mathfrak y) = \sum (\log \mathfrak y_i)\text{\bf e}_i$ and $\mathfrak y_i$ is determined from
$\mathfrak y_{l,s}$ by $(\ref{yandy})$.
\end{lem}
\begin{proof} Differentiating \eqref{POFanoformula}, we obtain
$$
\aligned y_{l',s} \frac{\partial \mathfrak{PO}^u_{\mathfrak b}}{\partial
y_{l',s}}(y) = &\sum_{l=l'}^{K}\sum_{r=1}^{a(l)}
\left(1+\mathfrak b_{l,r} + \sum_{h=2}^{\infty} \frac{1}{h!}  \mathfrak
b_{l,r}^h \right)
T^{S_l} v_{l,r;l',s} y^{\vec v_{i(l,r)}} \\
&+ \sum_{i= \mathcal K+1}^m  v_{i;l',s} T^{\ell_i(u)} y^{\vec v_i}\\
= &\sum_{l=l'}^{K}\sum_{r=1}^{a(l)}
\left(\mathfrak b_{l,r} + \sum_{h=2}^{\infty} \frac{1}{h!}  \mathfrak
b_{l,r}^h \right)
T^{S_l} v_{l,r;l',s} y^{\vec v_{i(l,r)}} \\
&+
T^{S_{l'}} \sum_{r=1}^{a(l')} v_{l',r;l',s} y^{\vec v_{i(l',r)}}
+ \sum_{l=l'+1}^{K}\sum_{r=1}^{a(l)}T^{S_l} v_{l,r;l',s} y^{\vec v_{i(l,r)}}\\
&+ \sum_{i= \mathcal K+1}^m  v_{i;l',s} T^{\ell_i(u)} y^{\vec v_i}.
\endaligned
$$
On the other hand, the leading term equation provides vanishing
$$
0 = \sum_{r=1}^{a(l')} v_{l',r;l',s} \mathfrak y^{\vec v_{i(l',r)}}.
$$
Therefore (\ref{critFano}) follows.
\end{proof}
\par
To highlight the idea of the proof of Theorem \ref{ellim}, we first consider the
\emph{rational} case.
We recall the definition of rational Lagrangian submanifolds.

\begin{defn}
We say that $(X,\omega)$ is rational if $c [\omega] \in H^2(X;\Q)$ for
some $c\in \R \setminus \{0\}$.
We say that a Lagrangian submanifold $L \subset X$ is rational if the subgroup
$
\{\omega \cap \beta \mid  \beta \in H_2(X,L;\Z)\} \subset \R
$
is isomorphic to $\Z$ or $\{0\}$.
\end{defn}
We remark that only rational symplectic manifold $(X,\omega)$
carries a rational Lagrangian submanifold $L$. (In the general
situation $\pi_2(X,L)$ is used sometimes in the definition of
rationality of $L$. In our case of toric fibers, they are
equivalent.)
\par
In this rational case, by rescaling the symplectic form
$\omega$ to $c \omega$ by some $c\in \R_+$, we may assume that $\omega$ is
integral, i.e.,
$$
\{ \omega \cap \beta/2\pi \mid \beta \in H_2(X,L(u);\Z) \} \subset \Z.
$$
It follows that $S_l, \ell_i(u) \in \Z$. Thus, we can reduce the coefficient
rings from the universal Novikov rings $\Lambda_0$, $\Lambda_+$, $\Lambda$
to the following rings respectively:
$$
\Lambda_0^{\text{int}}: = R[[T]], \quad \Lambda_+^{\text{int}}: =
TR[[T]], \quad \Lambda^{\text{int}}: = R[[T]][T^{-1}].
$$
Here $R[[T]]$ is the formal power series ring.
\par
We also consider the pairs $(\mathfrak b,b)$ only from
$\mathcal A^2(\Lambda_+^{\text{int}}) \times H^1(L(u);\Lambda_0^{\text{int}})$.
Under these restrictions, the exponents of $T$ appearing in the
following discussion always become integers.
\begin{lem}\label{indmain}
Suppose $X$ is Fano and $L(u)$ is rational.
We also assume that $\mathfrak y_{l,s} \in R \setminus \{0\}$ ($l=1,\ldots,K$, $s=1,\ldots,d(l)$) is a solution of the leading
term equation.  Then,
for each $k,l,r$, there exist $\mathfrak b_{l,r}(k) \in
\Lambda_+^{\text{\rm int}}$ such that the sum
$$
\mathfrak b(k) = \sum_{l=1}^{K}\sum_{r=1}^{a(l)} \mathfrak b_{l,r}(k) D_{i(l,r)}
$$
satisfies the equation
\begin{equation}\label{concindmain}
\mathfrak y_{l,s}\frac{\partial \mathfrak{PO}^{u}_{\mathfrak b(k)}}{\partial y_{l,s}}(\mathfrak y) \equiv 0 \mod T^{k}
\end{equation}
for $l=1,\ldots,K$, $s=1,\ldots,d(l)$.
We also have
\begin{equation}\label{bkconv}
\mathfrak b_{l,r}(k+1) \equiv \mathfrak b_{l,r}(k) \mod T^{k - S_l}.
\end{equation}
\end{lem}
\begin{proof}
The proof is by an induction over $k$.
If
$k \le S_1$,
we apply Lemma \ref{derivforFano} to $\mathfrak b = \mathfrak b(k)=0$
and obtain
$$
\mathfrak y_{l',s}\frac{\partial \mathfrak{PO}^{u}_{\mathfrak b(k)}}
{\partial y_{l',s}}(\mathfrak y) \equiv 0
\mod T^{S_1}.
$$
Hence (\ref{concindmain}) holds for $k
\le S_1$.
\par
Now suppose $k > S_1$ and assume $\mathfrak b(k-1)$ with the required
property.  By the induction hypothesis we may put
\begin{equation}\label{Eljform}
\mathfrak y_{l',s}\frac{\partial \mathfrak{PO}^{u}_{\mathfrak b(k-1)}}{\partial y_{l',s}}(\mathfrak y)
\equiv T^{k} E_{l',s} \mod T^{k+1},
\end{equation}
for some $E_{l',s} \in R$. Let
$$
\vec E = \sum_{l'=1}^K \sum_{s=1}^{d(l')} E_{l',s}\text{\bf
e}^*_{l',s} \in N_R = N \otimes_\Z R.
$$
\begin{sublem}\label{whereisEin}
$\vec E$ is contained in the vector space generated by
$\{ e^*_{l,s} \mid S_l < k, \,\, s = 1,\ldots,d(l)\}$.
\end{sublem}
\begin{proof}
This is a consequence of (\ref{critFano}).
In fact,  (\ref{critFano}) implies that only the first term of the right hand side 
contributes to $E_{l',s}$. (Note that the coefficients of
the first term do not vanish.)
\end{proof}
By Sublemma \ref{whereisEin}, we can express $\vec E$ as
\begin{equation}\label{cljform}
-\vec E = \sum_{S_l < k} c_{l,r}  \vec v_{i(l,r)}
\end{equation}
for some $c_{l,r} \in R$. Note $\vec v_{i(l,r)}$, $1 \le l\le l_0$, $1\le r \le
a(l)$ span the vector space
$$
A^{\perp}_{l_0} = \operatorname{span}_R\{e^*_{l,s} \mid l \le
l_0,\,\, s=1,\ldots,d(l) \}.
$$
We define $\mathfrak b_{l,r}(k)$ by
$$
\mathfrak b_{l,r}(k) = \mathfrak b_{l,r}(k-1) + c_{l,r}(\mathfrak y^{\vec v_{i(l,r)}})^{-1}T^{k-S_l} D_{i(l,r)}.
$$
Since $k-S_l > 0$, it follows $k-S_l \in \Z_+$ by the integrality
hypothesis of $\omega$. Namely $\mathfrak b_{l,r}(k) \in
\Lambda_+^{\text{\rm int}}$. Therefore Lemma \ref{derivforFano},
(\ref{Eljform}) and (\ref{cljform}) imply (\ref{concindmain}).
This finishes the induction steps and so the proof of Lemma
\ref{indmain} is complete.
\end{proof}
Now we consider
$$
\mathfrak b = \lim_{k\to\infty} \mathfrak b(k).
$$
The right hand side converges by (\ref{bkconv}) and so $\mathfrak b$ is
well-defined as an element of $\CA(\Lambda^{\text{\rm int}}_+)$ and satisfies
$$
\mathfrak y_{l',s}\frac{\partial \mathfrak{PO}^{u}_{\mathfrak b}}{\partial y_{l',s}}(\mathfrak y) = 0
$$
as required. Thus Theorem \ref{ellim} is  proved for the case where
$X$ is Fano and $L(u)$ is rational.
\par\medskip
\subsection{General case}

We now turn to the case where $X$ is not necessarily Fano or $L(u)$
is not necessarily rational. We will still use an induction argument
but we need to carefully choose the discrete submonoids of $\R$
to carry out the induction argument.
\par
Let $G(X)$ be as in (\ref{gapX}).
We define
\begin{equation}\label{gapL}
G(L(u)) =
\langle \{\omega[\beta]/2\pi \mid  \text{$\beta\in
\pi_2(X,L(u))$ realized by a holomorphic disc}\} \rangle.
\end{equation}
\begin{defn}\label{Gbulk} Let $G(X)$ be as in \eqref{gapX}. We
define $G_{\text{\rm bulk}}$ to be the discrete submonoid of $\R$ generated by
$G(X)$ and the subset
$$
\{\lambda - S_l \mid \lambda \in G(L(u)), \quad l=1,\ldots,K,
\lambda > S_l \} \subset \R_+ \subset \R.
$$
\end{defn}
We remark that $G_{\text{\rm bulk}} \supset G(L(u))$.
\begin{conds}\label{gapcondb}
We consider
\begin{equation}\label{bkconv2}
\mathfrak b = \sum_{l=1}^{K}\sum_{r=1}^{a(l)} \mathfrak b_{l,r}
D_{i(l,r)} \in \mathcal A^2(\Lambda_+)
\end{equation}
such that all $\mathfrak b_{l,r}$ are $G_{\text{\rm bulk}}$-gapped.
\end{conds}
\par
The main geometric input to the proof of the non-Fano case of Theorem \ref{ellim}
is the following.
\begin{prop}\label{nonfanopert}
Suppose that $\mathfrak b$ satisfies Condition $\ref{gapcondb}$ and
consider
\begin{equation}\label{babdbprime}
\mathfrak b^{\prime} = \mathfrak b + cT^{\lambda} D_{i(l,r)},
\end{equation}
with $c \in R$, $\lambda \in G_{\text{\rm bulk}}$, $l\le K$. Then we have
\begin{equation}\label{firstapprox}
\aligned
\mathfrak{PO}^u_{\mathfrak b^{\prime}}(y)
- \mathfrak{PO}^u_{\mathfrak b}(y)
&= & c T^{\lambda + \ell_{i(l,r)}(u)} y^{\vec v_{i(l,r)}}  +
\sum_{h=2}^{\infty} c_h T^{h\lambda + \ell_{i(l,r)}(u)} y^{\vec v_{i(l,r)}}\\
&{}&  + \sum_{h=1}^{\infty}\sum_{\sigma}
c_{h,\sigma} T^{h\lambda + \ell'_{\sigma}(u) + \rho_{\sigma}}
y^{\vec v_{\sigma}}.
\endaligned
\end{equation}
Here $c_h, c_{h,\sigma} \in R$, $\rho_{\sigma} \in G_{\text{\rm bulk}}$. Moreover
there exist $e_{\sigma}^{i} \in \Z_{\ge 0}$ such that $\vec v_{\sigma} = \sum e_{\sigma}^{i} \vec v_i$,
$\ell'_{\sigma} = \sum e_{\sigma}^{i} \ell_i$ and $\sum_{i}e_{\sigma}^{i} > 0$.
\end{prop}

We will prove Proposition \ref{nonfanopert} in Section \ref{sec:CalPot}.

\begin{defn}
We enumerate the elements of $G_{\text{\rm bulk}}$ so that
$$
G_{\text{\rm bulk}} = \{ \lambda^b_j \mid j=0,1,2,\ldots\}
$$
where $0 = \lambda^b_0 < \lambda^b_1 < \lambda^b_2 < \cdots$.
\begin{enumerate}
\item For $k \geq 1$, we define $\Lambda^{G_{\text{\rm bulk}}}_0\langle\!\langle  y_{**},y_{**}^{-1}
\rangle\!\rangle _k$ to be a subspace of
$\Lambda_0\langle\!\langle y_{**},y_{**}^{-1}\rangle\!\rangle$ consisting of elements of the form
\begin{equation}\label{LamGbulk}
\sum_{l=1}^K\sum_{j=k}^{\infty} T^{S_l + \lambda^b_j}
P_{j,l}(y_{1,1},y_{1,1}^{-1},\ldots,y_{l,d(l)},y_{l,d(l)}^{-1})
\end{equation}
where each $P_{j,l}$ is a Laurent polynomial of
$y_{1,1},\ldots,y_{l,d(l)}$ with $R$ coefficients, i.e.,
$$
P_{j,l} \in
R[y_{1,1},y_{1,1}^{-1},\ldots,y_{l,d(l)},y_{l,d(l)}^{-1}].
$$
We put $\Lambda^{G_{\text{\rm bulk}}}_0\langle\!\langle  y_{**},y_{**}^{-1}\rangle\!\rangle _0: =
\Lambda^{G_{\text{\rm bulk}}}_0\langle\!\langle y_{**},y_{**}^{-1}\rangle\!\rangle$.
\par
\item We define $N_R^{G_{\text{\rm bulk}}}(k)$ to be the set of elements
of the form
$$
\sum_{l=1}^K\sum_{s=1}^{d(l)}\left(\sum_{j=k}^{\infty} c_{l,s,j}
T^{S_l+\lambda_j^b}\right) \text{\bf e}^*_{l,s}
$$
from $N_R \otimes_R \Lambda_0$ with $c_{l,s,j} \in R$.
\end{enumerate}
\end{defn}

\begin{lem}\label{cor:gapping}
If $\mathfrak b$ satisfies Condition $\ref{gapcondb}$, then
$\mathfrak{PO}^u_{\mathfrak b}\in \Lambda^{G_{\text{\rm bulk}}}_0\langle\!\langle  y_{**},y_{**}^{-1}\rangle\!\rangle $.
\end{lem}
\begin{proof}
It is easy to see $\mathfrak{PO}^u_0 \in \Lambda^{G_{\text{\rm bulk}}}_0\langle\!\langle y_{**},y_{**}^{-1}\rangle\!\rangle$
from the definitions of $\mathfrak{PO}^u_0$ and $G_{\text{\rm bulk}}$.

So it suffices to show that the right hand side of (\ref{eq:weakPO}) in Theorem
\ref{weakpotential} lies in $\Lambda^{G_{\text{\rm bulk}}}_0\langle\!\langle y_{**},y_{**}^{-1}\rangle\!\rangle$.
We consider a term $c_{\sigma}y^{\vec
v'_{\sigma}}T^{\ell'_{\sigma}(u) + \rho_{\sigma}}$ thereof. Let
$\vec v'_{\sigma} = \sum_{i=1}^m e^i_{\sigma}\vec v_i$ as in
(\ref{edefform}). We put
$$
l_0 = \sup \{ l \mid \exists r\,\, e^{i(l,r)}_{\sigma} \ne 0\}.
$$
Then
$$
c_{\sigma}y^{\vec v'_{\sigma}} \in R[y_{1,1},y_{1,1}^{-1},\ldots,y_{l_0,d(l_0)},y_{l_0,d(l_0)}^{-1}].
$$
On the other hand
$$
\ell'_{\sigma}(u) = \sum e^{i}_{\sigma}\ell_i(u) \ge
\ell_{i(l_0,r)}(u) = S_{l_0}
$$
because $e_\sigma^i \geq 0$ and $\sum_i e_\sigma^i > 0$ and
$e_\sigma^{i(l_0,r)} \neq 0$ for some $r$. Therefore
$$
\ell'_{\sigma}(u) - S_{l_0} \in G_{\text{\rm bulk}}
$$
and so $\ell'_{\sigma}(u) + \rho_{\sigma} = S_{l_0} +
(\ell'_{\sigma}(u) - S_{l_0}) + \rho_\sigma \in G_{\text{\rm bulk}}$.
This implies
$$
c_{\sigma}y^{\vec v'_{\sigma}}T^{\ell'_{\sigma}(u) + \rho_{\sigma}}
\in \Lambda^{G_{\text{\rm bulk}}}_0\langle\!\langle y_{**},y_{**}^{-1}\rangle\!\rangle
$$
as required.
\end{proof}
We now state the following lemma:
\begin{lem}\label{lem418}
If a formal power series $\mathfrak P$ lies in $\Lambda^{G_{\text{\rm
bulk}}}_0\langle\!\langle y_{**},y_{**}^{-1}\rangle\!\rangle_k$ for some $k \in \Z_{\geq 0}$, so
does  $\frac{\partial\mathfrak P}{\partial y_{l',s}}$ for the same $k$
and so
\begin{equation}\label{inMR}
\sum_{l'=1}^K\sum_{s=1}^{d(l')} \mathfrak c_{l',s} \frac{\partial \mathfrak
P}{\partial y_{l',s}}(\mathfrak c) \text{\bf e}^*_{l',s} \in
N_R^{G_{\text{\rm bulk}}}(k)
\end{equation}
for $\mathfrak c = (\mathfrak c_{1,1},\cdots,\mathfrak c_{K,d(K)}) \in (R
\setminus \{0\})^n$.
\end{lem}
\begin{proof} By the form \eqref{LamGbulk}
of the elements from $\Lambda^{G_{\text{\rm
bulk}}}_0\langle\!\langle y_{**},y_{**}^{-1}\rangle\!\rangle_k$, the first statement immediately
follows.
Then the last statement follows from the definition of $N_R^{G_{\text{\rm
bulk}}}(k)$.
\end{proof}
\begin{prop}\label{induclema}
There exists a sequence
\begin{equation}\label{formfrakbk}
\mathfrak b(k) = \sum_{l=1}^{K}\sum_{r=1}^{a(l)} \mathfrak b_{l,r}(k) D_{i(l,r)}
\end{equation}\label{inducconcl}
for $k = 0, \ldots, $
that satisfies Condition $\ref{gapcondb}$ and
\begin{equation}\label{b(k)jyoken}
\sum_{l'=1}^K\sum_{s=1}^{d(l')} \mathfrak y_{l',s}\frac{\partial
\mathfrak{PO}^u_{\mathfrak b(k)}}{\partial y_{l',s}}(\mathfrak y)
\text{\bf e}^{*}_{l',s} \in N_R^{G_{\text{\rm bulk}}}(k)
\end{equation}
and
\begin{equation}\label{congruentK}
\mathfrak b(k+1) - \mathfrak b(k)  \equiv 0 \mod T^{\lambda^b_k}\Lambda_0.
\end{equation}
\end{prop}
\begin{proof} We prove this by induction over $k$.
The case $k=0$ follows from Lemma \ref{cor:gapping} for $\mathfrak b(0) = \bf 0$.
\par
Suppose we have found $\mathfrak b(k)$ as in the proposition. Then we
have
$$
\sum_{l'=1}^K\sum_{s=1}^{d(l')} \mathfrak y_{l',s}\frac{\partial
\mathfrak{PO}^u_{\mathfrak b(k)}}{\partial y_{l',s}}(\mathfrak y)
\text{\bf e}^{*}_{l',s} \equiv \sum_{l=1}^K\sum_{s=1}^{d(l)}
c_{l,s,k} T^{S_l+\lambda_k^b}\text{\bf e}^{*}_{l,s} \mod
N_R^{G_{\text{\rm bulk}}}(k+1)
$$
for some $c_{l,s,k} \in R$.
Since $\{\vec v_{i(l',r)}\mid l'\le l\}$ spans $A^{\perp}_{l}$ for
all $l \leq K$ by definition, we can find $a_{l,r,k} \in R$ such
that
$$
\sum_{s=1}^{d(l)} c_{l,s,k}\text{\bf e}^{*}_{l,s}
- \sum_{r=1}^{a(l)}  a_{l,r,k} \vec v_{i(l,r)} \in A^{\perp}_{l-1}.
$$
Therefore by definition of $N_R^{G_{\text{\rm bulk}}}(k)$ we have
$$
\sum_{l=1}^K\sum_{s=1}^{d(l)} c_{l,s,k} T^{S_l+\lambda_k^b}\text{\bf
e}^{*}_{l,s} - \sum_{l=1}^K\sum_{r=1}^{a(l)} a_{l,r,k}
T^{S_l+\lambda_k^b}\vec v_{i(l,r)} \in N_R^{G_{\text{\rm bulk}}}(k+1).
$$
Thus
\begin{equation}\label{ubdyctuib}
\aligned
&\sum_{l'=1}^K\sum_{s=1}^{d(l')}
\mathfrak y_{l',s}\frac{\partial \mathfrak{PO}^u_{\mathfrak b(k)}}{\partial y_{l',s}}(\mathfrak y)
\text{\bf e}^{*}_{l',s} \\
&\equiv \sum_{l=1}^K\sum_{r=1}^{a(l)} a_{l,r,k}
T^{S_l+\lambda_k^b}\vec v_{i(l,r)} \mod N_R^{G_{\text{\rm bulk}}}(k+1).
\endaligned\end{equation}
We now put
$$
\mathfrak b_{l,r}(k+1) = \mathfrak b_{l,r}(k) - T^{\lambda^b_k}a_{l,r,k}(\mathfrak y^{\vec v_{i(l,r)}})^{-1}.
$$
\begin{lem}\label{2ndthirest} Let $\mathfrak b(k)$ be as in the
induction hypothesis above. If $\lambda = \lambda_k^b$, then the
second and the third terms of $(\ref{firstapprox})$ are contained in
$\Lambda^{G_{\text{\rm bulk}}}_0\langle\!\langle y_{**},y_{**}^{-1}\rangle\!\rangle_{k+1}$.
\end{lem}
\begin{proof}
We first consider, for $h\ge 2$,
\begin{equation}\label{secondterm20}
c_h T^{h\lambda_k^b + \ell_{i(l,r)}(u)} y^{\vec v_{i(l,r)}}
\end{equation}
which is in the second term of $(\ref{firstapprox})$. We
remark that
$$
c_h y^{\vec v_{i(l,r)}} \in
R[y_{1,1},y_{1,1}^{-1},\ldots,y_{l,d(l)},y_{l,d(l)}^{-1}].
$$
On the other hand,
$$
h\lambda_k^b + \ell_{i(l,r)}(u) - S_l = h\lambda_k^b
$$
is contained in $G_{\text{\rm bulk}}$ and so must be equal to $\lambda_{k'}^b$ for
some $k'> k$ since $h \geq 2$. Therefore (\ref{secondterm20}) is
contained in $\Lambda^{G_{\text{\rm bulk}}}_0\langle\!\langle y_{**},y_{**}^{-1}\rangle\!\rangle_{k+1}$, as
required.
\par
We next consider
\begin{equation}\label{thirdterm20}
c_{h,\sigma} T^{h\lambda_k^b + \ell'_{\sigma}(u) +
\rho_{\sigma}}y^{\vec v_{\sigma}}
\end{equation}
which is in the third term of $(\ref{firstapprox})$. ($h\ge 1$.) We
have $\vec v_{\sigma} = \sum e_{\sigma}^i\vec v_i$, $\ell'_{\sigma}
= \sum e_{\sigma}^i\ell_i$. We put
$$
l_0 = \sup\{ l \mid \exists r \,\, e_{\sigma}^{i(l,r)} \ne 0\}.
$$
Then
$$
c_{h,\sigma} y^{\vec v_{\sigma}} \in
R[y_{1,1},y_{1,1}^{-1},\ldots,y_{l_0,d(l_0)},y_{l_0,d(l_0)}^{-1}].
$$
On the other hand, since
$$
\ell'_{\sigma}(u) = \sum_i e_{\sigma}^i\ell_i(u) \ge
\ell_{i(l_0,r)}(u) = S_{l_0},
$$
it follows that
$$
\ell'_{\sigma}(u) + \rho_{\sigma} - S_{l_0} \in G_{\text{\rm bulk}} \setminus \{0\}.
$$
Therefore since $h \geq 2$ (in fact $h \geq 1$ is enough for this), we have
$$
h\lambda^b_k + \ell'_{\sigma}(u) + \rho_{\sigma} - S_{l_0} >
\lambda^b_{k}
$$
and so equal to $\lambda^b_{k'}$ for some $k' > k$. Therefore we find  that 
(\ref{thirdterm20}) is contained in 
$\Lambda^{G_{\text{\rm bulk}}}_0\langle\!\langle y_{**},y_{**}^{-1}\rangle\!\rangle_{k+1}$, as required.
\par
The proof of Lemma \ref{2ndthirest} is complete.
\end{proof}

Then Proposition \ref{nonfanopert}, (\ref{ubdyctuib}), Lemma
\ref{lem418} and Lemma \ref{2ndthirest} imply that
(\ref{b(k)jyoken}) is satisfied for $k+1$. The proof of Proposition
\ref{induclema} is complete.
\end{proof}

Now we are ready to complete the proof of Theorem \ref{ellim}. By
(\ref{congruentK})
$$
\lim_{k\to\infty} \mathfrak b(k) = \mathfrak b
$$
converges. Then (\ref{b(k)jyoken}) implies
$$
\mathfrak y_{l,s}\frac{\partial \mathfrak{PO}^u_0}{\partial y_{l,s}}(\mathfrak y)
= 0,
$$
as required.
\end{proof}

\section{Two points blow up of $\C P^2$: an example}
\label{sec:Exam}

Our main example is the two-points blow up $X_2$ of $\C P^2$. We take
its K\"ahler form $\omega_{\alpha,\beta}$ such that
the moment polytope is
\begin{equation}\label{momentpoly}
P_{\alpha,\beta} = \{ (u_1,u_2) \mid 0\le u_1\le 1, ~0 \le u_2 \le 1-\alpha,~
\beta \le u_1+u_2 \le 1\}.
\end{equation}
Here
\begin{equation}\label{Kcone}
(\alpha,\beta) \in \Delta = \{ (\alpha,\beta) \mid 0 \le \alpha,\beta, \quad
\alpha + \beta \le 1\}.
\end{equation}
We remark that $\R_+ \Delta$ is the K\"ahler cone of $X_2$.
\par
In Example 10.17 \cite{fooo08} we studied this example in the case
\begin{equation}\label{cond}
\beta = \frac{1-\alpha}{2}, \quad \frac{1}{3} < \alpha.
\end{equation}
We continue the study this time involving bulk deformations.
\par
We consider the point
\begin{equation}\label{ucond}
\text{\bf u} = \left(u,\beta\right), \qquad u \in
\left(\beta,\frac{1-\beta}{2}\right) = \left(\frac{1-\alpha}{2},\frac{1+\alpha}{4}\right)
\end{equation}
and compute
\begin{equation}\label{POwithoutbulk}
\aligned
\mathfrak{PO}^{\text{\bf u}}(0;y_1,y_2) = T^{\beta}(y_2+y_2^{-1})
+ T^{u} (y_1+y_1y_2) + T^{1-\beta-u}y_1^{-1}y_2^{-1}.
\endaligned
\end{equation}
We note that \eqref{ucond} implies
\begin{equation}\label{beta<u<}
\beta < u < 1-\beta-u.
\end{equation}
Therefore the leading term equation is
\begin{equation}\label{LTEblowup}
1 - y_2^{-2} = 0, \qquad
1+ y_2 = 0.
\end{equation}
Namely $(y_1,-1)$ is its solution for any $y_1$. Therefore
Theorem \ref{ellim} implies:

\begin{prop}\label{2ptsblowup}
$L(\text{\bf u}) \subset (X_2,\omega_{\alpha,\beta})$ is bulk-balanced
if $(\ref{cond})$ and $(\ref{ucond})$ are satisfied.
\end{prop}
Theorem \ref{thm:bupP2} for $k=2$ will then follow from Proposition
\ref{prof:bal2}.
\begin{proof}[Proof of Theorem \ref{thm:bupP2}]
The case $k=2$ (the two points blow up) is already proved.
We consider $k=3$. We blow up $(X_2,\omega_{\alpha,\beta})$
at the fixed point corresponding to $(1,0) \in P_{\alpha,\beta}$.
Then we have a toric K\"ahler structure on $X_3$ whose
moment polytope is
$$
\{ (u_1,u_2)\in P_{\alpha,\beta} \mid u_1 \le 1-\epsilon \}.
$$
We have
\begin{equation}\label{POwithbulk3}
\aligned
\mathfrak{PO}^{\text{\bf u}}_0(y_1,y_2) =
T^{\beta}(y_2+y_2^{-1}) &+ T^{u} (y_1+y_1y_2) \\
&+ T^{1-\beta-u}y_1^{-1}y_2^{-1} + T^{1-\epsilon-u}y_1^{-1}.
\endaligned
\end{equation}
We remark that
$$
1-\beta-u < 1-\epsilon-u
$$
if $\epsilon$ is sufficiently small. Therefore
the leading term equation at (\ref{ucond}) is again (\ref{LTEblowup}).
Therefore we can
apply Theorem \ref{ellim} to show that
all $L(\text{\bf u})$ satisfying (\ref{ucond}) are bulk-balanced.
Thus Theorem \ref{thm:bupP2} is proved for $k=3$.
\par
We can blow up again at the fixed point corresponding to $(1-\epsilon,0)$.
We can then prove the case $k=4$. (We remark that this time our
toric manifold is not Fano. We never used the property $X$ to be Fano
in the above discussion.) We can repeat this process a arbitrarily many times to complete
the proof of Theorem \ref{thm:bupP2}.
\end{proof}

Below we will
examine the effect of bulk deformations more explicitly for the example of two
points blow up. 
We consider the divisor
$$
D_1 = \pi^{-1}(\{(u_1,u_2) \in P \mid u_2 = 0\})
$$
and let
\begin{equation}\label{defbulk}
\mathfrak b_{w,\kappa} = wT^{\kappa} [D_1] \in \mathcal A^2(\Lambda_+), \quad w \in R \setminus \{0\}, ~\kappa >0.
\end{equation} 
During the calculation below, we fix 
$\mathfrak b_{w,\kappa}$.
By Proposition \ref{POcalcFano}, we have:
\begin{equation}\label{POwithbulk}
\aligned
\mathfrak{PO}^{\text{\bf u}}(\mathfrak b_{w,\kappa};y_1,y_2) =
&T^{\beta}(\exp(wT^{\kappa})y_2+y_2^{-1}) \\
&+ T^{u} (y_1+y_1y_2) + T^{1-\beta-u}y_1^{-1}y_2^{-1}.
\endaligned
\end{equation}
We study the equation
\begin{equation}\label{JacPOwithbulk}
\frac{\partial \mathfrak{PO}^{\text{\bf u}}}{\partial y_1}(\mathfrak b_{w,\kappa};y_1,y_2)
= \frac{\partial \mathfrak{PO}^{\text{\bf u}}}{\partial y_2}(\mathfrak b_{w,\kappa};y_1,y_2) = 0.
\end{equation}

We put $y_2 = -1 + c T^{\mu}$, $y_1 = d$, with $c,d \in \Lambda_0
\setminus \Lambda_+$. Taking the inequality \eqref{beta<u<} into
account, we obtain
\begin{equation}
\aligned c T^{\mu} + d^{-2}T^{1-\beta-2u} &\equiv 0 \mod
T^{\max\{\mu,1-\beta-2u\}}\\
-2c T^{\mu} + wT^{\kappa} + d
T^{u-\beta} &\equiv 0 \mod T^{\max\{\mu,\kappa,u-\beta\}}.
\endaligned
\end{equation}
In the following calculation we take $\overline c, \overline d \in R$ such that
$c \equiv \overline c$, $d \equiv \overline  d$ $\mod \Lambda_+$.
\par\smallskip
\noindent(Case 1)
$\mu = \kappa < u-\beta$.
\par
We have $\overline c=w/2$, $\mu = 1-\beta-2u$. $\overline d = \pm\sqrt{-2/w}$.
$u = (1-\beta)/2 - \kappa/2 = (1+\alpha)/4 - \kappa/2$. It implies
$1/3 < u < (1+\alpha)/4$. The equation for $(\overline c,\overline d)$ has 2 solutions.
They are both simple. Hence in the same way as the proof of
Theorem 10.4 \cite{fooo08} (the strongly non-degenerate case)
we can show that these two solutions correspond to the
solutions of (\ref{JacPOwithbulk}).
\par\medskip
\noindent(Case 2)
$\mu =  u-\beta < \kappa$.
\par
We have $\overline d=2\overline c$, $1-\beta-2u = \mu$.
Hence $u =1/3$. We can show that there are 3 solutions of (\ref{JacPOwithbulk})
in the same way.
\par\medskip
\noindent(Case 3)
$\kappa = u-\beta < \mu$.
\par
We have $\overline d=-w$. Then $\mu = 1-\beta-2u$. $\overline c=-w^{-2}$. Hence $u <
1/3$. We can show that there is 1 solution of (\ref{JacPOwithbulk})
in the same way.
\par\medskip
\noindent(Case 4)
$\kappa = u-\beta = \mu$.
\par
We have $-2\overline c + w + \overline d = 0$ and $1-\beta-2u = \mu$. Hence $u =1/3$.
$\kappa = \alpha/2 - 1/6$.
\begin{equation}\label{solte}
\overline d^2(\overline d+w) + 2 =0.
\end{equation}
This has three simple roots unless
\begin{equation}\label{descri}
\frac{4}{27}w^3 + 2 = 0.
\end{equation}
\par\medskip
When $\kappa$ is small Case 1 and Case 3 occur. There are
two fibers with nontrivial Floer cohomology
(on $(\ref{ucond})$), that is
($(\beta+\kappa,\beta)$ and $((1+\alpha)/4 - \kappa/2,\beta)$). They move
from $(\beta,\beta)$, $(1+\alpha)/4,\beta)$ to $(1/3,\beta)$.
Then, when $\kappa = \alpha/2 - 1/6$, Case 4 occurs.
If $\kappa > \alpha/2 - 1/6$ then Case 2 occurs and
bulk deformation does not change the `secondary' leading term equation
(\ref{solte}).
\par
It might be interesting to observe that it actually occurs that the
`secondary' leading term equation (\ref{solte}) has multiple roots.
That is the case where (\ref{descri}) is satisfied. (We remark that
the example where there is a multiple root for the leading term
equation was found in \cite{osty}.)
\par
\par\medskip
\section{Operator $\mathfrak q$ in the toric case}
\label{sec:operatorq} In this section and the next, we study the
moduli space of holomorphic discs and its effects on the operator
$\mathfrak q$ and on the potential function $\mathfrak{PO}^u(\mathfrak
b;y_1,\ldots,y_n)$.
\par
Let $u \in \text{\rm Int}\,P$ and $\beta \in H_2(X,L(u);\Z)$. We
denote by $\mathcal M_{k+1;\ell}^{\text{\rm main}}(L(u),\beta)$ the
moduli space of stable maps from bordered Riemann surfaces of genus
zero with $k+1$ boundary marked points and $\ell$ interior marked
points, in homology class $\beta$. (See Section 3 \cite{fooo00} 
= Subsection 2.1.2 \cite{fooo06}.) We
require the boundary marked points to respect the cyclic order of
$S^1 =
\partial D^2$. (In other words, we consider the \emph{main
component} in the sense of Section 3 \cite{fooo00}.)) We assume
$k\ge 0$. Then $\mathcal M_{k+1;\ell}^{\text{\rm main}}(L(u),\beta)$
is compact. (See Subsections 3.8.2 and 7.4.1 \cite{fooo06} = Sections 13.2 and 32.1 \cite{fooo06pre}, for the
reason why we need to assume $k\ge 0$ for compactness.)
\par
We denote an element of $\mathcal M^{\text{\rm main}}_{k+1;\ell}(L(u),\beta)$ by
$$
(\Sigma,\varphi,\{z^+_i  \mid 1 = 1, \ldots, \ell\}, \{z_i \mid
i=0,1,\ldots,k\})
$$
where $\Sigma$ is a connected genus zero bordered semi-stable curve,
$\varphi: (\Sigma,\partial \Sigma) \to (X,L(u))$ is a holomorphic
map and $z^+_i \in \text{Int} \Sigma$ and $z_i \in
\partial \Sigma$.
Let $\mathcal M^{\text{\rm main},\text{\rm
reg}}_{k+1;\ell}(L(u),\beta)$ be its subset consisting of all maps
from a \emph{smooth} disc. (Namely the stable map without disc or
sphere bubbles.)
\par
We have the following proposition.
Let
$\beta_i \in H_2(X,L(u);\Z)$ $(i=1,\ldots,m)$ be the
classes with
$\mu(\beta_i) = 2$ and
$$
\beta_i \cap D_j =
\begin{cases}
1  &\text{$i=j$,} \\
0  &\text{$i\ne j$}.
\end{cases}
$$
We recall from \cite{cho-oh} that the spin structure of $L(u)$
induced from the torus $T^n = \R^n/\Z^n$ as its orbit is called the
\emph{standard} spin structure.
\begin{prop}\label{copy37.177}
\begin{enumerate}
\item If $\mu(\beta) < 0$, or $\mu(\beta) = 0$, $\beta \ne 0$, then the moduli space $\mathcal M^{\text{\rm main},\text{\rm
reg}}_{k+1;\ell}(L(u),\beta)$  is empty.
\par
\item If $\mu(\beta) = 2$, $\beta \ne \beta_1,\cdots,\beta_m$, then $\mathcal M^{{\text{\rm main}},\text{\rm
reg}}_{k+1;\ell}(L(u),\beta)$ is also empty.
\par
\item  For $i=1,\cdots,m$, we have
\begin{equation}\label{178}
\aligned
\mathcal M^{{\text{\rm main}},\text{\rm reg}}_{1;0}(L(u),\beta_i) &=
\mathcal M^{\text{\rm main}}_{1;0}(L(u),\beta_i), \\
\mathcal M^{{\text{\rm main}}}_{1;\ell}(L(u),\beta_i) &= \mathcal
M^{\text{\rm main},\text{\rm reg}}_{1;0}(L(u),\beta_i) \times
\overline{\text{\rm Conf}}(\ell;D^2).
\endaligned
\end{equation}
Here $\overline{\text{\rm Conf}}(\ell;D^2)$ is a compactification of
the configuration space:
$$
\{ (z_1^+,\ldots,z_{\ell}^+) \mid  z_i^+ \in \text{\rm Int} D^2,
z_i^+ \ne z_j^+ \,\,\,\text{for $i\ne j$}\}=:\text{\rm Conf}(\ell;D^2).
$$
(See Remark $\ref{configrem}$.)
Moreover $\mathcal M_{1;0}^{\text{\rm main}}(L(u),\beta_i)$ is Fredholm regular.
Furthermore the evaluation map
$$
ev: \mathcal M^{\text{\rm main}}_{1;0}(L(u),\beta_i) \to  L(u)
$$
is an orientation preserving diffeomorphism if we equip $L(u)$ with
the standard spin structure.
\par
\item For any $\beta$, the moduli space $\mathcal M^{{\text{\rm main}},\text{\rm reg}}_{1;\ell}(L(u),\beta)$ is
Fredholm regular. Moreover
$$
ev: \mathcal M^{\text{\rm main,reg}}_{1;\ell}(L(u),\beta) \to  L(u)
$$
is a submersion.
\par
\item  If $\mathcal M^{\text{\rm main}}_{1;\ell}(L(u),\beta)$ is not empty, then
there exist $k_i \in \Z_{\ge 0}$ and $\alpha_j \in H_2(X;\Z)$ such that
$$
\beta = \sum_i k_i\beta_i + \sum_j \alpha_j
$$
and $\alpha_j$ is realized by a holomorphic sphere.
There is at least one nonzero $k_i$.
\end{enumerate}
\end{prop}
\begin{rem}\label{configrem}
We define the compactification of $\text{\rm Conf}(\ell;D^2)$ as
follows. We consider $X = \C$, $L = S^1$. 
Let $\beta_1$ be the
generator of $H_2(X;L)$ which is represented by a holomorphic disc. 
We consider the moduli space 
$\mathcal M_{1;\ell}^{\text{main}}(L; \beta_1)$ and the evaluation map at the boundary marked point 
$ev_0 :  \mathcal M_{1;\ell}^{\text{main}}(L; \beta_1) \to L=S^1$. 
We fix a point $p_0 \in L=S^1 \subset \C$ and 
put $\overline{\text{\rm Conf}}(\ell;D^2) = ev_0^{-1}(p_0)$. 
We use $\overline{\text{\rm Conf}}(\ell;D^2)$ as a compactification of $\text{\rm Conf}(\ell;D^2)$. 
\end{rem}
Proposition \ref{copy37.177} follows easily from Theorem 11.1
\cite{fooo08} which in turn follows from \cite{cho-oh} as we
explained in Section 11 \cite{fooo08}.
\par
We next discuss Kuranishi structure of $\mathcal M^{\text{\rm
main}}_{k+1;\ell}(L(u),\beta)$. In Sections 17,18 \cite{fooo00}
or in Section 7.1 \cite{fooo06} (=Section 29 \cite{fooo06pre}), we equipped
$\mathcal M^{\text{\rm main}}_{k+1;\ell}(L(u),\beta)$ with a Kuranishi structure.
In our toric
case, this structure can be chosen to be $T^n$ equivariant in the
following sense: Let $(V,E,\Gamma,\psi,s)$ be a Kuranishi chart (see
Section 5 \cite{FO} and Section A1 \cite{fooo06} = Section A1 \cite{fooo06pre}). 
Here $V$ is a smooth manifold with an action of a finite group
$\Gamma$, $E \to V$ is a $\Gamma$ equivariant vector bundle, $s$ its
$\Gamma$-equivariant section and $\psi: s^{-1}(0)/\Gamma \to
\mathcal M^{\text{\rm main}}_{k+1;\ell}(L(u),\beta)$ is a
homeomorphism onto an open set. Then  we have a $T^n$ action on
$V,E$ which commutes with $\Gamma$ action, such that $s$ is $T^n$
equivariant. Moreover $\psi$ is $T^n$ equivariant. Here the $T^n$
action is induced by one on $X$. (We recall that $L(u)$ is $T^n$
invariant.) The construction of such Kuranishi structure is in Appendix 2 \cite{fooo08}.
\par
Let $\{ D_a \mid a = 1,\cdots, B\}$ be the basis of $\mathcal
A(\Z)$. (Each $D_a$ corresponds to a face of $P$.) We note that each
of $D_a$ is a $T^n$ invariant submanifold. Let
$$
ev_i^{\text{\rm int}}: \mathcal M^{\text{\rm main}}_{k+1;\ell}(L(u),\beta) \to X
$$
be the evaluation map at the $i$-th interior marked point.
($i=1,\ldots,\ell$.)
Namely
$$
ev_i^{\text{\rm int}}((\Sigma,\varphi,\{z^+_i\}, \{z_i\})) = \varphi(z^+_i).
$$
We put $\underline B = \{1,\ldots,B\}$ and denote the set of all maps $\text{\bf p}:
\{1,\ldots,\ell\} \to \underline B$ by
$Map(\ell,\underline B)$. We write $\vert\text{\bf
p}\vert = \ell$ if $\text{\bf p} \in Map(\ell,\underline B)$.
\par
We define a fiber product
\begin{equation}
\mathcal M^{\text{\rm main}}_{k+1;\ell}(L(u),\beta;\text{\bf p})
=
\mathcal M^{\text{\rm main}}_{k+1;\ell}(L(u),\beta)
{}_{(ev_1^{\text{int}},\ldots,ev_{\ell}^{\text{int}})}
\times_{X^{\ell}} \prod_{i=1}^{\ell} D_{\text{\bf p}(i)}.
\end{equation}
Here the right hand side is the set of all
$((\Sigma,\varphi,\{z^+_i\}, \{z_i\}),(p_1,\ldots,p_{\ell}))$ such
that $(\Sigma,\varphi,\{z^+_i\}, \{z_i\}) \in \mathcal M^{\text{\rm
main}}_{k+1;\ell}(L(u),\beta)$, $p_i \in D_{\text{\bf p}(i)}$, and
that $ \varphi(z^+_i) = p_i $.
\par
We define
$$
ev_i:
\mathcal M^{\text{\rm main}}_{k+1;\ell}(L(u),\beta)
\to L(u)
$$
by
$$
ev_i((\Sigma,\varphi,\{z^+_i\}, \{z_i\})) = \varphi(z_i).
$$
It induces
$$
ev_i:
\mathcal M^{\text{\rm main}}_{k+1;\ell}(L(u),\beta;\text{\bf p})
\to L(u)
$$
in an obvious way.
\begin{lem}\label{kstructure}
$\mathcal M^{\text{\rm main}}_{k+1;\ell}(L(u),\beta;\text{\bf p})$
has a Kuranishi structure such that each Kuranishi chart is
$T^n$-equivariant and the coordinate change preserves the $T^n$
action. Moreover the evaluation map
$$
ev = (ev_0,ev_1,\ldots,ev_k):
\mathcal M^{\text{\rm main}}_{k+1;\ell}(L(u),\beta;\text{\bf p})
\to L(u)^{k+1}
$$
is weakly submersive and $T^n$-equivariant. Our Kuranishi structure
has a tangent bundle and is oriented.
\end{lem}
\begin{proof}
The fiber product of Kuranishi structures is defined in Section A1.2
\cite{fooo06} (= Section A1.2 \cite{fooo06pre}). 
Since the maps we used here to define the fiber
product are all $T^n$-equivariant it follows that the Kuranishi
structure on the fiber product is $T^n$-equivariant. The
orientability is proved in Chapter 8 \cite{fooo06} (= Chapter 9 \cite{fooo06pre}). The fact that
$ev$ is well defined and is weakly submersive is proved in  
Section 7.1 \cite{fooo06} (=Section
29 \cite{fooo06pre}) also.
\end{proof}

We next describe the boundary of our Kuranishi structure. For the
description, we need to prepare some notations. We denote the set of
shuffles of $\ell$ elements by
\begin{equation}\label{shuff}
\text{\rm Shuff}(\ell) = \{ (\mathbb L_1,\mathbb L_2) \mid \mathbb
L_1 \cup \mathbb L_2 = \{1,\ldots,\ell\}, \,\,\mathbb L_1 \cap
\mathbb L_2 = \emptyset \}.
\end{equation}
We will define a map
\begin{equation}\label{split}
\text{Split}: \text{\rm Shuff}(\ell) \times Map(\ell,\underline B)
\longrightarrow \bigcup_{\ell_1+\ell_2 = \ell} Map(\ell_1,\underline
B) \times Map(\ell_2,\underline B),
\end{equation}
as follows: Let $\text{\bf p} \in Map(\ell,\underline B)$ and
$(\mathbb L_1,\mathbb L_2) \in \text{\rm Shuff}(\ell)$. We put
$\ell_j = \# (\mathbb L_j)$ and let $\mathfrak i_j: \{1,\ldots,\ell_j\}
\cong \mathbb L_j$ be the order preserving bijection. We consider
the map $\text{\bf p}_j: \{1,\ldots,\ell_j\} \to \underline B$
defined by $ \text{\bf p}_j(i) = \text{\bf p}(\mathfrak i_j(i)) $, and
set
$$
\text{Split}((\mathbb L_1,\mathbb L_2),\text{\bf p}): = (\text{\bf
p}_1,\text{\bf p}_2).
$$
\par
We now define a gluing map 
\begin{equation}\label{glueing}
\aligned \text{Glue}^{(\mathbb L_1,\mathbb L_2),\text{\bf
p}}_{\ell_1,\ell_2;k_1,k_2;i;\beta_1,\beta_2}: &\mathcal
M^{\text{\rm main}}_{k_1+1;\ell_1}(L(u),\beta_1;\text{\bf p}_1)
{}_{ev_0}\times_{ev_i} \mathcal M^{\text{\rm
main}}_{k_2+1;\ell_2}(L(u),\beta_2;\text{\bf p}_2)
\\
&\to
\mathcal M^{\text{\rm main}}_{k+1;\ell}(L(u),\beta;\text{\bf p})
\endaligned
\end{equation}
associated to below.
Here $k = k_1 + k_2 -1$, $\ell = \ell_1 + \ell_2$,
$\beta = \beta_1 + \beta_2$, and $i = 1,\ldots,k_2$.
Let
$$
\mathbb S_j = ((\Sigma_{(j)},\varphi_{(j)},\{z^+_{i,{(j)}}\},
\{z_{i,{(j)}}\}) \in \mathcal M^{\text{\rm
main}}_{k_j+1;\ell_j}(L(u),\beta_j;\text{\bf p}_j)
$$
$j = 1,2$. We glue $z_{0,(1)} \in \partial \Sigma_1$ with $z_{i,(2)}
\in
\partial \Sigma_2$ to obtain
$$
\Sigma = \Sigma_1 \#_{i} \Sigma_2.
$$
Suppose $(\mathbb S_1,\mathbb S_2)$ is an element of the fiber
product in the left hand side of (\ref{glueing}). Namely we assume
$$
\varphi_{(1)}(z_{0,(1)}) = \varphi_{(2)}(z_{i,(2)}).
$$
This defines a holomorphic map
$$
\varphi = \varphi_{(1)} \#_i \varphi_{(2)}
: \Sigma \to X
$$
by putting $\varphi = \varphi_{(j)}$ on $\Sigma_j$.
\par
We define the $m$-th interior marked point
$z_m^{+}$ of $\varphi$ as follows: 
If $m \in \mathbb L_j$, we have 
$c \in \{1,\ldots,\ell_j\}$ such that $\mathfrak i_j(c) = m$ 
where $\mathfrak i_j:
\{1,\ldots,\ell_j\} \to \mathbb L_j$ is the order preserving
bijection. Then we put $z_m^+=z_{c;(j)}^+  \in
\Sigma_j \subset \Sigma$.
\par
We define the boundary marked points
$(z_0,z_1,\ldots,z_k)$ by
$$
(z_0,z_1,\ldots,z_k)
= (z_{0,(2)},\ldots,z_{i-1,(2)},z_{1,(1)},\ldots,
z_{k_1,(1)},z_{i+1,(2)},\ldots,z_{k_2,(2)}).
$$
Now we put
$$
\mathbb S = ((\Sigma,\varphi,\{z^+_{i}\}, \{z_{i}\})
$$
and
$$
\text{Glue}^{(\mathbb L_1,\mathbb L_2),\text{\bf
p}}_{\ell_1,\ell_2;k_1,k_2;i;\beta_1,\beta_2} (\mathbb S_1,\mathbb
S_2) = \mathbb S.
$$
\begin{lem}\label{bdry}
The boundary of $\mathcal M^{\text{\rm
main}}_{k+1;\ell}(L(u),\beta;\text{\bf p})$ is isomorphic to the
union of  the images of $\text{\rm Glue}^{(\mathbb L_1,\mathbb
L_2),\text{\bf p}}_{\ell_1,\ell_2;k_1,k_2;i;\beta_1,\beta_2}$ for $k
= k_1 + k_2 -1$, $\ell = \ell_1 + \ell_2$, $\beta = \beta_1 +
\beta_2$, and $i = 1,\cdots,k_2$ as a space with Kuranishi
structure. The isomorphism preserves the $T^n$ action.
\par
The isomorphism commutes with the evaluation maps at the boundary marked
points.
\end{lem}
The lemma directly follows from our construction of the Kuranishi
structure we gave in Section 7.1 \cite{fooo06} (= Section 29 \cite{fooo06pre}).
\par
Let $\mathfrak S_{\ell}$ be the symmetric group of $\ell$ elements. It acts
on $\mathcal M^{\text{\rm main}}_{k+1;\ell}(L(u),\beta)$ by changing
the indices of interior marked points.  It also acts on
$Map(\ell,\underline B)$ by $\sigma \cdot \text{\bf p} = \text{\bf
p} \circ \sigma^{-1}$. Then for $\sigma \in \mathfrak S_{\ell}$ we have
\begin{equation}\label{Saction}
\sigma_*:
\mathcal M^{\text{\rm main}}_{k+1;\ell}(L(u),\beta;\text{\bf p})
\to
\mathcal M^{\text{\rm main}}_{k+1;\ell}(L(u),\beta;\sigma\cdot\text{\bf p}).
\end{equation}
\par
We next generalize Lemma 11.2 \cite{fooo08} to our situation.
Let
\begin{equation}
\mathfrak{forget}_0:  \mathcal M^{\text{\rm
main}}_{k+1;\ell}(L(u),\beta;\text{\bf p}) \to  \mathcal
M^{\text{\rm main}}_{1;\ell}(L(u),\beta;\text{\bf p})
\label{37.183}\end{equation} be the forgetful map which forgets all
the boundary marked points except the $0$-th one. We may choose our
Kuranishi structures so that (\ref{37.183}) is compatible with
$\mathfrak{forget}_0$ in the same sense as in Lemma 7.3.8 \cite{fooo06} (= Lemma 31.8 \cite{fooo06pre}).
\begin{lem}\label{37.184}
\!For each given $E > 0$, there exists a system of
multisections $\mathfrak s_{\!\beta,k+1,\ell,\text{\bf p}}$ on
$\mathcal M^{\text{\rm main}}_{k+1;\ell}(L(u),\beta;\text{\bf p})$
for $\beta \cap \omega < E$, $\text{\bf p} \in
Map(\ell,\underline B)$. They have the following properties:
\begin{enumerate}
\item They are transversal to $0$.
\par
\item They are invariant under the $T^n$ action.
\par
\item The multisection $\mathfrak s_{\beta,k+1,\ell,\text{\bf p}}$
is the pull-back of the multisection $\mathfrak s_{\beta,1,\ell,\text{\bf p}}$
by the forgetful map $(\ref{37.183})$.
\par
\item The restriction of $\mathfrak s_{\beta,k+1,\ell,\text{\bf p}}$
to the image of $\text{\rm Glue}^{(\mathbb L_1,\mathbb
L_2),\text{\bf p}}_{\ell_1,\ell_2;k_1,k_2;i; \beta_1,\beta_2}$ is
the fiber product of the multisections $\mathfrak
s_{\beta_j,k_j+1,\ell_j,\text{\bf p}_j}$ $j=1,2$ with respect to the
identification of the boundary given in Lemma $\ref{bdry}$.
\par
\item For $\ell=0$ the multisection
$\mathfrak s_{\beta,k+1,0,\emptyset}$ coincides with one
defined in Lemma $11.2$ \cite{fooo08}.
\item The map $(\ref{Saction})$ preserves our system of multisections.
\end{enumerate}
\end{lem}
\begin{proof}
The proof is similar to the proof of Lemma $11.2$ \cite{fooo08}. We
define $\mathfrak s_{\beta,k+1,\ell,\text{\bf p}}$ for $\text{\bf p}
\in Map(\ell,\underline B)$ by an induction over $\omega \cap \beta$. The case $\ell=0$ is proved in Lemma $11.2$
\cite{fooo08}. Condition (4) above determines the multisection on
the boundary. $T^n$ equivariance implies that $ ev_0: \mathcal
M^{\text{\rm main,reg}}_{k+1;\ell}(L(u),\beta;\text{\bf p})
^{\mathfrak s_{\beta,k+1,\ell,\text{\bf p}}} \to L(u) $ is a
submersion. Here
$$
\mathcal M^{\text{\rm main,reg}}_{k+1;\ell}(L(u),\beta;\text{\bf p})
^{\mathfrak s_{\beta,k+1,\ell,\text{\bf p}}}
=(\mathfrak s_{\beta,k+1,\ell,\text{\bf p}})^{-1}(0).
$$
This fact and the induction hypothesis imply that the multisection
we defined by (4) on the boundary of our moduli space is
automatically transversal. (This is the important point that makes
the proof of Lemma \ref{37.184} easier than corresponding general
discussion given in Section 7.1 \cite{fooo06} (= Section 30 \cite{fooo06pre}). See Section 11
\cite{fooo08} for more discussion about this point.)
\par
Thus we have defined a multisection on a neighborhood of the
boundary. We can extend it to the interior so that it satisfies (1)
and (2) in the following way: We first take the quotient
$(V/T^n,E/T^n)$ of our Kuranishi chart. Since the $T^n$ action is free
on $V$, the quotient space is a manifold on which $\Gamma$ acts. Thus
we can use the standard result of the theory of Kuranishi structure
to define a transversal multisection on this chart where the
multisection is already defined. We lift it to $V$ and obtain a
required multisection there. In this way we can construct the
multisection inductively on the Kuranishi charts using the good
coordinate system. (See Corollary 15.15 \cite{fooo08}.)
\par
To show (6), it suffices to take the quotient by the action of
symmetric group and work out the induction on the quotient spaces.
The proof of Lemma \ref{37.184} is now complete.
\end{proof}

\begin{cor}\label{37.187}
$\mathcal M^{\text{\rm main}}_{k+1;\ell}(L(u),\beta;\text{\bf p})
^{\mathfrak s_{\beta,k+1,\ell,\text{\bf p}}}$ is empty,
if one of the following conditions are satisfied.
\begin{enumerate}
\item
$\mu(\beta) - \sum_i (2n - 
\dim D_{\text{\bf p}(i)} -2) < 0. $
\item
$\mu(\beta) - \sum (2n - 
\dim D_{\text{\bf p}(i)} -2) = 0 $ and
$\beta\ne 0$.
\end{enumerate}
\end{cor}
\begin{proof}
We may assume $k = 0$, by Lemma \ref{37.184} (3).

We first consider the case of $\beta = 0$. All the holomorphic
curves in this homotopy class are constant maps. Then our moduli
space is empty for $\ell >0$, since $L(u) \cap D = \emptyset$. This
implies the lemma for the case $\beta = 0$.

We next consider the case $\beta \ne 0$. The virtual dimension of
$\mathcal M^{\text{\rm main}}_{1;\ell}(L(u),\beta;\text{\bf p})$
(which is, by definition, its dimension as a space with Kuranishi
structure) is
\begin{equation}\label{dimensionformula}
n + \mu(\beta) - \sum (2n - 
\dim D_{\text{\bf p}(i)}-2) - 2.
\end{equation}
By the transversality (Lemma \ref{37.184} (1)) and $T^n$
equivariance (Lemma \ref{37.184} (2)), we find that
(\ref{dimensionformula}) is not smaller than $\dim L(u) = n$ if the
perturbed moduli space is nonempty. (We use $\beta \ne 0$ here: If
$\beta = 0$ the virtual dimension of $\mathcal M^{\text{\rm
main}}_{1;0}(L(u),\beta_0)$ is $n-2$ but it is nonempty.) This
finishes the proof of the lemma for the case $\beta \ne 0$.
\end{proof}

We now assume
\begin{equation}\label{dimensionformula2}
\mu(\beta) - \sum (2n - 
\dim D_{\text{\bf p}(i)}-2) = 2,
\end{equation}
and $\beta\ne 0$. Then
$$
\mathcal M^{\text{\rm main}}_{1;\ell}(L(u),\beta;\text{\bf p})
^{\mathfrak s_{\beta,1,\ell,\text{\bf p}}}
$$
has a virtual fundamental {\it cycle}, by Corollary \ref{37.187}. We
introduce the following invariant

\begin{defn}\label{cdef}
We define $c(\beta;\text{\bf p}) \in \Q$ by
$$
c(\beta;\text{\bf p})[L(u)]
= ev_{0*}([\mathcal M^{\text{\rm main}}_{1;\ell}(L(u),\beta;\text{\bf p})
^{\mathfrak s_{\beta,1,\ell,\text{\bf p}}}]).
$$
\end{defn}
\begin{lem}\label{cbetawell}
The number $c(\beta;\text{\bf p})$ is independent of the choice of the system of
multisections $\mathfrak s_{\beta,k+1}$ satisfying $(1)$ - $(6)$ of
Proposition $\ref{37.184}$.
\end{lem}
The proof is the same as the proof of  Lemma 11.7 \cite{fooo08} and
so is omitted.
\begin{rem}\label{Remark69}
\begin{enumerate}
\item
The independence of an open Gromov-Witten invariant similar to
$c(\beta;\text{\bf p})$ was proved in \cite{KL} by taking
equivariant perturbations in the situation where an appropriate
$S^1$-action exists.
\item Since $\mathcal M^{\text{\rm main}}_{1;\ell}(L(u),\beta)$ 
is independent under the permutation of the interior marked points, 
it follows that 
$\mathcal M^{\text{\rm main}}_{1;\ell}(L(u),\beta;\text{\bf p})$ 
is invariant under the permutation of the factors of $\text{\bf p}$.
We take our multisection so that it is invariant under this permutation.
\end{enumerate}
\end{rem}
We use the above moduli spaces to define the operators $\mathfrak
q_{\beta;k,\ell}$ as follows. 
\begin{rem}\label{rem610}
Actually we need to fix $E_0$ and construct $\mathfrak
q_{\beta;k,\ell}$ for $\beta\cap\omega < E_0$. 
We then take inductive limit. This construction is explained in detail in 
Sections 7.2 and 7.4 of \cite{fooo06} 
(= Sections 30 and 32 \cite{fooo06pre}) so is omitted.
For the application in this paper,
we can use $A_{n,K}$ structure in place of $A_{\infty}$ structure, so the
process to take inductive limit is not necessary for applications.
\end{rem}
Let $\text{\bf p} \in
Map(\ell,\underline B)$.  We put
$$\aligned
D(\text{\bf p}) &= D_{\text{\bf p}(1)}\otimes \cdots \otimes D_{\text{\bf p}(\ell)} 
\in \mathcal{A}^{\otimes \ell}, \\
SD(\text{\bf p}) &= \frac{1}{\ell !}\sum_{\sigma \in \mathfrak S_{\ell}}
D_{\text{\bf p}(\sigma(1))}\otimes \cdots \otimes D_{\text{\bf p}(\sigma(\ell))}.
\endaligned$$
Here we recall that 
$D_a =\pi^{-1}(P_a), (a=1,\dots ,B)$ are $T^n$-invariant complex 
submanifolds of $X$ which generate the free abelian group 
$\mathcal {A}$. 
(See the beginning of Section \ref{sec:Potbul}.)
Let  $h_1, \ldots, h_k$ be differential forms on $L(u)$.
We put
$$
\sum (\deg h_i -1) -\mu(\beta) + \sum (2n - 
\dim D_{\text{\bf p}(i)}-2) + 2 = d
$$
where we note that
$$
\aligned \text{deg} [\mathcal M^{\text{\rm
main}}_{1;\ell}(L(u),\beta;\text{\bf p}) ^{\mathfrak
s_{\beta,1,\ell,\text{\bf p}}}] & = \text{codim}[\mathcal
M^{\text{\rm main}}_{1;\ell}(L(u),\beta;\text{\bf p}) ^{\mathfrak
s_{\beta,1,\ell,\text{\bf p}}}] \\
& = -\mu(\beta) + \sum (2n - 
\dim D_{\text{\bf p}(i)}-2) + 2.
\endaligned
$$
(See \eqref{dimensionformula}.) We then define a differential form
of degree $d$ on $L(u)$ by
\begin{equation}
\mathfrak q_{\beta;\ell,k}^{dR}(D(\text{\bf p});h_1,\ldots,h_k) 
= \frac{1}{\ell !}(ev_0)_! (ev_1,\ldots,ev_k)^*(h_1 \wedge \cdots \wedge h_k),
\label{37.188}\end{equation}
where $ev_0$, $ev_i$ are the maps
$$
(ev_0,\ldots,ev_k):
\mathcal M^{\text{\rm main}}_{k+1;\ell}(L(u),\beta;\text{\bf p})^{\mathfrak s_{\beta,1,\ell,\text{\bf p}}}
\longrightarrow L(u)^{k+1}
$$
and $(ev_0)_!$ is the integration along the fiber. 
(Actually we need to put an appropriate sign.  See the end of Section \ref{sec:integration} for sign.)
More precisely we
use the formula (\ref{37.188}) for $(\beta,\ell,k) \ne (0,0,0), (0,0,1)$ and we
put
$$
\mathfrak q_{0;0,1}(h) = (-1)^{n+\deg h+1} dh, \quad \mathfrak q_{0;0,2}(h_1,h_2)
= (-1)^{\deg h_1(\deg h_2+1)} h_1 \wedge h_2.
$$
(\ref{37.188}) defines a map
$
B_{\ell}(\mathcal A[2]) \otimes B_k(H(L(u);\R)[1]) \to  H(L(u);\R)[1].
$
We restrict  it to 
$E_{\ell}(\mathcal A[2]) \otimes B_k(H(L(u);\R)[1])$ and denote the 
restriction by the same symbol 
$\mathfrak q_{\beta;\ell,k}$. Remark \ref{Remark69} (2) 
implies
$$
\mathfrak q_{\beta;\ell,k}^{dR}(SD(\text{\bf p});h_1,\ldots,h_k)
=
\mathfrak q_{\beta;\ell,k}^{dR}(D(\text{\bf p});h_1,\ldots,h_k). 
$$
We use $T^n$-equivariance to show that
$$
ev_0: \mathcal M^{\text{\rm main}}_{k+1;\ell}(L(u),\beta;\text{\bf p})^{\mathfrak s_{\beta,1,\ell,\text{\bf p}}}
\to L(u)
$$
is a proper submersion. Hence the integration along the fiber is
well-defined and gives rise to smooth forms. (It is fairly obvious
that the integration along the fiber on the zero set of a
transversal multisection is well defined and that it satisfies
Stokes' theorem. See Section 16
\cite{fooo08} or Section \ref{sec:integration} of present paper.)
Let $\Omega(L(u))$ be the de Rham complex of $L(u)$.
\begin{defn}\label{def610}
We put
$$
\mathfrak q_{\ell,k}^{dR} = \sum_{\beta} T^{\omega \cap \beta/2\pi}
\mathfrak q_{\beta;\ell,k}^{dR}.
$$
By restricting $\mathfrak q_{\ell,k}^{dR}$ to $E_{\ell}\mathcal A \subset
B_{\ell}\mathcal A$ we obtain
$$
\mathfrak q_{\beta;\ell,k}^{dR}
: E_{\ell}(\mathcal A[2]) \otimes B_k(\Omega(L(u))[1]) \to \Omega(L(u))[1]
$$
of degree $1-\mu(\beta)$ and
$$
\mathfrak q_{\ell,k}^{dR}: E_{\ell}(\mathcal A(\Lambda_+(\R))[2]) \otimes
B_k((\Omega(L(u))\,\widehat\otimes\, \Lambda_0(\R)) [1]) \to (\Omega(L(u))\,\widehat\otimes\, \Lambda_0(\R))[1].
$$
\end{defn}
\begin{prop}\label{qtoric}
$\mathfrak q_{\beta;\ell,k}^{dR}$ satisfies $(\ref{qmaineq})$.
\end{prop}
\begin{proof}
For $\text{\bf p} \in Map(\ell,\underline B)$, $(\mathbb L_1,\mathbb
L_2) \in \text{\rm Shuff}(\ell)$ we put
$$
\text{Split}((\mathbb L_1,\mathbb L_2),\text{\bf p})
= (\text{Split}((\mathbb L_1,\mathbb L_2),\text{\bf p})_1,
\text{Split}((\mathbb L_1,\mathbb L_2),\text{\bf p})_2).
$$
It is easy to see that the coproduct $\Delta(SD(\text{\bf p}))$ is
given by the formula
$$\aligned
\Delta(SD(\text{\bf p})) = 
\sum_{(\mathbb L_1,\mathbb L_2) \in
\text{\rm Shuff}(\ell) } 
\frac{\#|\mathbb L_1|! \#|\mathbb
L_2|!} {\ell!}
&SD(\text{Split}((\mathbb L_1,\mathbb L_2),\text{\bf p})_1) \\
&\otimes SD(\text{Split}((\mathbb L_1,\mathbb L_2),\text{\bf p})_2).
\endaligned$$
Then (\ref{bdry}) and (\ref{37.184}) imply $(\ref{qmaineq})$ in
the same way as Section 7.1 \cite{fooo06} (= Section 13 \cite{fooo06pre}).
\end{proof}
Now for $\mathfrak b\in \mathcal A^2(\Lambda_+)$, we define
\begin{equation}\label{misq}
\mathfrak m_{k}^{dR,\mathfrak b}(h_1,\ldots,h_k) = \mathfrak q^{dR}(e^{\mathfrak b};
h_1,\ldots, h_k) 
= \sum_{\ell\ge 0} \mathfrak q_{\ell,k}^{dR}({\mathfrak b}^{\otimes\ell};
h_1,\ldots, h_k).
\end{equation}
Here
$$
e^{\mathfrak b} = 1 + \mathfrak b + \mathfrak b \otimes \mathfrak b + \cdots.
$$
Proposition \ref{qtoric} implies that
$\mathfrak m_{k}^{dR,\mathfrak b}$ defines a structure of
filtered $A_{\infty}$ algebra on $\Omega(L(u))$.
\par
\begin{rem}\label{37.190}
We can prove that
$(\Omega(L(u))\,\widehat\otimes\, \Lambda_0(\R),\{\mathfrak
m_{k}^{dR,\mathfrak b}\}_{k=0}^{\infty})$ is homotopy equivalent to the
filtered $A_{\infty}$ algebra defined by $(\ref{mkdefeq})$ where we use smooth singular chains.
The proof is a straight forward generalization of that of  Lemma 37.55
\cite{fooo06pre} and is omitted here. We refer readers thereto for the
details. In fact we do not need to use this equivalence for our
applications in this paper. We can just use the de Rham version without involving the
singular homology version.
\end{rem}
We take a canonical model of $(\Omega(L(u))\widehat\otimes \Lambda_0(\R)),\{\mathfrak
m_{k}^{\mathfrak b}\}_{k=0}^{\infty})$ to obtain a filtered
$A_{\infty}$ algebra $(H(L(u);\Lambda_0(\R)),\{\mathfrak
m_{k}^{\mathfrak b,{\text{\rm can}}}\}_{k=0}^{\infty})$. 
Namely we define
\begin{equation}
\mathfrak m_k^{\mathfrak b,{\text{\rm can}}} : (H(L(u);\Lambda_0(\R))^{\otimes k} \to (H(L(u);\Lambda_0(\R))
\end{equation}
as follows. We fix a $T^n$-equivariant Riemannian metric on $L(u)$.
Then a differential form is harmonic with respect to this metric if and only if it 
is $T^n$ invariant. We identify $H^k(L(u);\R)$ with the set of $T^n$-invariant $k$-forms.
\par
Suppose $h_1,\cdots,h_k \in H(L(u);\R)$ then 
$\mathfrak q^{dR}_{\beta;\ell,k}(\text{\bf p};
h_1,\cdots, h_k)$ is $T^n$ invariant and so is an element of $H(L(u);\R)$.
In fact, all the moduli spaces and evaluation maps involved are $T^n$ equivariant.
Thus the restriction of   $\mathfrak q^{dR}_{\beta;\ell,k}$ to the harmonic forms 
defines an operator 
$$
\mathfrak q^{{\text{\rm can}}}_{\beta;\ell,k} :  E_{\ell}(\mathcal A[2]) \otimes B_k(H(L(u);\R)[1]) \to H(L(u);\R)[1].
$$
It induces
$$
\mathfrak q^{{\text{\rm can}}}_{\ell,k} :  
E_{\ell}(\mathcal A(\Lambda_+(\R))[2]) \otimes B_k(H(L(u);\Lambda_0(\R))[1]) \to H(L(u);\Lambda_0(\R))[1]
$$
in the same way as Definition \ref{def610}. Proposition \ref{qtoric} 
holds for $\mathfrak q^{{\text{\rm can}}}_{\ell,k}$. Therefore 
by the same formula as (\ref{misq}) we define $\mathfrak m_k^{\mathfrak b,{\text{\rm can}}}$.
\begin{lem}\label{caneqcannasi}
The filtered $A_{\infty}$ algebra
$(H(L(u);\Lambda_0(\R)),\{\mathfrak
m_{k}^{\mathfrak b,{\text{\rm can}}}\}_{k=0}^{\infty})$ is strict and unital homotopy equivalent to 
$(\Omega(L(u))\,\widehat\otimes\, \Lambda_0(\R),\{\mathfrak
m_{k}^{dR,\mathfrak b}\}_{k=0}^{\infty})$.
\end{lem}
\begin{proof}
Let $\mathfrak f_{1} : H(L(u);\Lambda_0(\R)) \to \Omega(L(u))\,\widehat\otimes\, \Lambda_0(\R)$
be the inclusion as harmonic forms. We set other $\mathfrak f_k$ to be $0$. 
By definition they define filtered $A_{\infty}$ homomorphism, which is strict and unital.  
(See Definitions 3.2.29,  3.3.11 \cite{fooo06} = Definitions 7.20, 8.17 \cite{fooo06pre}.)
It induces isomorphism on $\overline{\mathfrak m}_1$ cohomology.
The lemma now follows from Theorem 5.2.45 \cite{fooo06} (= Theorem 15.45 \cite{fooo06pre}.)
\end{proof}
\begin{rem}
In the general situation the construction of the $A_{\infty}$ operators of the canonical model 
and the homotopy equivalence in Lemma \ref{caneqcannasi} is by summation over trees and is more 
involved.  (See Section 5.4 \cite{fooo06} = Section 23 \cite{fooo06pre}.)
Here we take a short cut using the fact $L(u)$ is a torus and the $T^n$ 
equivariance.
\end{rem}
\section{Calculation of potential function with bulk}
\label{sec:CalPot}
\par
We have defined a potential function $\mathfrak{PO}^u$
in Section \ref{sec:Potbul}. In this section we will
partially calculate it and prove various results stated in that section
whose proofs have been postponed until this section.
The following is the key lemma for this purpose.
\begin{lem}\label{37.189}
Let $\mathfrak x \in H^1(L(u);\Lambda_{+})$, $\beta \in
H_2(X,L;\Z)$, and $\text{\bf p} \in Map(\ell,\underline B)$.  We
assume $(\ref{dimensionformula2})$. Then we have
$$
\mathfrak q_{\beta;\ell,k}^{{\text{\rm can}}}(D(\text{\bf p});\mathfrak x,\ldots,\mathfrak x)
= \frac{c(\beta;\text{\bf p})}{\ell!k!}
(\partial \beta\cap \mathfrak x)^k \cdot PD([L(u)]),
$$
where $PD([L(u)])$ denotes the Poincar\'e dual to the fundamental class
$[L(u)] \in H_n(L(u);\Z)$ and 
$c(\beta;\text{\bf p})$ is defined by 
Definition $\ref{cdef}$.
\end{lem}
\begin{proof}
The proof is similar to that of Lemma 11.8 \cite{fooo08}
and proceed as follows.
Let $h$ be a harmonic representative of the class $\mathfrak x$.
We have:
\begin{equation}\label{eq:mkonharmonic}
\int_{L(u)} \mathfrak q^{dR}_{\beta;\ell,k}(D(\text{\bf p});h,\ldots,h)
=
\frac{c(\beta;\text{\bf p})}{\ell!k!}(\partial \beta\cap \mathfrak x)^k.
\end{equation}
The proof is the same as that of Formula (11.10) \cite{fooo08},
using Definition \ref{cdef}.
The lemma follows immediately, by using Lemma \ref{lem72}.
\end{proof}
We put
$$
C_{n} = \{(t_1,\dots,t_n) \in [0,1]^{n} \mid 
0 \le t_1 \le  \dots \le  t_n \le 1.
\}
$$
\begin{lem}\label{lem72}
Let $\theta$ be a one form on $[0,1]$.
It induces $n$ form $\theta_n = \theta\times \dots \times \theta$ on $[0,1]^{n}$.
We then have
$$
\int_{C_{n}}\theta_n = \frac{1}{n!} \left(\int_{[0,1]}\theta\right)^n.
$$
\end{lem}
\begin{proof}
Let $\mathfrak S_{n}$ be the group of permutations of $\{1,\dots,n\}$.
It acts on $[0,1]^{n}$ as permutations of the factors.
It is then easy to see that
$
\sigma^* \theta_n = \text{\rm sign}\,\sigma \theta_n,
$
and $\sigma$ is orientation preserving if and only if 
$\text{\rm sign}\,\sigma = 1$.
Moreover $C_{n}$ is a fundamental domain of $\mathfrak S_{n}$
action on $[0,1]^{n}$.
It implies
$$
\int_{C_{n}}\theta_n = 
\frac{1}{n!} \int_{[0,1]^{n}}\theta_n = \frac{1}{n!} \left(\int_{[0,1]}\theta\right)^n.
$$
\end{proof}
\begin{proof}[Proof of Proposition \ref{unobstruct}]
This is an immediate consequence of Corollary \ref{37.187} and Lemma \ref{37.189}. In fact, 
it implies that $ \mathfrak m_{k}^{{\text{\rm can}}, \mathfrak b}(b,\ldots,b) $ can be
only degree $0$, that is proportional to $PD[L(u)]$.
\end{proof}
\begin{proof}[Proof of Theorem \ref{weakpotential}]
Let $\mathfrak b = \sum_{a=1}^{B}\mathfrak b_a D_a$ with 
$\mathfrak b_a \in \Lambda_+$.
We assume $\mathfrak b_a$ is $G_{\text{\rm bulk}}$-gapped.
We have
$$
e^{\mathfrak b} = \sum_{\ell}\sum_{\text{\bf p} \in Map(\ell,\underline
B)} \mathfrak b^{\text{\bf p}}D(\text{\bf p}).
$$
Here
$$
\mathfrak b^{\text{\bf p}} = \prod_j \mathfrak b_{\text{\bf p}(j)}.
$$
We have:
$$
\mathfrak{PO}^u(\mathfrak b;b)
=
\sum_{\beta,\text{\bf p},k}\mathfrak b^{\text{\bf p}}
T^{\beta\cap\omega/2\pi}
\mathfrak q^{{\text{\rm can}}}_{\beta;\vert\text{\bf p}\vert,k}(D(\text{\bf p});b,\ldots,b).
$$
By the degree counting the sum is nonzero only when $(\ref{dimensionformula2})$
is satisfied. Therefore by Lemma \ref{37.189} we have
\begin{equation}\label{calcPO}
\aligned
\mathfrak{PO}^u(\mathfrak b;b)
&= \sum_{\beta,\text{\bf p},k} \mathfrak b^{\text{\bf p}} T^{\beta\cap\omega/2\pi}\frac{c(\beta;\text{\bf p})}{k!\vert\text{\bf p}\vert!}
(b\cap\partial \beta)^k
\\
&= \sum_{\beta,\text{\bf p}}\frac{1}{\vert\text{\bf p}\vert!}\mathfrak b^{\text{\bf p}}T^{\beta\cap\omega/2\pi}
c(\beta;\text{\bf p}) \exp(b \cap \partial \beta).
\endaligned
\end{equation}

The sum of the cases $\beta = \beta_i$ $(i=1,\ldots,m)$ and
$\vert\text{\bf p}\vert = 0$ is $\mathfrak{PO}^u_0(b)$.
\par
We next study other terms for $\vert\text{\bf p}\vert \neq 0$.
We first consider the case $\beta = \beta_i$, $\ell \ne 0$.
Then the corresponding term is a sum of the terms written as
\begin{equation}\label{term3.41}
c T^{\ell_i(u) + \rho} y^{\vec v_i}.
\end{equation}
Here $c\in \Q$ and $\rho$ is a sum of the numbers which appear as exponents of $\mathfrak b_a$
for various $a$. It is nonzero since $\ell \ne 0$ and  $\mathfrak b_a \in \Lambda_+$.
Therefore $\rho \in G_{\text{\rm bulk}} \setminus\{0\}$. Therefore (\ref{term3.41}) is of the
form appearing in the right hand side of (\ref{eq:weakPO}).
\par
We next consider the case $\beta \ne \beta_i$
$(i=1,\ldots,m)$. We assume $c(\beta;\text{\bf p}) \ne 0$ in addition. Then
by Proposition \ref{copy37.177} (5) we have $e^i$ and $\rho$ such that
$$
\frac{\beta\cap\omega}{2\pi} = \sum_{i} e^i\ell_i(u) + \rho.
$$
Here $e^i \in \Z_{\ge 0}$ and $\sum e^i > 0$ and $\rho$ is a sum of symplectic
areas of holomorphic spheres divided by $2\pi$. Thus the corresponding term is a sum of the terms
$$
c T^{\sum_{i} e^i\ell_i(u) + \rho+\rho'} y^{\sum e^i\vec v_i}.
$$
Here $c \in \Q$ and $\rho'$ is a sum of the numbers that appear
as the exponents of $\mathfrak b_a$ for various $a$. This is exactly of the
form in the right hand side of (\ref{eq:weakPO}).
\par
Finally we prove (\ref{rhogoinfty}) by contradiction.
\par
We put
$$
F(\beta;\ell)
=\sum_{\vert \text{\bf p}\vert = \ell}
\mathfrak b^{\text{\bf p}}T^{\beta\cap\omega/2\pi}
c(\beta;\text{\bf p}) \exp(b \cap \partial \beta)
= 
\sum_{\vert \text{\bf p}\vert = \ell}
\mathfrak b^{\text{\bf p}}T^{\beta\cap\omega/2\pi}
c(\beta;\text{\bf p})y^{\vec v_i}.
$$
Let the $\sigma$-th term of (\ref{eq:weakPO}) 
appear in $F(\beta(\sigma);\ell(\sigma))$.
We assume that $\rho_{\sigma}$, the $T$-exponent of 
the $\sigma$-th term, is bounded and will deduce a contradiction.
\par
We first consider the case where 
there are infinitely many  $\sigma$'s with the same
$\beta = \beta(\sigma)$. 
The $T$-exponent of $F(\beta(\sigma);\ell(\sigma))$ is not smaller than 
\begin{equation}\label{exponenttoinfinity1}
\frac{\beta \cap \omega}{2\pi} + \ell(\sigma) \rho_0
\end{equation}
where
$$
\rho_0 = \inf (G_{\text{\rm bulk}} \setminus \{0\}).
$$
(\ref{exponenttoinfinity1}) and the boundedness of $\rho_{\sigma}$ imply that 
$\ell(\sigma)$ is bounded. This is impossible.
\par
We next consider the case where 
there are infinitely many different $\beta(\sigma)$'s 
such that 
$\rho_{\sigma}$ is bounded.
Each term of $F(\beta(\sigma);\ell(\sigma))$ is of the form:
\begin{equation}\label{exponenttoinfinity2}
c_{\sigma} T^{\ell'_{\beta(\sigma)}(u) + \rho} y_1^{v'_{\sigma,1}}\cdots  y_n^{v'_{\sigma,n}}
\end{equation}
such that $d\ell'_{\beta(\sigma)}= (v'_{\sigma,1},\ldots,v'_{\sigma,n})$
and $\rho \ge 0$ are in the discrete 
monoid generated by the exponents of $\mathfrak b$.
\par
We apply Proposition \ref{copy37.177} (5) and obtain
$$
\beta(\sigma) = \sum_{i=1}^m k_{i,\sigma}\beta_i + \sum_j \alpha_{\sigma,j}
$$
with $k_{i,\sigma} \geq 0$.
We have
$$
\ell'_{\beta(\sigma)}(u) = \sum_{i=1}^m k_{i,\sigma}\ell_i(u)
$$
and
$$
\rho_{\sigma} = \sum_j \frac{\alpha_{\sigma,j}\cap [\omega]}{2\pi} +
(\text{a sum of exponents appearing in $\mathfrak b$}).
$$
If $(k_{1,\sigma},\ldots,k_{n,\sigma}) \in \Z^n$ is bounded as $\sigma \to \infty$,
then $\sum_j \alpha_{\sigma,j} \in H_2(X;\Z)$ is necessarily unbounded. 
(This is because $\beta(\sigma)$ are assumed to be different to each other.)
Therefore
$$
\rho_{\sigma} \ge \sum_j \frac{\alpha_{\sigma,j}\cap [\omega]}{2\pi}
$$
goes to infinity, a contradiction.
\par
We next assume that $(k_{1,\sigma},\ldots,k_{n,\sigma}) \in \Z^n$ is unbounded. Then
the sum of its Maslov indices
$$
\sum_{i=1}^n k_{i,\sigma} \mu(\beta_i) = 2\sum_{i=1}^n k_{i,\sigma}
$$
is unbounded. (We recall $k_{i,\sigma} \ge 0$.) Therefore thanks to the formula
$$
\mu(\beta({\sigma})) = \sum_{i=1}^m k_{i,\sigma}\mu(\beta_i)  + 2\sum_j c_1(X) \cap \alpha_{\sigma,j},
$$
one of the following alternatives (not-necessarily exclusive to each other)
must occur:
\par\medskip
\noindent(a) $\vert\sum_j c_1(X) \cap \alpha_{\sigma,j}\vert$ is unbounded.
\par
\noindent(b) $\mu(\beta({\sigma}))$ is unbounded.
\par\medskip
For the case (a), $\sum_j \alpha_{\sigma,j} \in H_2(X;\Z)$ is unbounded.
Therefore
$$
\rho_{\sigma} \ge \sum_j \frac{\alpha_{\sigma,j}\cap [\omega]}{2\pi}
$$
goes to infinity similarly as before. This is a contradiction.
\par
For the case (b), we recall the general dimension formula
$$
\dim \mathcal M^{\text{\rm main}}_{1;\ell({\sigma})}(L(u);\beta({\sigma}))
= 2 \ell({\sigma}) + n + \mu(\beta({\sigma})) -2.
$$
On the other hand, we must have
$$
\dim \mathcal M^{\text{\rm main}}_{1;\ell({\sigma})}(L(u);\beta({\sigma}):\text{\bf p})
= n,
$$
since $\mathfrak q_{\beta({\sigma});\ell({\sigma});k}(\text{\bf p};b)$ is nonzero.
Therefore we should have
$$
\sum_{j=1}^{\ell({\sigma})} (\deg 
\text{\bf p}(j) - 2) = \mu(\beta({\sigma}))-2.
$$
This goes to infinity and so does $\ell({\sigma})$ as $\sigma \to \infty$.
Since we have
$$
\rho_{\sigma} \ge \ell({\sigma}) \rho_0,
$$
$\rho_{\sigma}$ goes to infinity. This is a contradiction.
The proof of Theorem \ref{weakpotential} is now complete.
\end{proof}

\begin{proof}[Proof of Proposition \ref{POcalcFano}]
We assume that $\mathfrak b$ is as in (\ref{formfrakb}).
We note that $D_{i(l,r)} \in H^2(D;\Z)$.
Therefore a dimension counting argument shows that only $\beta$ with
$\mu(\beta) =2$ contributes to $\mathfrak{PO}^u(\mathfrak b;b)$.
Then by the assumption that $X$ is Fano we derive that
only $\beta_i$'s for $(i=1,\ldots,m)$ contribute among those $\beta$'s.
\par
Thus we have obtained
\begin{equation}\label{calcPOFano}
\mathfrak{PO}^u(\mathfrak b;b)
= \sum_{i=1}^m\sum_{\text{\bf p}}\frac{1}{\vert\text{\bf p}\vert !}
\mathfrak b^{\text{\bf p}}T^{\ell_i(u)}
c(\beta_i;\text{\bf p}) y^{\vec v_i}.
\end{equation}
We next calculate $c(\beta_i;\text{\bf p})$.
By definition we have
$$
c(\beta_i;\text{\bf p})[{L(u)}]
= ev_{0*}(\mathcal M^{\text{\rm main}}_{1;\vert\text{\bf p}\vert}(L(u),\beta_i;\text{\bf p})
^{\mathfrak s_{\beta,1,\ell,\text{\bf p}}})
$$
and
$$
\mathcal M^{\text{\rm main}}_{1;\vert\text{\bf p}\vert}(L(u),\beta_i;\text{\bf p})
= \mathcal M^{\text{\rm main}}_{1;\vert\text{\bf p}\vert}(L(u),\beta_i)
\times_{X^{\vert\text{\bf p}\vert}} \prod_{j=1}^{\vert\text{\bf p}\vert} D_{\text{\bf p}(j)}.
$$
We consider
$$
ev_0: \mathcal M^{\text{\rm main}}_{1;0}(L(u),\beta_i) \to L(u).
$$
It is a diffeomorphism by Proposition \ref{copy37.177}.
We fix $p_0 \in L(u)$ and let $ev_0^{-1}(p_0)=\{\varphi\}$.
Since $[\varphi] = \beta_i$ it follows that
\begin{equation}
[\varphi] \cap D(\text{\bf p}(j)) =
\begin{cases}
1 &\text{$j = i$}, \\
0 &\text{$j\ne i$}.
\end{cases}
\nonumber
\end{equation}
We remark that the number $c(\beta_i;\text{\bf p}) $ is well defined,
that is, independent of the perturbation.
So we can perform the calculation in the homology level to find that
\begin{equation}\label{cpj}
c(\beta_i;\text{\bf p})
=
\begin{cases}
1 &\text{$\text{\bf p}(j) = i$ for all $j$}, \\
0 &\text{otherwise}.
\end{cases}
\end{equation}
Thus (\ref{calcPOFano}) is equal to
$$
\sum_{i=1}^m\exp(\mathfrak b_i)T^{\ell_i(u)} y^{\vec v_i}.
$$
Since $\mathfrak b_i = 0$ by definition for $i \geq \CK +1$ in the decomposition of $\mathfrak b$
in \eqref{formfrakb}, this sum can be rewritten as \eqref{POFanoformula}
which finishes the proof of Proposition \ref{POcalcFano}.
\end{proof}
\begin{proof} [Proof of Proposition \ref{nonfanopert}.]
We assume that $\mathfrak b$ satisfies Condition \ref{gapcondb}.
Again by dimension counting only $\beta$ with
$\mu(\beta) =2$ contributes to $\mathfrak{PO}^u(\mathfrak b;b)$.
In Proposition \ref{nonfanopert} we do not assume
that $X$ is Fano. So the homology classes $\beta$ other than $\beta_i$ ($i=1,\cdots,m$)
may contribute.
\par
We first study the contribution of $\beta_i$'s. We put
$$
\mathfrak P_0(\mathfrak b;b)
= \sum_{i=1}^m\sum_{\text{\bf p}}\frac{1}{\vert\text{\bf p}\vert !}
\mathfrak b^{\text{\bf p}}T^{\ell_i(u)}
c(\beta_i;\text{\bf p}) y^{\vec v_i}.
$$
Substituting (\ref{cpj}) into $\mathfrak P_0$, we obtain
$$\aligned
\mathfrak P_0(\mathfrak b';b)  - \mathfrak P_0(\mathfrak b;b)
&=  \sum_{i=1}^m(\exp({\mathfrak b'_i}) - \exp({\mathfrak b_i}))
T^{\ell_i(u)} y^{\vec v_i} \\
& = \sum_{i=1}^m(\exp({\mathfrak b_{i(l,r)}}+c T^{\lambda} ) - \exp({\mathfrak b_{i(l,r)}}))
T^{S_l} y^{\vec v_{i(l,r)}}.
\endaligned$$
This is in the form of the sum of the first 2 terms of
the right hand side of (\ref{firstapprox}).
\par
We next study the contribution of  $\beta \ne \beta_i$.
We have only to consider $\beta$'s satisfying that
$\mu(\beta) = 2$ and $\mathcal M^{\text{\rm main}}(L(u);\beta) \neq \emptyset$. We put
$$
\mathfrak P_{\beta}(\mathfrak b;b)
= \sum_{\text{\bf p}}\frac{1}{\vert\text{\bf p}\vert !}
\mathfrak b^{\text{\bf p}}T^{\beta \cap \omega/2\pi}
c(\beta;\text{\bf p}) \exp(b \cap \partial \beta).
$$
If we write
$$
\beta = \sum_{i=1}^m e^i_{\beta}\beta_i + \sum_j\alpha_{\beta,j}
$$
as in Proposition \ref{copy37.177} (5),
then we have
$$\aligned
\frac{\beta \cap [\omega]}{2\pi} &= \sum_{i=1}^m e^i_{\beta} \ell_i(u) + \sum_j
\frac{\alpha_{\beta,j} \cap [\omega]}{2\pi} \\
\exp(b\cap \partial\beta) &= y^{\sum_{i=1}^m e^i_{\beta}\vec v_i}.
\endaligned$$
We have $e_{\beta}^i \ge 0$ and
$\sum_i e_{\beta}^i > 0$. Moreover, since
$\beta \ne \beta_i$ ($i = 1,\ldots,m$) it follows that
$\sum_j \alpha_{\beta,j} \ne 0$. (We use $\mu(\beta) = 2$ to  prove this.)
Therefore
$$
\rho_{\beta} = \sum_j \alpha_{\beta,j} \cap [\omega]/2\pi > 0.
$$
Hence
$$
\mathfrak P_{\beta}(\mathfrak b';b) - \mathfrak P_{\beta}(\mathfrak b;b)
= \sum_{\sigma}\sum_{h=1}^{\infty} c_{\sigma,h}
T^{\sum_i e_{\beta}^i \ell_i(u) + \rho_{\beta} + h\lambda
+\rho'_{\sigma}} y^{\sum_i e_{\beta}^i \vec v_i},
$$
where $c_{\sigma,h} \in R$ and $\rho'_{\sigma}$ is a sum
of exponents of $T$ in $\mathfrak b$. This corresponds to the
third term of (\ref{firstapprox}).
In fact $\ell'_{\sigma} = \sum_i e_{\beta}^i \ell_i$,
$\rho_{\sigma} = \rho'_{\sigma} + \rho_{\beta}$ therein.
\par
Now Proposition  \ref{nonfanopert} follows if we rewrite
$$
\mathfrak{PO}^u(\mathfrak b';b) - \mathfrak{PO}^u(\mathfrak b;b)
= (\mathfrak P_0(\mathfrak b';b)  - \mathfrak P_0(\mathfrak b;b) )
+ \sum_{\beta} (\mathfrak P_{\beta}(\mathfrak b';b) - \mathfrak P_{\beta}(\mathfrak b;b) ).
$$
\end{proof}

\section{Floer cohomology and non-displacement of Lagrangian submanifolds}
\label{sec:HFbulk}
\par
In this section we discuss how we apply Floer cohomology and
the potential function to the study of non-displacement property of
Lagrangian submanifolds. Especially we will prove Proposition
\ref{prof:bal2}. The argument of this
section is a minor modification and combination of the one given in
\cite{fooo08} except that we integrate
\emph{bulk deformations} into the argument therein.
(The way to use bulk deformation in the study of
non-displacement of Lagrangian submanifold is described in 
Section 3.8 \cite{fooo06} (= Section 13 \cite{fooo06pre}.)) This
generalization is quite straightforward. However we  gives details in
order to make this paper as self-contained as possible for readers'
convenience. To avoid too much repetition of the materials from
\cite{fooo06}, we will use the de Rham cohomology
version instead of the singular cohomology version of the filtered $A_\infty$
algebra associated to a Lagrangian submanifolds in this section.
The de Rham version is suitable for the purpose of the present paper since
we can easily realize {\it strict} unit in de Rham theory. We are
using weak bounding cochains which are easier to handle in case
strict unit (rather than homotopy unit) exists.
\par
In this section we put $R=\C$. We write $\Lambda_0$, $\Lambda_+$,
$\Lambda$ in place of $\Lambda_0(\C)$, $\Lambda_+(\C)$,
$\Lambda(\C)$ respectively, in this section.
\par
We first explain how we enlarge the deformation parameters
$(\mathfrak b,\mathfrak x)$ of Floer cohomology to
$$
\mathcal A(\Lambda_+) \times H^1(L(u);\Lambda_0) \supset \mathcal A(\Lambda_+)
\times H^1(L(u);\Lambda_+),
$$
by including $b \in H^1(L(u);\Lambda_0) \supset H^1(L(u);\Lambda_+)$
as in \cite{fooo08} where we borrowed the idea of Cho \cite{cho07}
of considering Floer cohomology twisted with flat \emph{non-unitary} line bundles
in the study of non-displacement problem of Lagrangian submanifolds.
\begin{defn}\label{rhoyx}
Let
\begin{equation}\label{frakxdef}
\mathfrak x = \sum_i \mathfrak x_i \text{\bf e}_i \in H^1(L(u);\Lambda_{0})
\end{equation}
and
\begin{equation}\label{frakxdef2}
\mathfrak x_i = \mathfrak x_{i,0} + \mathfrak x_{i,+}
\end{equation}
where $\mathfrak x_{i,0} \in \C$ and
$\mathfrak x_{i,+} \in \Lambda_{+}$.
We put
$$
\mathfrak y_{i,0} = \exp(\mathfrak x_{i,0})
= \sum_{n=0}^{\infty}\frac{\mathfrak x_{i,0}^n}{n!} \in \C.
$$
Let
$
\rho: H_1(L(u);\Z) \to \C\setminus \{0\}
$
be the representation defined by
$
\rho(\text{\bf e}_i) = \mathfrak y_{i,0}.
$
\end{defn}
\begin{defn}\label{mkbwithtwist}
We define
$$
\mathfrak q_{\ell,k}^{{\text{\rm can}},\rho}: E_{\ell}\mathcal A(\Lambda_+)[2] \otimes
B_k(H(L(u);\Lambda_0)[1] \to H(L(u);\Lambda_0)[1]
$$
by
\begin{equation}\label{qrhodef}
\mathfrak q^{{\text{\rm can}},\rho}_{\ell,k} = \sum_{\beta}
\mathfrak y_{1,0}^{\partial \beta \cap \text{\bf e}_1^*}
\cdots
\mathfrak y_{n,0}^{\partial \beta \cap \text{\bf e}_n^*}
T^{\beta\cap \omega/2\pi}
\mathfrak q^{{\text{\rm can}}}_{\beta;\ell,k}.
\end{equation}
We then define:
\begin{equation}
\mathfrak m_k^{\mathfrak b,{\text{\rm can}},\mathfrak x}(x_1,\ldots,x_k) = \sum_{\ell}
\mathfrak q^{{\text{\rm can}},\rho}_{\ell}(\mathfrak b^{\ell};e^{\mathfrak x_+} x_1
e^{\mathfrak x_+} \cdots e^{\mathfrak x_+} x_k e^{\mathfrak x_+}),
\end{equation}
and
\begin{equation}\label{POwithtwist}
\mathfrak{PO}^u_{\rho}(\mathfrak b,\mathfrak x_+) = \sum_{\ell,k} \mathfrak
q^{{\text{\rm can}},\rho}_{\ell,k}(\mathfrak b^{\ell};\mathfrak x_+^k).
\end{equation}
We also define
$$
\mathfrak q_{\ell,k}^{\rho}: E_{\ell}(\mathcal A(\Lambda_+)[2]) \otimes
B_k((\Omega(L(u)\,\widehat\otimes\,\Lambda_0)[1]) \to (\Omega(L(u))\,\widehat\otimes\,\Lambda_0)[1]
$$
and $\mathfrak m_k^{\mathfrak b,\mathfrak x}$ in the same way.
\end{defn}
\begin{lem}
\begin{enumerate}
\item
$\mathfrak m_k^{\mathfrak b,\mathfrak x}$, $\mathfrak m_k^{\mathfrak b,{\text{\rm can}},\mathfrak x}$ define the structures of
filtered $A_{\infty}$ algebras on $\Omega(L(u)) \widehat{\otimes} \Lambda_0$ and on
$H(L(u);\Lambda_0)$, respectively.
\item Let $\mathfrak{PO}^u(\mathfrak b; \cdot): H^1(L(u);\Lambda_0) \to \Lambda_+$ be the extended 
potential function as in Lemma $\ref{lem:extend}$. Then we have
$$
\mathfrak{PO}^u_{\rho}(\mathfrak b;\mathfrak x_+) = \mathfrak{PO}^u(\mathfrak b;\mathfrak x)
$$
if $(\ref{frakxdef2})$ holds.
\end{enumerate}
\end{lem}
\begin{proof}
The proof of (1) is the same as that of Proposition 12.2 \cite{fooo08}.
The proof of (2) is the same as the proof of Lemma 4.9 \cite{fooo08}.
\end{proof}
\begin{defn}
$$
HF((L(u),\mathfrak b,\mathfrak x),(L(u),\mathfrak b,\mathfrak x);\Lambda_0)
:= \frac{\text{\rm Ker}\,\,\mathfrak m_1^{\mathfrak b,{\text{\rm can}},\mathfrak x}}
{\text{\rm Im}\,\,\mathfrak m_1^{\mathfrak b,{\text{\rm can}},\mathfrak x}}.
$$
\end{defn}
In the same way as Lemma \ref{caneqcannasi} we can prove
\begin{equation}\label{caniso}
\frac{\text{\rm Ker}\,\,\mathfrak m_1^{\mathfrak b,{\text{\rm can}},\mathfrak x}}
{\text{\rm Im}\,\,
\mathfrak m_1^{\mathfrak b,{\text{\rm {\text{\rm can}}}},\mathfrak x}}
\cong
\frac{\text{\rm Ker}\,\,\mathfrak m_1^{\mathfrak b,\mathfrak x}}
{\text{\rm Im}\,\,\mathfrak m_1^{\mathfrak b,\mathfrak x}}.
\end{equation}
\begin{proof}[Proof of Theorem \ref{homologynonzero}]
Based on the above definition the proof goes in the same way as the
proof of Theorem 4.10 \cite{fooo08}.
\end{proof}
We next prove Proposition \ref{prof:bal2}.
Again the proofs will be similar to the proofs of
Proposition 4.12 and Theorem 5.11 \cite{fooo08} in which we use
a variant of Theorem \ref{HFdisplace} that also employs Floer
cohomology twisted by non-unitary flat bundles (whose holonomy is $\rho$ as above).
\par
Now we provide the details of the above mentioned proofs.
\par
Let $\psi_t: X \to X$ be a Hamiltonian isotopy with $\psi_0 =$ identity.
We put $\psi_1 = \psi$. We consider the pair
$$
L^{(0)}=L(u), \quad L^{(1)} = \psi(L(u))
$$
such that $L^{(1)}$ is transversal to $L^{(0)}$.
By perturbing $\psi_t$ a bit, we may assume the following:
\begin{conds}\label{butukarazu}
If $p \in L(u) \cap \psi(L(u))$, then
\begin{equation}
\psi_t(p) \notin \pi^{-1}(\partial P)
\end{equation}
for any $t \in [0,1]$.
\end{conds}
We put $\psi_{t *} J = J_t$ where $J$ is the standard complex structure of $X$. Then
$J_0 = J$ and $J_1 = \psi_*(J)$.
\par
Let $p,q \in L^{(0)}\cap L^{(1)}$. We consider the
homotopy class of maps
\begin{equation}\label{conncorbitmap}
\varphi: \R \times [0,1] \to X
\end{equation}
such that
\begin{enumerate}
\item
$\lim_{\tau \to -\infty} \varphi(\tau,t) = p$,
$\lim_{\tau \to +\infty} \varphi(\tau,t) = q$.
\item
$\varphi(\tau,0) \in L^{(0)}$, $\varphi(\tau,1) \in L^{(1)}$.
\end{enumerate}
We denote by $\pi_2(L^{(1)},L^{(0)};p,q)$ the set of all such homotopy
classes. We then define maps
\begin{equation}\label{joinhomotopy}
\aligned
& \pi_2(L^{(1)},L^{(0)};p,r) \times \pi_2(L^{(1)},L^{(0)};r,q)
\to \pi_2(L^{(1)},L^{(0)};p,q), \\
& \pi_2(X,L^{(1)}) \times \pi_2(L^{(1)},L^{(0)};p,q)
\to \pi_2(L^{(1)},L^{(0)};p,q), \\
&  \pi_2(L^{(1)},L^{(0)};p,q)  \times \pi_2(X,L^{(0)})
\to \pi_2(L^{(1)},L^{(0)};p,q),
\endaligned
\end{equation}
as follows. The map in the first line is an
obvious concatenation.
To define the map in the second line we first fix a
base point $p_0 \in L^{(1)}$. Let $\varphi: \R \times [0,1] \to X$ represent an
element of $\pi_2(L^{(1)},L^{(0)};p,q)$ and
$\phi: D^2 \to X$ an element of $\pi_2(X,L^{(1)})$, respectively.
($\phi(1) = p_0$ and $\phi(\partial D^2) \subset L^{(1)}$.) We take
a path $\gamma$ joining $p_0$ and $\varphi(0,1)$ in $L^{(1)}$.
We take the boundary connected sum $(\R \times [0,1])\# D^2$
of $\R \times [0,1]$ and $D^2$ along $(0,1)$ and $1$,
which is nothing but  $\R \times [0,1]$.
We use $\gamma$ to obtain the map $\varphi \#_{\gamma} \phi
: \R \times [0,1] \cong (\R \times [0,1])\# D^2 \to X$ joining $\varphi$ and $\phi$.
The homotopy class of
$\varphi \#_{\gamma} \phi$ is independent of $\gamma$ since $\pi_1(L^{(1)})$ acts trivially on $\pi_2(X,L^{(1)})$.
(We use the fact that $L^{(1)}$ is a torus here.) Thus we have defined the map
in the second line. The map in the third line is defined in the same way.
\par
We denote the maps in (\ref{joinhomotopy}) by $\#$.
\begin{rem}
\begin{enumerate}
\item We here use the set $\pi_2(L^{(1)},L^{(0)};p,q)$ of
\emph{homotopy} classes. In the last two sections we use homology
group $H_2(X,L(u);\Z)$. In fact $H_2(X,L(u);\Z) \cong \pi_2(X,L(u))$
in our situation and so we can instead use the latter.
\item
The definition of $\#$ above is rather ad hoc since we use the fact that $L^{(1)}$ is a torus.
In the general case we use the set of $\Gamma$ equivalence classes of the elements of
$\pi_2(L^{(1)},L^{(0)};p,q)$ in place of $\pi_2(L^{(1)},L^{(0)};p,q)$ itself.
(See Definition-Proposition 2.3.9 \cite{fooo06} = Definition-Proposition 4.9 \cite{fooo06pre}.)
\end{enumerate}
\end{rem}
\begin{defn}
We consider the moduli space of maps (\ref{conncorbitmap})
satisfying (1), (2) above, in homotopy class $B \in
\pi_2(L^{(1)},L^{(0)};p,q)$, and satisfying the equation:
\begin{equation}\label{CReqn-Jt}
\frac{\partial\varphi}{\partial \tau}
+ J_t \left(\frac{\partial\varphi}{\partial t}\right) = 0.
\end{equation}
We denote it by
$$
\widetilde{\mathcal M}^{\text{reg}}(L^{(1)},L^{(0)};p,q;B).
$$
We put $k_1$ marked points $(\tau_i^{(1)},1)$ on $\{(\tau,1)\mid
\tau \in \R\}$, $k_0$ marked points $(\tau^{(0)}_i,0)$ on
$\{(\tau,0)\mid \tau \in \R\}$, and $\ell$ marked points
$(\tau_i,t_i)$ on $\R \times (0,1)$. We number the $k_1 + k_0$
marked points so that it respects to the counter-clockwise cyclic
order. The totality of such
$(\varphi,\{(\tau^{(1)}_i,1)\},\{(\tau^{(0)}_i,0)\},
\{(\tau_i,t_i)\})$ is denoted by
$$
\widetilde{\mathcal M}^{\text{reg}}_{k_1,k_0;\ell}(L^{(1)},L^{(0)};p,q;B).
$$
We divide this space by the $\R$ action induced by the translation of
$\tau$ direction to obtain
$
{\mathcal M}^{\text{reg}}(L^{(1)},L^{(0)};p,q;B),
$
and
$
{\mathcal M}^{\text{reg}}_{k_1,k_0;\ell}(L^{(1)},L^{(0)};p,q;B).
$
Finally we compactify them to obtain
$
{\mathcal M}(L^{(1)},L^{(0)};p,q;B),
$
and
$
{\mathcal M}_{k_1,k_0;\ell}(L^{(1)},L^{(0)};p,q;B).
$
\end{defn}

See Definition 3.7.24 \cite{fooo06} (= Definition 12.24 \cite{fooo06pre})  (the case $\ell=0$) and 
Subsection 3.8.8 \cite{fooo06} (= Section 13.8 \cite{fooo06pre}) for the detail.

\begin{rem}
In \cite{fooo06} we defined ${\mathcal
M}_{k_1,k_0}(L^{(1)},L^{(0)};[\ell_p,w_1],[\ell_q,w_2])$. The choice
of $[w_1]$ and $B$ uniquely determines $[w_2]$ by the relation
$[w_1] \# B = [w_2]$, but there could be more than one element $B
\in \pi_2(L^{(1)},L^{(0)};p,q)$ satisfying $[w_1] \# B = [w_2]$.
This is because the equivalence class $[\ell_p,w]$ is not the
homotopy class but the equivalence class of a weaker relation. But
the number of such classes $B$ for which ${\mathcal
M}(L^{(1)},L^{(0)};p,q;B) \neq \emptyset$ is finite by Gromov's
compactness. Therefore ${\mathcal
M}_{k_1,k_0}(L^{(1)},L^{(0)};[\ell_p,w_1],[\ell_q,w_2])$ is a finite
union of ${\mathcal M}(L^{(1)},L^{(0)};p,q;B)$ with $B$ satisfying
$[w_1] \# B = [w_2]$.
\end{rem}
We define the evaluation map
$$
ev = (ev^{\text{int}},ev^{(1)},ev^{(0)}): 
{\mathcal M}_{k_1,k_0;\ell}^{\text {reg}}
(L^{(1)},L^{(0)};p,q;B) \to X^{\ell} \times
(L(u))^{k_1} \times (L(u))^{k_0},
$$
by
\begin{equation}\label{evaluationFH}
\aligned
ev_i^{(0)}(\varphi,\{(\tau^{(1)}_i,1)\},\{(\tau^{(0)}_i,0)\},
\{(\tau_i,t_i)\})
&= \varphi((\tau^{(0)}_i,0)), \\
ev_i^{(1)}(\varphi,\{(\tau^{(1)}_i,1)\},\{(\tau^{(0)}_i,0)\},
\{(\tau_i,t_i)\})
&= \varphi((\tau^{(1)}_i,1)), \\
ev_i^{\text{int}}(\varphi,\{(\tau^{(1)}_i,1)\},\{(\tau^{(0)}_i,0)\},
\{(\tau_i,t_i)\})
&= \varphi((\tau_i,t_i)).
\endaligned
\end{equation}
Proposition 3.7.26 \cite{fooo06} (= Proposition 12.26 
\cite{fooo06pre}) shows that the evaluation maps 
extend to the compactification 
$$
ev = (ev^{\text{int}},ev^{(1)},ev^{(0)}): 
{\mathcal M}_{k_1,k_0;\ell}
(L^{(1)},L^{(0)};p,q;B) \to X^{\ell} \times
(L(u))^{k_1} \times (L(u))^{k_0}
$$
so that 
they are weakly submersive.
 
We have diffeomorphisms $L(u) \cong L^{(0)}$ and
$L(u) \cong L^{(1)}$. (The former is the identity and the latter is $\psi$.)

\begin{lem}\label{boundaryN}
${\mathcal M}_{k_1,k_0;\ell}(L^{(1)},L^{(0)};p,q;B)$ has an oriented
Kuranishi structure with corners. 
The evaluation map $(\ref{evaluationFH})$ 
extends to
${\mathcal M}_{k_1,k_0;\ell}(L^{(1)},L^{(0)};p,q;B)$ so that 
it is smooth strongly continuous and weakly submersive. 
(See Definition {\rm A1.13} \cite{fooo06} for the definitions.)
Its boundary is isomorphic to the
union of the following three kinds of fiber products as spaces with
Kuranishi structure.
\begin{enumerate}
\item
$${\mathcal M}_{k'_1,k'_0;\ell'}(L^{(1)},L^{(0)};p,r;B')
\times
{\mathcal M}_{k''_1,k''_0;\ell''}(L^{(1)},L^{(0)};r,q;B'')$$
where $k'_j + k''_j = k_j$, $\ell' + \ell'' = \ell$,
$B' \# B'' = B$. The product is the direct product.
\item
$${\mathcal M}_{k'_1+1;\ell'}(L(u);\beta')
\,\,{}_{ev_0}\times_{ev^{(1)}_i} \,\,
{\mathcal M}_{k''_1,k_0;\ell''}(L^{(1)},L^{(0)};p,q;B'').$$
Here $\beta' \in \pi_2(X,L^{(1)}) \cong \pi_2(X,L(u))$,
$k'_1 + k''_1 = k_1+1$, $\ell' + \ell'' = \ell$,
$\beta' \# B'' = B$. The fiber product is taken over
$L^{(1)} \cong L(u)$ by using $ev_0: {\mathcal M}_{\ell';k'_1+1}(L(u);\beta')
\to L(u)$ and
$ev^{(1)}_i: {\mathcal M}_{k''_1,k_0;\ell''}(L^{(1)},L^{(0)};p,q;B'')
\to L^{(1)}$. Here $i=1,\ldots,k''_1$.
\item
$$
{\mathcal M}_{k_1,k'_0;\ell'}(L^{(1)},L^{(0)};p,q;B')
\,\,{}_{ev^{(0)}_i}\times_{ev_0} \,\,
{\mathcal M}_{k''_0+1;\ell''}(L(u);\beta'') .$$
Here $\beta'' \in \pi_2(X,L^{(0)}) \cong \pi_2(X,L(u))$,
$k'_0 + k''_0 = k_0+1$, $\ell' + \ell'' = \ell$,
$B' \# \beta'' = B$. The fiber product is taken over
$L^{(0)} \cong L(u)$ by using $ev_0: {\mathcal M}_{k''_0+1;\ell''}(L(u);\beta'')
\to L(u)$ and
$ev^{(0)}_i: {\mathcal M}_{k_1,k'_0;\ell'}(L^{(1)},L^{(0)};p,q;B')
\to L^{(0)}$.
\end{enumerate}
\end{lem}
See Proposition 3.7.26 \cite {fooo06}. 
Lemma \ref{boundaryN} is proved in the same way as in Subsection 7.1.4 \cite{fooo06} (= Section 29.4 \cite{fooo06pre}).
\begin{defn}
We next take $\text{\bf p} \in Map(\ell,\underline B)$ and define
\begin{equation}\label{Mwithpbimodule}
{\mathcal M}_{k_1,k_0;\ell}(L^{(1)},L^{(0)};p,q;B;\text{\bf p}) =
{\mathcal M}_{k_1,k_0;\ell}(L^{(1)},L^{(0)};p,q;B) \,\,
{}_{ev^{\text{int}}}\times \prod_{i=1}^{\ell} D_{\text{\bf p}(i)}.
\end{equation}
It is a space with oriented Kuranishi structure with corners.
\par
We remark that Condition \ref{butukarazu} implies that if $p=q$ and
$B = B_0 = 0$, then the set ${\mathcal M}_{k_1,k_0;\ell}(L^{(1)},L^{(0)};p,p;B_0;\text{\bf p})$ with $\ell\ne 0$ is empty.
\end{defn}
\begin{lem}\label{bdrywithpn}
The boundary of
${\mathcal M}_{k_1,k_0;\ell}(L^{(1)},L^{(0)};p,q;B;\text{\bf p})$
is a union of the following three types of fiber products
as a space with Kuranishi structure.
\begin{enumerate}
\item
$${\mathcal M}_{k'_1,k'_0;\ell'}(L^{(1)},L^{(0)};p,r;B';\text{\bf p}_1)
\times {\mathcal M}_{k''_1,k''_0;\ell''}(L^{(1)},L^{(0)};r,q;B'';\text{\bf p}_2).$$
Here the notations are the same as in Lemma $\ref{boundaryN}$ $(1)$ and
\begin{equation}\label{p1p2}
(\text{\bf p}_1,\text{\bf p}_2)
= \text{\rm Split}((\mathbb L_1,\mathbb L_2),\text{\bf p})
\end{equation}
for some $(\mathbb L_1,\mathbb L_2) \in \text{\rm Shuff}(\ell)$.
\item
$${\mathcal M}_{k'_1+1;\ell'}(L(u);\beta';\text{\bf p}_1)
\,\,{}_{ev_0}\times_{ev^{(1)}_i} \,\, {\mathcal
M}_{k''_1,k_0;\ell''}(L^{(1)},L^{(0)};p,q;B'';\text{\bf p}_2).
$$
Here the notations are the same as in Lemma $\ref{boundaryN}$ $(2)$
and $(\ref{p1p2})$.
\item
$$
{\mathcal M}_{k_1,k'_0;\ell'}(L^{(1)},L^{(0)};p,q;B';\text{\bf p}_1)
\,\,{}_{ev^{(0)}_i}\times_{ev_0} \,\, {\mathcal
M}_{k''_0+1;\ell''}(L(u);\beta'';\text{\bf p}_2).
$$
Here the notations are the same as in Lemma $\ref{boundaryN}$ $(3)$
and $(\ref{p1p2})$. 
\end{enumerate}
\end{lem}
The proof is immediate from Lemma \ref{boundaryN}.
We remark that by our definition of evaluation map $ev^{\text{\rm int}}_i$
the homology classes $\beta'$, $\beta''$ in (2), (3) above are nonzero if $\ell', \ell'' \ne 0$.
\par
We now construct a virtual fundamental chain on
the moduli space (\ref{Mwithpbimodule}).
We remark that we already defined a system of multisections on
${\mathcal M}_{k+1;\ell}(L(u);\beta;\text{\bf p})$ in Lemma \ref{37.184}.
\begin{lem}\label{multisectionN}
For any $E_0$,  
there exists a system of multisections $(\ref{Mwithpbimodule})$
of the moduli space ${\mathcal M}_{k_1,k_0;\ell}(L^{(1)},L^{(0)};p,q;B;\text{\bf p})$ 
with $B \cap \omega < E_0$, 
which are compatible to one another and to the multisections
provided in Lemma $\ref{37.184}$ under the identification of the
boundaries given in Lemma $\ref{bdrywithpn}$.
\end{lem}
\begin{proof}
We construct multisections on the moduli space
(\ref{Mwithpbimodule}) by induction over 
$\int_{\beta}\omega$.
\par
We remark that the boundary condition for \eqref{CReqn-Jt} is not
$T^n$ equivariant anymore: while the boundary $L^{(0)} = L(u)$ is
$T^n$ invariant, $L^{(1)} = \psi(L(u))$ is not. So there is no way
to define a $T^n$-action on our moduli space (\ref{Mwithpbimodule}).
\par
However we remark that $ev_0$ in (2) and (3) of Lemma
\ref{bdrywithpn} are submersions after perturbation. This is a
consequence of (2) of Lemma \ref{37.184}. Moreover the fiber product
in (1) of  Lemma \ref{bdrywithpn} is actually a direct product.
Therefore the perturbation near the boundary at each step of the
induction is automatically transversal by the induction hypothesis.
Therefore we can extend the perturbation by the standard theory of
Kuranishi structure and multisection. This implies Lemma
\ref{multisectionN}.
\end{proof}
We are now ready to define Floer cohomology with bulk deformation
denoted by
$$
HF((L^{(1)},\mathfrak b,\psi_*(\mathfrak x)),(L^{(0)},\mathfrak b,\mathfrak x);\Lambda_0).
$$
\par
Let us use the notation of Definition \ref{rhoyx}. We have a
representation $\rho: \pi_1(L(u)) \to \C \setminus\{0\}$. We choose
a flat $\C$-bundle $(\mathcal L,\nabla_{\rho})$ whose holonomy
representation is $\rho$. It determines flat $\C$ bundles on
$L^{(0)}$, $L^{(1)}$, which we denote by $\mathcal L^{(0)}$ and
$\mathcal L^{(1)}$, respectively. The fiber of $\mathcal L^{(j)}$ at
$p$ is denoted by $\mathcal L^{(j)}_p$.
\par
\begin{defn}
We define
$$
CF((L^{(1)},\rho),(L^{(0)},\rho);\Lambda_0)
=
\bigoplus_{p\in L^{(1)}\cap L^{(0)}} Hom(\mathcal L^{(1)}_p,\mathcal L^{(0)}_p)  \otimes_{\C} \Lambda_0.
$$
With elements $p \in L^{(1)} \cap L^{(0)}$ equipped with the degree
$0$ or $1$ according to the parity of the Maslov index, it becomes a
$\Z_2$-graded free $\Lambda_0$-module.
\end{defn}
\begin{rem}\label{rem814}
Actually we need to fix $E_0$ and construct $\mathfrak
q_{\beta;k,\ell}$ for $\beta\cap\omega < E_0$. 
We then take inductive limit.
We omit the detail and refer Proposition 7.4.17 \cite{fooo06}.
For the application 
we can use $A_{n,K}$ structure in place of $A_{\infty}$ structure so the
process to go to inductive limit is not necessary.
See also the end of the proof of Proposition \ref{calcFloer}.
\end{rem}
We are now ready to define an operator $\mathfrak r$, following 
Section 3.8 \cite{fooo06} (= Section
13.8 \cite{fooo06pre}). We first define a map
$$
\text{\rm Comp}: \pi_2(L^{(1)},L^{(0)};p,q) \times
Hom(\mathcal L^{(1)}_p,\mathcal L^{(0)}_p)
\to Hom(\mathcal L^{(1)}_q,\mathcal L^{(0)}_q).
$$
\par
Let $B = [\varphi] \in \pi_2(L^{(1)},L^{(0)};p,q)$ and 
$\sigma \in
Hom(\mathcal L^{(1)}_p,\mathcal L^{(0)}_p)$. The restriction of $\varphi$ to
$\R \times \{j\}$, $\tau \mapsto \varphi(\tau,j)$
defines a path $\partial_jB$ 
joining $p$ to $q$ in $L^{(j)}$  for each $j=0,\, 1$. 
Let
\begin{equation}\label{defpal}
\text{\rm Pal}_{\partial_jB}: \mathcal L^{(j)}_p \to \mathcal
L^{(j)}_q
\end{equation}
be the parallel transport along this path with respect to the flat
connection $\nabla_{\rho}$. Since $\nabla_{\rho}$ is flat, this is independent of
the choice of the representative $\varphi$ but depends only on $B$.
We define
\begin{equation}\label{holonomyinduce}
\text{\rm Comp}(B,\sigma) = \text{\rm Pal}_{\partial_0B} \circ
\sigma \circ\text{\rm Pal}_{\partial_1B}^{-1}.
\end{equation}
\begin{lem}\label{comphomo}
Let $B \in \pi_2(L^{(1)},L^{(0)};p,q)$, $B' \in
\pi_2(L^{(1)},L^{(0)};q,r)$ and $\beta_j \in \pi_2(X,L^{(j)})$,
$\sigma \in Hom(\mathcal L^{(1)}_p,\mathcal L^{(0)}_p)$. Then we
have
$$
\aligned \text{\rm Comp}(B\# B',\sigma)
&= \text{\rm Comp}(B',\text{\rm Comp}(B,\sigma)), \\
\text{\rm Comp}(\beta_0\# B,\sigma)
&= \rho(\partial\beta_0)\text{\rm Comp}(B,\sigma),\\
\text{\rm Comp}(B\# \beta_1,\sigma) &= \rho(\partial\beta_1)\text{\rm
Comp}(B,\sigma).
\endaligned
$$
\end{lem}
The proof is easy and so omitted.
\begin{defn}
Let $B \in \pi_2(L^{(1)},L^{(0)};p,q)$, $\text{\bf p} \in
Map(\ell,\underline B)$ and let $h^{(j)}_i$ ($i = 1,\ldots,k_j$) be
differential forms on $L^{(j)}$. We define
\begin{equation}\label{rdRdef}
\aligned
&\mathfrak r_{\rho,k_1,k_0;\ell;B}(D(\text{\bf
p});
h^{(1)}_1,\ldots,h^{(1)}_{k_1};\sigma;h^{(0)}_1,\ldots,h^{(0)}_{k_0})\\
&= \frac{1}{\ell!}T^{\omega \cap B/2\pi}\text{\rm Comp}(B,\sigma)
\int_{{\mathcal M}_{k_1,k_0;\ell}(L^{(1)},L^{(0)};p,q;B;\text{\bf
p})}
ev^{(1) *}h^{(1)} \wedge ev^{(0) *}h^{(0)} \\
& \in Hom(\mathcal L^{(1)}_q,\mathcal L^{(0)}_q)  \otimes_{\C} \Lambda_0.
\endaligned
\end{equation}
Here
$$
h^{(j)} = h^{(j)}_1\times \cdots \times h^{(j)}_{k_j}
$$
is a differential form on $(L^{(j)})^{k_j}$.
\end{defn}
Then Gromov's compactness theorem implies that
$$
\mathfrak r_{\rho,k_1,k_0;\ell} =\sum_{B}\mathfrak r_{\rho,k_1,k_0;\ell;B}
$$
converges in non-Archimedean topology 
and so defines
$$
\aligned
\mathfrak r_{\rho,k_1,k_0;\ell}
: E_{\ell}\mathcal A(\Lambda_+)[2] &\otimes B_{k_1}((\Omega(L^{(1)}) \,\widehat{\otimes}\, \Lambda_0)[1]) \\
&\otimes
CF((L^{(1)},\rho),(L^{(0)},\rho);\Lambda_0)
\otimes B_{k_0}((\Omega(L^{(0)}) \,\widehat{\otimes}\, \Lambda_0)[1]) \\
& \longrightarrow CF((L^{(1)},\rho),(L^{(0)},\rho);\Lambda_0).
\endaligned$$
The following is a slight modification of 
Theorem 3.8.71 \cite{fooo06} (= Theorem 13.71 \cite{fooo06pre}).
\begin{prop}\label{rmainformulaprop}
Let $\text{\bf y} \in \mathcal A(\Lambda_+)[2]$, $\text{\bf x} \in
B_{k_1}((\Omega(L^{(1)}) \,\widehat{\otimes}\, \Lambda_0)[1])$,  and let $\text{\bf z} \in
B_{k_0}((\Omega(L^{(0)})\,\widehat{\otimes}\, \Lambda_0)[1])$, $v \in
CF((L^{(1)},\rho),(L^{(0)},\rho);\Lambda_0) $. Then we have
\begin{equation}\label{rmainformula}
\aligned
0 =
&
\sum_{c_1,c_2} (-1)^{\deg \text{\bf
y}^{(2;2)}_{c_1}\deg'\text{\bf x}^{(3;1)}_{c_2}  + \deg'\text{\bf
x}^{(3;1)}_{c_2} + \deg \text{\bf
y}^{(2;1)}_{c_1}}\\
&\hskip1cm \mathfrak r_{\rho}( \text{\bf y}^{(2;1)}_{c_1} ;
(
\text{\bf x}^{(3;1)}_{c_2}
\otimes
\mathfrak q_{\rho}(\text{\bf y}^{(2;2)}_{c_1}; \text{\bf x}^{(3;2)}_{c_2} )
\otimes
\text{\bf x}^{(3;3)}_{c_2} )
\otimes v \otimes \text{\bf z})
\\
&+
\sum_{c_1,c_2,c_3} (-1)^{\deg \text{\bf
y}^{(2;2)}_{c_1}\deg'\text{\bf x}^{(2;1)}_{c_2}  + \deg'\text{\bf
x}^{(2;1)}_{c_2} + \deg \text{\bf y}^{(2;1)}_{c_1}}\\
&\hskip1cm \mathfrak r_{\rho}( \text{\bf y}^{(2;1)}_{c_1} ;
\text{\bf x}^{(2;1)}_{c_2}
\otimes
\mathfrak r_{\rho}(\text{\bf y}^{(2;2)}_{c_1} ; \text{\bf x}^{(2;2)}_{c_2}
\otimes v \otimes \text{\bf z}^{(2;1)}_{c_3}))
\otimes \text{\bf z}^{(2;2)}_{c_3})\\
&+
\sum_{c_1,c_3} (-1)^{(\deg \text{\bf
y}^{(2;2)}_{c_1}+1)(\deg'\text{\bf x}+\deg' v +
\deg'\text{\bf z}^{(3;1)}_{c_3}) + \deg \text{\bf y}^{(2;1)}_{c_1}}
\\
&\hskip1cm \mathfrak r_{\rho}( \text{\bf y}^{(2;1)}_{c_1} ;
(
\text{\bf x}
\otimes v \otimes(\text{\bf z}^{(3;1)}_{c_3}\otimes
\mathfrak q_{\rho}(\text{\bf y}^{(2;2)}_{c_1}; \text{\bf z}^{(3;2)}_{c_3} )
\otimes
\text{\bf z}^{(3;3)}_{c_3})).
\endaligned\end{equation}
\end{prop}
\begin{proof}
The 1st, 2nd and 3rd terms correspond to (2), (1) and (3) of Lemma
\ref{bdrywithpn} respectively. The associated weights of symplectic area
behave correctly under the composition rules in Lemma
\ref{comphomo}. The proposition follows from Stokes' formula. (We
do not discuss sign here, since the sign will be trivial for the
case of our interest where the degrees of ambient cohomology classes
are even and the degrees of the cohomology classes of
Lagrangian submanifold are odd.)
\end{proof}
\begin{lem}\label{unitality}
Let $\text{\bf y} \in \mathcal {A}(\Lambda_+)[2]$, 
$x_i \in \Omega(L^{(1)})
\,\widehat{\otimes}\, \Lambda_0[1]$,
$z_i \in \Omega(L^{(0)})
\,\widehat{\otimes}\, \Lambda_0[1]$ 
and $v \in
CF((L^{(1)},\rho),(L^{(0)},\rho);\Lambda_0)$. 
\par
If $p=q$, $B=B_0=0 \in 
\pi_2(L^{(1)},L^{(0)};p,p)$ and 
$\ell =0$, then we have 
\begin{equation}\label{unitalityr2}
\mathfrak r_{\rho,1,0;\ell=0;B=0}(\text{\bf y} ; 1 \otimes v)
=(-1)^{\deg v}
\mathfrak r_{\rho,0,1;\ell=0;B=0}(\text{\bf y} ; v \otimes 1) = v. 
\end{equation}
Here $1$ is the degree $0$ form $1$ which 
represents the Poincar\'e dual to the fundamental cycle of $L^{(0)}$, (or $L^{(1)}$).   
Otherwise, we have
\begin{equation}\label{unitalityr1}
\aligned
\mathfrak r_{\rho,k_1,k_0;\ell}(\text{\bf y} ; 
x_1 \otimes \dots \otimes 1 \otimes \dots \otimes x_{k_1-1} \otimes v \otimes z_1 \otimes \dots \otimes z_{k_0}) 
& = 0, \\
\mathfrak r_{\rho,k_1,k_0;\ell}(\text{\bf y} ; 
x_1 \otimes \dots \otimes x_{k_1} \otimes v \otimes z_1 \otimes \dots \otimes 1 \otimes \dots \otimes z_{k_0-1})
& = 0.
\endaligned
\end{equation}
\end{lem}
\begin{proof}
This is an immediate consequence of the definition. 
See Theorems 3.7.21, 3.8.71 \cite{fooo06} ( = Theorems 12.21, 13.71 \cite{fooo06pre}) and Section 12 \cite{fooo08}. 
\end{proof}
\begin{rem}
In this paper we do not consider bulk deformation by 
the Poincar\'e dual $PD([X]) \in H^0(X;\Z)$ to the fundamental cycle of $X$. 
(See the beginning of Section \ref{sec:Potbul}.) 
If we incorporate it into the story, we further obtain 
the following 
equality besides (\ref{unitalityr2}):
\be\label{unitalityr3}
\mathfrak r_{\rho,1,0;\ell=1;B=0}(PD([X]) ; 1 \otimes v)
=(-1)^{\deg v}
\mathfrak r_{\rho,0,1;\ell=1;B=0}(PD([X]) ; v \otimes 1) = v, 
\ee
for the case $p=q$, 
$B=B_0=0 \in \pi_2(L^{(1)},L^{(0)};p,p)$ and 
$\ell=1$, ${\bf y}=PD([X])$. 
\end{rem}

We define $T^n$ action on $L^{(1)}$ by 
$$
g\cdot x = \psi (g \psi^{-1}(x)),
$$
where $g \psi^{-1}(x)$ is defined by $T^n$ action on 
$L^{(0)}$ induced by the  
$T^n$ action on $X$. (Note  $L^{(1)} \subset X$ is not necessarily $T^n$ invariant 
under the $T^n$ action on $X$, since $\psi$ may not be $T^n$ equivariant.)
\par
We identify $H(L^{(1)},\C)$ with the set of $T^n$ invariant forms, with respect to the 
above action. Now by restricting (\ref{rdRdef}) we obtain:
$$
\aligned
\mathfrak r^{{\text{\rm can}}}_{\rho,k_1,k_0;\ell}
: E_{\ell}\mathcal A(\Lambda_+)[2] &\otimes B_{k_1}(H(L^{(1)};\Lambda_0)[1]) \\
&\otimes
CF((L^{(1)},\rho),(L^{(0)},\rho);\Lambda_0)
\otimes B_{k_0}(H(L^{(0)};\Lambda_0)[1]) \\
& \longrightarrow CF((L^{(1)},\rho),(L^{(0)},\rho);\Lambda_0).
\endaligned$$
By definition, (\ref{rmainformula}) and Lemma 
\ref{unitality} hold when
$\mathfrak r$ and $\mathfrak q$ are replaced by 
$\mathfrak r^{{\text{\rm can}}}$ and $\mathfrak q^{{\text{\rm can}}}$.
\begin{defn}
Let $\mathfrak b \in \mathcal A(\Lambda_+)$, $\mathfrak x \in H^1(L(u),\Lambda_0)$.
We use the notations of Definition \ref{rhoyx} and define
$$
\delta^{\mathfrak b,\mathfrak x}: CF((L^{(1)},\rho),(L^{(0)},\rho);\Lambda_0)
\to CF((L^{(1)},\rho),(L^{(0)},\rho);\Lambda_0)
$$
by
$$
\delta^{\mathfrak b,\mathfrak x}(v)
= \mathfrak r^{{\text{\rm can}}}_{\rho}(e^{\mathfrak b}; e^{\psi_* (\mathfrak x_+)} \otimes v
\otimes e^{\mathfrak x_+}).
$$
\end{defn}
By taking a harmonic representative of $\mathfrak x_+$ we also have
$$
\delta^{\mathfrak b,\mathfrak x}(v)
= \mathfrak r_{\rho}(e^{\mathfrak b}; e^{\psi_* (\mathfrak x_+)} \otimes v
\otimes e^{\mathfrak x_+}).
$$
\begin{lem}\label{boundaryfloer2}
$$
\delta^{\mathfrak b,\mathfrak x} \circ \delta^{\mathfrak b,\mathfrak x} = 0.
$$
\end{lem}
\begin{proof}
We remark that
$\Delta e^{\mathfrak b} = e^{\mathfrak b} \otimes e^{\mathfrak b}$
and
$\Delta e^{\mathfrak x_+} = e^{\mathfrak x_+} \otimes e^{\mathfrak x_+}$.
Therefore Proposition \ref{rmainformulaprop} implies
$$
\aligned
0 = & \mathfrak r_{\rho}^{{\text{\rm can}}}(e^{\mathfrak b} ; 
e^{\psi_* (\mathfrak x_+)}
\otimes \mathfrak r_{\rho}^{{\text{\rm can}}}(e^{\mathfrak b} ; 
e^{\psi_* (\mathfrak x_+)}
\otimes v
\otimes e^{\mathfrak x_+})\otimes e^{\mathfrak x_+}) \\
& +
\mathfrak r_{\rho}^{{\text{\rm can}}}(e^{\mathfrak b} ; 
e^{\psi_* (\mathfrak x_+)}\otimes
\mathfrak q_{\rho}^{{\text{\rm can}}}(e^{\mathfrak b} ; e^{\psi_* (\mathfrak x_+)})
\otimes e^{\psi_* (\mathfrak x_+)}\otimes v
\otimes e^{\mathfrak x_+}) \\
&+ (-1)^{\deg  v +1} \mathfrak r_{\rho}^{{\text{\rm can}}}(e^{\mathfrak b}; 
e^{\psi_* (\mathfrak x_+)}\otimes v
\otimes e^{\mathfrak x_+}\otimes
\mathfrak q_{\rho}^{{\text{\rm can}}}(e^{\mathfrak b} ; e^{\mathfrak x_+})
\otimes e^{\mathfrak x_+}).
\endaligned$$
Since $\mathfrak q_{\rho}^{{\text{\rm can}}}(e^{\mathfrak b} ; e^{\mathfrak x_+}) $ is
a (harmonic) form of degree $0$, Lemma \ref{unitality},  especially (\ref{unitalityr1}), implies that the second
and the third terms almost vanish. 
All other non-vanishing terms  
in the second and the third lines come from 
(\ref{unitalityr2}).
Then a calculation similar to the proofs of 
Lemma 12.7 \cite{fooo08} and 
Proposition 3.7.17 \cite{fooo06} ( = Proposition 12.17 
\cite{fooo06pre}) shows that sum of the second and third lines 
are equal to   
\be\label{differencePO}
\left(-\mathfrak{PO}_{\rho}(\mathfrak b, \psi_* (\mathfrak x_+)) + 
\mathfrak{PO}_{\rho}(\mathfrak b, \mathfrak x_+)
\right) v.
\ee
By invariance of the potential function  
(see Theorem B (B.3) \cite{fooo06} ( = Theorem B (B.3) \cite{fooo06pre}) we obtain 
$\mathfrak{PO}_{\rho}(\mathfrak b, \psi_* (\mathfrak x_+)) =
\mathfrak{PO}_{\rho}(\mathfrak b, \mathfrak x_+)$.
This proves the lemma.
\end{proof}
\begin{rem} Such a 
cancellation mechanism related to  
(\ref{differencePO}) was observed in Example 7.4 
\cite{fooo00} for the case $\mathfrak b=0 (\ell =0), \rho=1$. 
\end{rem}
\begin{defn}
$$
HF((L^{(1)},\mathfrak b,\psi_*(\mathfrak x)),(L^{(0)},\mathfrak b,\mathfrak x);\Lambda_0)
= \frac{\text{\rm Ker}\,\,\delta^{\mathfrak b,\mathfrak x}}
{\text{\rm Im}\,\,\delta^{\mathfrak b,\mathfrak x}}.
$$
\end{defn}

We recall that we are considering the Hamiltonian isotopic pair
$$
L^{(0)} = L(u), \quad L^{(1)} = \psi(L(u)).
$$
For this case, we prove

\begin{prop}\label{calcFloer}
We have
$$
HF((L^{(1)},\mathfrak b,\psi_*(\mathfrak x)),(L^{(0)},\mathfrak b,\mathfrak x);\Lambda)
\cong HF((L(u),\mathfrak b,\mathfrak x),(L(u),\mathfrak b,\mathfrak x);\Lambda).
$$
\end{prop}
\begin{rem}
We use $\Lambda$ coefficients instead of $\Lambda_0$ coefficients
in Proposition \ref{calcFloer}.
\end{rem}
\begin{rem}\label{824}
In the proof of Proposition \ref{calcFloer}, we need to choose a system of 
multisections of various moduli spaces. In doing so, we need to fix 
the energy level $E_0$ and restrict the construction 
to the moduli spaces with energy smaller than $E_0$.
\par
In the situation of the proof of Proposition \ref{calcFloer} this point is slightly more nontrivial than the similar 
problem mentioned in Remarks \ref{rem610}, \ref{rem814}, since we need to work 
with Novikov field $\Lambda$ instead of $\Lambda_0$.
\par
To simplify the description we ignore this problem for a while and will explain it 
at the end of the proof of Proposition \ref{calcFloer}.
\end{rem}
\begin{proof} 
We can prove
Proposition \ref{calcFloer} by the same way as in 
Sections 3.8, 5.3, 7.4 of \cite{fooo06} (= Sections 13, 22, 32 of \cite{fooo06pre}).
We will give the detail of the proof using de Rham theory, for completeness. 
The rest of this section is almost occupied with 
the proof of Proposition \ref{calcFloer}. 
\par\smallskip
Firstly we will define a chain map 
$\mathfrak f:  \Omega(L(u)) \widehat{\otimes} \Lambda
\to CF((L^{(1)},\rho),(L^{(0)},\rho);\Lambda)
$.
\par
Let $\psi_t$ be a Hamiltonian isotopy such that
$\psi_0$ is the identity and $\psi_1$ is $\psi$.
We put $L^{(t)} = \psi_t(L(u))$.
Let $\chi: \R \to [0,1]$ be a smooth function such that
$$
\chi(\tau) =
\begin{cases}
0 & \text{$\tau \le 0$}, \\
1 &\text{$\tau \ge 1$}.
\end{cases}
$$
We choose a two-parameter family 
$\{J_{\tau,t}\}_{\tau,t}$ of compatible almost
complex structures defined by
$$
J_{\tau,t} = \psi_{t\chi(\tau) *} J.
$$
Then it satisfies the following:
\begin{enumerate}
\item $J_{\tau,t} = J_t$ for $\tau \ge 1$.
\item $J_{\tau,t} = J$ for $\tau \le 0$.
\item $J_{\tau,1} = \psi_{\chi(\tau) *}J$.
\item $J_{\tau,0} = J$.
\end{enumerate}

Let $p \in L^{(0)}\cap L^{(1)}$. We consider maps
$
\varphi: \R \times [0,1] \to X
$
such that
\begin{enumerate}
\item
$\lim_{\tau \to +\infty} \varphi(\tau,t) = p$.
\item
$\lim_{\tau \to -\infty} \varphi(\tau,t)$ converges
to a point in $L^{(0)}$ independent of $t$.
\item
$\varphi(\tau,0) \in L^{(0)}$, $\varphi(\tau,1) \in L^{(\chi(\tau))}$.
\end{enumerate}
We denote by $\pi_2(L^{(1)},L^{(0)};*,p)$ the set of homotopy
classes of such maps.
There are obvious maps
\begin{equation}\label{sharp2}
\aligned
& \pi_2(L^{(1)},L^{(0)};*,p) \times \pi_2(L^{(1)},L^{(0)};p,q)
\to \pi_2(L^{(1)},L^{(0)};*,q), \\
& \pi_2(X,L^{(1)}) \times \pi_2(L^{(1)},L^{(0)};*,p)
\to \pi_2(L^{(1)},L^{(0)};*,p), \\
&  \pi_2(L^{(1)},L^{(0)};*,p)  \times \pi_2(X,L^{(0)})
\to \pi_2(L^{(1)},L^{(0)};*,p).
\endaligned
\end{equation}
(We here use the fact that the action of $\pi_1(L^{(i)})$ on $\pi_2(X,L^{(i)})$
is trivial.) We denote (\ref{sharp2}) by $\#$.

\begin{defn}
We consider the moduli space of maps
satisfying (1) - (3) above and of homotopy class
$C_+ \in \pi_2(L^{(0)},L^{(1)};*,p)$ and satisfying the following
equation:
\begin{equation}\label{tautCR}
\frac{\partial\varphi}{\partial \tau}
+ J_{\tau,t} \left(\frac{\partial\varphi}{\partial t}\right) = 0.
\end{equation}
We denote it by
$$
{\mathcal M}^{\text{reg}}(L^{(1)},L^{(0)};*,p;C_+).
$$
We also consider the moduli spaces of maps with interior and boundary marked points and
their compactifications. We then get the moduli space
$$
{\mathcal M}_{k_1,k_0;\ell}(L^{(1)},L^{(0)};*,p;C_+).
$$
\end{defn}
We remark that we do not divide by $\R$ action since (\ref{tautCR})
is not invariant under the translation.
We define evaluation maps
$$
ev = (ev^{\text{int}},ev^{(1)},ev^{(0)}): {\mathcal
M}^{\text{reg}}_{k_1,k_0;\ell}(L^{(1)},L^{(0)};*,p;C_+) \to X^{\ell} \times
(L^{(0)})^{k_0+k_1},
$$
in a similar way as (\ref{evaluationFH}):
\begin{equation}\label{evaluationFH2}
\aligned
ev_i^{(0)}(\varphi,\{(\tau^{(1)}_i,1)\},\{(\tau^{(0)}_i,0)\},
\{(\tau_i,t_i)\})
&= \varphi((\tau^{(0)}_i,0)), \\
ev_i^{(1)}(\varphi,\{(\tau^{(1)}_i,1)\},\{(\tau^{(0)}_i,0)\},
\{(\tau_i,t_i)\})
&= \psi^{-1}_{\chi(\tau_i)}(\varphi((\tau^{(1)}_i,1))), \\
ev_i^{\text{int}}(\varphi,\{(\tau^{(1)}_i,1)\},\{(\tau^{(0)}_i,0)\},
\{(\tau_i,t_i)\})
&= \psi_{t_i\chi(\tau_i)}^{-1}(\varphi((\tau_i,t_i))),
\endaligned
\end{equation}  
and extend to its compactification 
$$
ev = (ev^{\text{int}},ev^{(1)},ev^{(0)}): {\mathcal
M}_{k_1,k_0;\ell}(L^{(1)},L^{(0)};*,p;C_+) \to X^{\ell} \times
(L^{(0)})^{k_0+k_1}.
$$
Moreover
there is another evaluation map
$$
ev_{-\infty}: {\mathcal M}_{k_1,k_0;\ell}(L^{(1)},L^{(0)};*,p;C_+)
\to L(u)
$$
defined by
$$
ev_{-\infty}(\varphi) = \lim_{\tau\to -\infty}\varphi(\tau,t).
$$
Using fiber product with the cycle $D(\text{\bf p})$
we define ${\mathcal M}_{k_1,k_0;\ell}(L^{(1)},L^{(0)};*,p;C_+;\text{\bf p})$ in the same way as above.
\begin{lem}\label{2bdrywithpn}
The space ${\mathcal M}_{k_1,k_0;\ell}(L^{(1)},L^{(0)};*,p;C_+;\text{\bf p})$
has an oriented Kuranishi structure with boundary.
Its boundary is a union of the following four types of fiber products
as the space with Kuranishi structure.
\begin{enumerate}
\item
$${\mathcal M}_{k'_1,k'_0;\ell'}(L^{(1)},L^{(0)};*,q;C_+';\text{\bf p}_1)
\times
{\mathcal M}_{k''_1,k''_0;\ell''}(L^{(1)},L^{(0)};q,p;B'';\text{\bf p}_2).$$
Here the notations are the same as in Lemma $\ref{boundaryN}$ $(1)$
and $(\ref{p1p2})$.
\item
$${\mathcal M}_{k'_1+1;\ell'}(L(u);\beta';\text{\bf p}_1)
\,\,{}_{ev_0}\times_{ev^{(1)}_i} \,\,
{\mathcal M}_{k''_1,k_0;\ell''}(L^{(1)},L^{(0)};*,p;C_+'';\text{\bf p}_2).$$
Here the notations are the same as in Lemma $\ref{boundaryN}$ $(2)$
and $(\ref{p1p2})$.
\item
$$
{\mathcal M}_{k_1,k'_0;\ell'}(L^{(1)},L^{(0)};*,p;C_+';\text{\bf p}_1)
\,\,{}_{ev^{(0)}_i}\times_{ev_0} \,\,
{\mathcal M}_{k''_0+1;\ell''}(L(u);\beta'';\text{\bf p}_2) .$$
Here the notations are the same as in Lemma $\ref{boundaryN}$ $(3)$
and $(\ref{p1p2})$.
\item
$$
{\mathcal M}_{k'_1+k'_0+1;\ell'}(L(u);\beta';\text{\bf p}_1)
{}_{ev_0}\times_{ev_{-\infty}} {\mathcal M}_{k''_1,k''_0;\ell'}(L^{(1)},L^{(0)};*,p;C_+'';\text{\bf p}_2),
$$
where $k'_j + k''_j = k_j$, $\ell' + \ell'' = \ell$,
$\beta' \# C_+'' = C_+$ and $(\ref{p1p2})$.
\end{enumerate}
\end{lem}
The proof is the same as one in Subsection 7.1.4 
\cite{fooo06} (= Section 29.4 \cite{fooo06pre}).
\par
\begin{lem}\label{fmulti}
There exists a system of multisections on
$${\mathcal M}_{k_1,k_0;\ell}(L^{(1)},L^{(0)};*,p;C_+;\text{\bf p})$$
so that it is compatible with one constructed before
at the boundaries described in Lemma $\ref{2bdrywithpn}$.
\end{lem}
\begin{proof}
We can still use the fact that $ev_0$ is a submersion on the
perturbed moduli space to perform the inductive construction of
multisection in the same way as the proof of Lemma \ref{multisectionN}.
\end{proof}
\par
For $C_+ \in \pi_2(L^{(1)},L^{(0)};*,p)$, we define $\rho(C_+)
\in Hom(\mathcal L^{(1)}_p,\mathcal L^{(0)}_p)$ by
\begin{equation}\label{parallel}
\rho(C_+) =  \text{\rm Pal}_{\partial_0C_+}
\circ\text{\rm Pal}_{\partial_1C_+}^{-1}.
\end{equation}
Here we use the notation of (\ref{defpal}).
\begin{lem}\label{comphomo2}
Let $C_+ \in \pi_2(L^{(1)},L^{(0)};*,p)$,
$B' \in \pi_2(L^{(1)},L^{(0)};p,q)$ and
$\beta_j \in \pi_2(X,L^{(j)})$.
Then we have
$$
\aligned \text{\rm Comp}(B',\rho(C_+)) & = \rho(C_+\#B'), \\
\rho(\beta_0\# C_+) = \rho(\partial\beta_0)\rho(C_+), \, &{}\quad
\rho(C_+\# \beta_1) = \rho(\partial\beta_1)\rho(C_+).
\endaligned
$$
\end{lem}
The proof is easy and is left to the reader.
\par
Now let $C_+ \in \pi_2(L^{(1)},L^{(0)};*,p)$, $\text{\bf p} \in
Map(\ell,\underline B)$ and let $h^{(j)}_i$ ($i = 1,\ldots,k_j$) be
differential forms on $L^{(j)}$ and $h$ also a differential form
on $L(u)$. We define
\begin{equation}\label{deffraf}
\aligned
&\mathfrak f_{k_1,k_0;\ell;C_+}(SD(\text{\bf p});h^{(1)}_1,\ldots,h^{(1)}_{k_1};
h;h^{(0)}_1,\ldots,h^{(0)}_{k_0})\\
&= \frac{1}{\ell!}\rho(C_+)
\int_{{\mathcal M}_{k_1,k_0;\ell}(L^{(1)},L^{(0)};*,p;C_+;\text{\bf p})}
ev^{(1) *}h^{(1)} \wedge ev_{-\infty}^* h \wedge ev^{(0) *}h^{(0)} \\
& \in Hom(\mathcal L^{(1)}_p,\mathcal L^{(0)}_p)  \otimes \Lambda.
\endaligned
\end{equation}
Here
$$
h^{(j)} = h^{(j)}_1\times \cdots \times h^{(j)}_{k_j}
$$
is a differential form on $(L^{(j)})^{k_j}$.
It induces
$$\aligned
\mathfrak f_{C_+}: B((\Omega(L^{(1)}) \,\widehat\otimes\, \Lambda_0)[1]) &\otimes (\Omega(L(u)) \,\widehat\otimes\, \Lambda)[1]
\otimes B((\Omega(L^{(0)}) \,\widehat\otimes\, \Lambda_0)[1]) \\
&\to \bigoplus_{p\in L^{(1)}\cap L^{(0)}}Hom(\mathcal L^{(1)}_p,\mathcal L^{(0)}_p) \otimes \Lambda.
\endaligned$$
\par
Now we define
$$
\mathfrak f:  \Omega(L(u)) \widehat{\otimes} \Lambda
\to CF((L^{(1)},\rho),(L^{(0)},\rho);\Lambda)
$$
by
\begin{equation}
\mathfrak f(h) =
\sum_{C_+} T^{\omega \cap C_+/2\pi}\mathfrak f_{C_+}(e^{\mathfrak b} ; e^{\psi_{*}(\mathfrak x_+)} \otimes h \otimes e^{\mathfrak x_+}).
\end{equation}
We remark that $\omega \cap C_+/2\pi$ may not be
positive in this case since (\ref{tautCR}) is $\tau$-dependent.
\par
The fact that the right hand side converges in
non-Archimedean topology follows from 
Gromov's compactness theorem.
\begin{lem}\label{fischain}
$\mathfrak f$ is a chain map.
\end{lem}
\begin{proof}
With Lemmata \ref{2bdrywithpn},  \ref{fmulti}, \ref{comphomo2},
the proof is similar to the proof of Proposition  \ref{rmainformula} and Lemmata
\ref{unitality}, \ref{boundaryfloer2}.
\end{proof}
\par\smallskip
Next we define a chain map of the opposite direction of $\mathfrak f$.
Let $p \in L^{(0)}\cap L^{(1)}$.
We consider maps
$
\varphi: \R \times [0,1] \to X
$
such that
\begin{enumerate}
\item
$\lim_{\tau \to -\infty} \varphi(\tau,t) = p$.
\item
$\lim_{\tau \to +\infty} \varphi(\tau,t)$ converges
to a point in $L^{(0)}$ and is independent of $t$.
\item
$\varphi(\tau,0) \in L^{(0)}$, $\varphi(\tau,1) \in L^{(\chi(-\tau))}$.
\end{enumerate}
We denote by $\pi_2(L^{(1)},L^{(0)};p,*)$ the set of homotopy
classes of such maps.
There are obvious maps
\begin{equation}
\aligned
& \pi_2(L^{(1)},L^{(0)};p,q) \times \pi_2(L^{(1)},L^{(0)};q,*)
\to \pi_2(L^{(1)},L^{(0)};p,*), \\
& \pi_2(X,L^{(1)}) \times \pi_2(L^{(1)},L^{(0)};p,*)
\to \pi_2(L^{(1)},L^{(0)};p,*), \\
&  \pi_2(L^{(1)},L^{(0)};p,*)  \times \pi_2(X,L^{(0)})
\to \pi_2(L^{(1)},L^{(0)};p,*).
\endaligned
\end{equation}
We denote them by $\#$.
\begin{defn}
We consider the moduli space of maps
satisfying (1) - (3) above and of homotopy class
$C_- \in \pi_2(L^{(1)},L^{(0)};p,*)$ and satisfying the following
equation:
\begin{equation}\label{tautCR2}
\frac{\partial\varphi}{\partial \tau}
+ J_{-\tau,t} \left(\frac{\partial\varphi}{\partial t}\right) = 0.
\end{equation}
We denote it by
$$
{\mathcal M}^{\text{reg}}(L^{(1)},L^{(0)};p,*;C_-).
$$
We include interior and boundary marked points and
compactify it. We then get the moduli space
$
{\mathcal M}_{k_1,k_0;\ell}(L^{(1)},L^{(0)};p,*;C_-).
$
\end{defn}
We can define the evaluation maps
$$
ev = (ev^+,ev^{(1)},ev^{(0)}): {\mathcal
M}_{k_1,k_0;\ell}(L^{(1)},L^{(0)};p,*;C_-) \to X^{\ell} \times
(L^{(0)})^{k_0+k_1},
$$
and
$$
ev_{+\infty}: {\mathcal M}_{k_1,k_0;\ell}(L^{(1)},L^{(0)};p,*;C_-)
\to L(u).
$$
Here $L^{(0)}=L(u)$ and
$$
ev_{+\infty}(\varphi) = \lim_{\tau\to +\infty}\varphi(\tau,t).
$$
Using $ev^+$, we take fiber product with $D(\text{\bf p})$ and obtain
$$ {\mathcal M}_{k_1,k_0;\ell}(L^{(1)},L^{(0)};p,*;C_-;\text{\bf
p}). $$
\begin{lem}\label{2bdrywithpn2}
The space ${\mathcal M}_{k_1,k_0;\ell}(L^{(1)},L^{(0)};p,*;C_-;\text{\bf p})$
has an oriented Kuranishi structure with boundary.
Its boundary is a union of the following four types of fiber product
as a space with Kuranishi structure.
\begin{enumerate}
\item
$${\mathcal M}_{k'_1,k'_0;\ell'}(L^{(1)},L^{(0)};p,q;B';\text{\bf p}_1)
\times
{\mathcal M}_{k''_1,k''_0;\ell''}(L^{(1)},L^{(0)};q,*;C_-'';\text{\bf p}_2).$$
Here the notations are the same as in Lemma $\ref{boundaryN}$ $(1)$
and $(\ref{p1p2})$.
\item
$${\mathcal M}_{k'_1+1;\ell'}(L(u);\beta';\text{\bf p}_1)
\,\,{}_{ev_0}\times_{ev^{(1)}_i} \,\,
{\mathcal M}_{k''_1,k_0;\ell''}(L^{(1)},L^{(0)};p,*;C_-'';\text{\bf p}_2).$$
Here the notations are the same as in Lemma $\ref{boundaryN}$ $(2)$
and $(\ref{p1p2})$.
\item
$$
{\mathcal M}_{k_1,k'_0;\ell'}(L^{(1)},L^{(0)};p,*;C_-';\text{\bf p}_1)
\,\,{}_{ev^{(0)}_i}\times_{ev_0} \,\,
{\mathcal M}_{k''_0+1;\ell''}(L(u);\beta'';\text{\bf p}_2) .$$
Here the notations are the same as in Lemma $\ref{boundaryN}$ $(3)$
and $(\ref{p1p2})$.
\item
$$
{\mathcal M}_{k'_1,k'_0;\ell'}(L^{(1)},L^{(0)};p,*;C_-';\text{\bf p}_1)
{}_{ev_{-\infty}}\times_{ev_0} {\mathcal M}_{k''_1+k''_0+1;\ell''}
(L(u);\beta'';\text{\bf p}_2),
$$
where $k'_j + k''_j = k_j$, $\ell' + \ell'' = \ell$,
$\beta' \# C_-'' = C_-$ and $(\ref{p1p2})$.
\end{enumerate}
\end{lem}
The proof is the same as one in Subsection 7.1.4 \cite{fooo06} 
( = Section 29.4 \cite{fooo06pre}).
\par
We define the map
$$
\text{\rm Comp}: \pi_2(L^{(1)},L^{(0)};p,*) \times
Hom(\mathcal L^{(1)}_p,\mathcal L^{(0)}_p)
\to \C
$$
as follows. Let $\sigma \in Hom(\mathcal L^{(1)}_p,\mathcal L^{(0)}_p)$
and $C_- \in \pi_2(L^{(1)},L^{(0)};p,*)$. Then
\begin{equation}\label{holonomyinduce2}
\text{\rm Comp}(C_-,\sigma) v
= \text{\rm Pal}_{\partial_0C_-}
\circ \sigma \circ\text{\rm Pal}_{\partial_1C_-}^{-1}(v),
\end{equation}
where $v \in \mathcal L_{\lim_{\tau\to+\infty}\varphi(\tau,t)}$ and
we use the notation of (\ref{defpal}).
\par
Let $\sigma \in Hom(\mathcal L^{(1)}_p,\mathcal L^{(0)}_p)$, $C_- \in \pi_2(L^{(1)},L^{(0)};q,*)$,
$B' \in \pi_2(L^{(1)},L^{(0)};p,q)$ and
$\beta_j \in \pi_2(X,L^{(j)})$.
Then we have
\begin{equation}\label{comp3}
\aligned
\text{\rm Comp}(B',\text{\rm Comp}(C_-,\sigma))
&= \text{\rm Comp}(B'\# C_-,\sigma), \\
\text{\rm Comp}(\beta_0\# C_-,\sigma)
&= \rho(\partial\beta_0)\text{\rm Comp}(C_-,\sigma), \\
\text{\rm Comp}(C_- \# \beta_1,\sigma)
&= \rho(\partial\beta_1)\text{\rm Comp}(C_-,\sigma).
\endaligned\end{equation}
Now let $C_- \in \pi_2(L^{(1)},L^{(0)};p,*)$, $\text{\bf p} \in
Map(\ell,\underline B)$ and $h^{(j)}_i$ ($i = 1,\ldots,k_j$) be
differential forms on $L^{(j)}$ and $\sigma \in Hom(\mathcal
L^{(1)}_p,\mathcal L^{(0)}_p)$. We will define an element
\begin{equation}\label{defg}
\mathfrak g_{\ell;k_1,k_0;C_-}(SD(\text{\bf p});h^{(1)}_1,\ldots,h^{(1)}_{k_1};\sigma;
h^{(0)}_1,\ldots,h^{(0)}_{k_0})
\in \Omega(L(u)) \otimes \Lambda.
\end{equation}
We will define it as
\begin{equation}\label{deffraf4}
\aligned
&\mathfrak g_{\ell;k_1,k_0;C_-}(SD(\text{\bf p});h^{(1)}_1,\ldots,h^{(1)}_{k_1};\sigma;
h^{(0)}_1,\ldots,h^{(0)}_{k_0})\\
&= \frac{1}{\ell!}\text{\rm Comp}(C_-,\sigma)
((ev_{+\infty})_!)
(ev^{(1) *}h^{(1)} \wedge ev^{(0) *}h^{(0)}).
\endaligned
\end{equation}
Here
$(ev_{+\infty})_!$  is the integration along the fiber of the map
\begin{equation}\label{evaluationinfty}
ev_{+\infty}: {\mathcal M}_{k_1,k_0;\ell}(L^{(1)},L^{(0)};p,*;C_-;\text{\bf p})^{\mathfrak s}
\to L(u)
\end{equation}
of the appropriately perturbed moduli space.
More precise definition is in order.
\par
We can inductively define a multisection on
${\mathcal M}_{k_1,k_0;\ell}(L^{(1)},L^{(0)};p,*;C_-;\text{\bf p})$
so that this is transversal to $0$ and is compatible with other
multisections we have constructed in the earlier stage of induction.
We can prove it in the same way as Lemma \ref{fmulti}.
\par
However it is impossible to make the evaluation map
(\ref{evaluationinfty}) a submersion in general by the obvious
dimensional reason if we just use multisections over the moduli
space ${\mathcal M}_{k_1,k_0;\ell}(L^{(1)},L^{(0)};p,*;C_-;\text{\bf
p})$: We need to enlarge the base by considering a \emph{continuous
family} of multisections. This method was introduced in Section 7.5 \cite{fooo06} 
(= Section 33
\cite{fooo06pre}) for example and the form we need here is detailed in
Section 12 \cite{fukaya;operad}. We recall the detail of this
construction in Appendix of the present paper for readers'
convenience. More precisely we take $M_s = L(u)^{k_0+k_1}$,
$\mathcal M = {\mathcal
M}_{k_1,k_0;\ell}(L^{(1)},L^{(0)};p,*;C_-;\text{\bf p})$, $M_t =
L(u)$, $ev_s =(ev^{(1)},ev^{(0)})$, $ev_t = ev_{+\infty}$ and apply
Definition \ref{cordefinitionss}. Then the next lemma follows from
Lemma \ref{cornerextend} in Appendix.
\begin{lem}\label{gmulticonti}
There exists a continuous family $\{\mathfrak s_{\alpha}\}$ of
multisections on our moduli space ${\mathcal
M}_{k_1,k_0;\ell}(L^{(1)},L^{(0)};p,*;\beta;\text{\bf p})$ so that
it is compatible in the sense of Definition $\ref{compdefini}$ and
is also compatible with the multisections constructed before in the
inductive process at the boundaries described in Lemma
$\ref{2bdrywithpn2}$. Moreover $(\ref{evaluationinfty})$ is a
submersion.
\end{lem}
By Definition \ref{cordefinitionss}, the integration along the fiber
(\ref{deffraf4}) (or smooth correspondence map) is defined. Now we
have finished the description of the element (\ref{defg}). This
assignment induces a homomorphism
$$\aligned
\mathfrak g_{\beta}: B((\Omega(L^{(1)}) \,\widehat\otimes\, \Lambda_0)[1]) &\otimes
\left(\bigoplus_{p\in L^{(1)}\cap L^{(0)}}Hom(\mathcal L^{(1)}_p,\mathcal L^{(0)}_p)  
 \otimes_{\C} \Lambda
\right) \\
&\otimes B((\Omega(L^{(0)})\,\widehat\otimes\, \Lambda_0)[1]) \to
(\Omega(L(u))\,\widehat\otimes\, \Lambda)[1].
\endaligned$$
Now we define
$$
\mathfrak g:  CF((L^{(1)},\rho),(L^{(0)},\rho);\Lambda)
\to \Omega(L(u))\widehat{\otimes} \Lambda
$$
by
\begin{equation}
\mathfrak g(\sigma) =
\sum_{\beta} T^{\omega \cap C_-/2\pi}\mathfrak g_{\beta}(e^{\mathfrak b}
; e^{\psi_{*}(\mathfrak x_+)} \otimes \sigma \otimes e^{\mathfrak x_+}).
\end{equation}
\par
With these preparation,
we can prove the following lemma in the same way as Lemma
\ref{fischain} using Lemmata \ref{stokes} and \ref{compformula}. So its proof is
omitted.

\begin{lem}
$\mathfrak g$ is a chain map.
\end{lem}

\begin{prop}\label{homotopygyaku}
$\mathfrak f \circ \mathfrak g$ and $\mathfrak g\circ \mathfrak f$ are
chain homotopic to the identity.
\end{prop}
\begin{proof}
We will prove that $\mathfrak g \circ \mathfrak f$ is chain homotopic to the
identity. Let $S_0$ be a sufficiently large positive number.
(Say $S_0 = 10$.) For $S > S_0$ we put
$$
\chi_{S}(\tau) =
\begin{cases}
\chi(\tau +S) & \tau\le 0, \\
\chi(-\tau+S) & \tau \ge 0.
\end{cases}
$$
Here we recall that $\chi : \R \to [0,1]$ is a smooth function satisfying 
$\chi (\tau)=0$ for  
$\tau \le 0$  
and $\chi(\tau)=1$ for $\tau \ge 1$.
We will extend it to $0\le S \le S_0$ so that $\chi_0(\tau) = 0$.
\par
We consider maps
$
\varphi: \R \times [0,1] \to X
$
such that the following holds:
\begin{enumerate}
\item
$\lim_{\tau \to -\infty} \varphi(\tau,t)$ converges
to a point in $L(u)$ and is independent of $t$.
\item
$\lim_{\tau \to +\infty} \varphi(\tau,t)$ converges
to a point in $L(u)$ and is independent of $t$.
\item
$\varphi(\tau,0) \in L^{(0)}$, $\varphi(\tau,1) \in L^{(\chi_S(\tau))}$.
\end{enumerate}
We denote by $\pi_2(L^{(1)},L^{(0)};*,*;S)$ the set of homotopy
classes of such maps. There exists a natural isomorphism
$\pi_2(L^{(1)},L^{(0)};*,*;S) \cong  \pi_2(X,L(u))$,
$$
[\varphi] \mapsto [\varphi'], \qquad \text{where} \,\,
\varphi'(\tau,t) = \psi^{-1}_{\chi_S(\tau)}(\varphi(\tau,t)).
$$
Here we recall $L^{(0)} = L(u)$, $L^{(1)} = \psi_1(L(u))$. Therefore
we will denote an element of $\pi_2(L^{(1)},L^{(0)};*,*;S)$ again by
$\beta$ as for the case of $\pi_2(X,L(u))$.

We have the obvious gluing maps
\begin{equation}
\aligned
& \pi_2(L^{(1)},L^{(0)};*,p) \times \pi_2(L^{(1)},L^{(0)};p,*;S)
\to \pi_2(L^{(1)},L^{(0)};*,*;S), \\
& \pi_2(X,L^{(1)}) \times \pi_2(L^{(1)},L^{(0)};*,*;S)
\to \pi_2(L^{(1)},L^{(0)};*,*;S), \\
&  \pi_2(L^{(1)},L^{(0)};*,*;S)  \times \pi_2(X,L^{(0)}) \to
\pi_2(L^{(1)},L^{(0)};*,*;S)
\endaligned
\end{equation}
which we denote all by $\#$.
\par
We consider a three-parameter family of compatible almost complex
structures $J_{S,\tau,t}$ given by
$$
J_{S,\tau,t}
=
\psi_{t\chi_S(\tau) *} J.
$$
Then it satisfies:
\begin{equation}
J_{S,\tau,t}
=
\begin{cases}
J  &\text{$\tau$ is sufficiently small and $S \ge S_0$},
\\
J  &\text{$\tau$ is sufficiently large and $S \ge S_0$},
\\
J &\text{$t=0$}, \\
\psi_{\chi_S(\tau)*} J&\text{$t=1$},\\
J &\text{$S=0$}.
\end{cases}
\end{equation}
\begin{defn}
Consider the moduli space of maps satisfying (1) - (3) above and of
homotopy class $\beta \in \pi_2(L^{(0)},L^{(1)};*,*)$ and satisfying
the following equation
\begin{equation}\label{tautCR5}
\frac{\partial\varphi}{\partial \tau}
+ J_{S,\tau,t} \left(\frac{\partial\varphi}{\partial t}\right) = 0.
\end{equation}
For each $0 \le S < \infty$, we denote the moduli space by
$$
{\mathcal M}^{\text{reg}}_S(L^{(1)},L^{(0)};*,*;\beta).
$$
We also put
$$\aligned
&{\mathcal M}^{\text{reg}}_{+\infty}(L^{(1)},L^{(0)};*,*;\beta) \\
&= \bigcup_{p\in L^{(1)}\cap L^{(0)}}\bigcup_{C_+'\# C_-''=\beta}
({\mathcal M}^{\text{reg}}(L^{(1)},L^{(0)};*,p;C_+') \times
{\mathcal M}^{\text{reg}}(L^{(1)},L^{(0)};p,*;C_-''))
\endaligned$$
and define
$$
{\mathcal M}^{\text{reg}}(L^{(1)},L^{(0)};*,*;\beta;para)
= \bigcup_{S \in [0,+\infty]}(\{S\}\times
{\mathcal M}^{\text{reg}}_S(L^{(1)},L^{(0)};*,*;\beta)).
$$
We can also include interior and boundary marked points and
compactify the corresponding moduli space which then gives rise to
the moduli space
$$ {\mathcal
M}_{k_1,k_0;\ell}(L^{(1)},L^{(0)};*,*;\beta;para).
$$
\end{defn}
We can define the evaluation maps
$$
ev = (ev^{\text{\rm int}},ev^{(1)},ev^{(0)}):
{\mathcal M}_{k_1,k_0;\ell}(L^{(1)},L^{(0)};*,*;\beta) \to X^{\ell} \times
(L^{(1)})^{k_1} \times (L^{(0)})^{k_0},
$$
and
$$
ev_{\pm\infty}: {\mathcal M}_{k_1,k_0;\ell}(L^{(1)},L^{(0)};*,*;\beta)
\to L(u).
$$
Here
$$
ev_{\pm\infty}(\varphi) = \lim_{\tau\to \pm\infty}\varphi(\tau,t).
$$
Using $ev^{\text{\rm int}}$, we take fiber product with $D(\text{\bf p})$ and obtain
$$ {\mathcal
M}_{k_1,k_0;\ell}(L^{(1)},L^{(0)};*,*;\beta;para;\text{\bf p}). $$
\begin{lem}\label{2bdrywithpn6}
The space ${\mathcal
M}_{k_1,k_0;\ell}(L^{(1)},L^{(0)};*,*;\beta;para;\text{\bf p})$ has
an oriented Kuranishi structure with corners. Its boundary is a
union of the following six types of fiber products as a space with
Kuranishi structure:
\begin{enumerate}
\item
$${\mathcal M}_{k'_1+1;\ell'}(L(u);\beta';\text{\bf p}_1)
\,\,{}_{ev_0}\times_{ev^{(1)}_i} \,\,
{\mathcal M}_{k''_1,k_0;\ell''}(L^{(1)},L^{(0)};*,*;para;\beta'';\text{\bf p}_2).$$
Here the notations are the same as in Lemma $\ref{boundaryN}$ $(2)$
and $(\ref{p1p2})$.
\item
$$
{\mathcal M}_{k_1,k'_0;\ell'}(L^{(1)},L^{(0)};*,*;para;\beta';\text{\bf p}_1)
\,\,{}_{ev^{(0)}_i}\times_{ev_0} \,\,
{\mathcal M}_{k''_0+1;\ell''}(L(u);\beta'';\text{\bf p}_2) .$$
Here the notations are the same as in Lemma $\ref{boundaryN}$ $(3)$
and $(\ref{p1p2})$.
\item
$$
{\mathcal M}_{k'_1,k'_0;\ell'}(L^{(1)},L^{(0)};*,*;\beta';para;\text{\bf p}_1)
{}_{ev_{+\infty}}\times_{ev_0} {\mathcal M}_{k''_1+k''_0+2;\ell''}(L(u);\beta'';\text{\bf p}_2),
$$
where $k'_j + k''_j = k_j$, $\ell' + \ell'' = \ell$,
$\beta' \# \beta'' = \beta$ and $(\ref{p1p2})$.
\item
$$
{\mathcal M}_{k'_1+k'_0+2;\ell'}(L(u);\beta';\text{\bf p}_1)
{}_{ev_0}\times_{ev_{-\infty}}
{\mathcal M}_{k''_1,k''_0;\ell''}(L^{(1)},L^{(0)};*,*;\beta''
;para;\text{\bf p}_2),
$$
where $k'_j + k''_j = k_j$, $\ell' + \ell'' = \ell$,
$\beta' \# \beta'' = \beta$ and $(\ref{p1p2})$.
\item
$$
{\mathcal M}_{k'_1,k'_0;\ell'}(L^{(1)},L^{(0)};*,p;\text{\bf p}_1;C_+')
\times {\mathcal M}_{k''_1,k''_0;\ell''}(L^{(1)},L^{(0)};p,*;\text{\bf p}_2;C_-'')
$$
where $k'_j + k''_j = k_j$, $\ell' + \ell'' = \ell$,
$C_+' \# C_-'' = \beta$ and $(\ref{p1p2})$.
\item A space
$\widetilde{\mathcal M}_{k_1+k_0+2;\ell}(L(u);\beta;\text{\bf p}).
$
There exists an $\R$ action on it such that the quotient space is
${\mathcal M}_{k_1+k_0+2;\ell}(L(u);\beta;\text{\bf p})
$.
\end{enumerate}
\end{lem}
\begin{proof}
The proof is similar to the proofs of Lemma \ref{2bdrywithpn2} etc.
\par
We remark that the case $S=\infty$ corresponds to
(5).
\par
The case when $S=0$ corresponds to $(6)$. 
In fact, $\chi_0(\tau,t) =
0$. So the boundary condition reduces $\varphi(\partial(\R \times
[0,1])) \subset L(u)$ and the equation (\ref{tautCR5}) is $J$
holomorphicity. The $\tau$-translations define an $\R$-action on the
moduli space at the part $S=0$. The quotient space is the moduli
space of holomorphic discs with boundary and interior marked points.
\par
To construct a Kuranishi chart in a neighborhood of $S=\infty$, we
need to choose a smooth structure of $[0,\infty]$ at $\infty$. We
can do this so that the coordinate change of the Kuranishi structure
is smooth using the standard exponential decay estimate: Namely,
for a sufficiently large $S$, every element of ${\mathcal
M}^{\text{reg}}_S(L^{(1)},L^{(0)};*,*;\beta)$, together with its
$S$-derivatives, is close to an element of ${\mathcal
M}^{\text{reg}}_{\infty}(L^{(1)},L^{(0)};*,*;\beta)$ in the order of
$Ce^{-cS}$. We can prove this estimate in a way similar to the proof
of Lemma A1.58 \cite{fooo06} (= Lemma A1.58 \cite{fooo06pre}).
\end{proof}
\begin{lem}\label{gmulticonti2}
There exists a continuous family $\mathfrak s$ of multisections on our
moduli space
${\mathcal M}_{k_1,k_0;\ell}(L^{(1)},L^{(0)};*,*;\beta;para;\text{\bf p})$
such that it is compatible in the sense of Definition $\ref{compdefini}$
and also compatible with the one constructed before in the
induction process
at the boundaries described in Lemma $\ref{2bdrywithpn6}$.
Moreover $ev_{\pm \infty}$ are submersions on the
moduli space perturbed by this family.
\end{lem}
\begin{proof}
The proof is the same as the proof of Lemma \ref{gmulticonti}.
\end{proof}
We use $\pi_2(L^{(1)},L^{(0)};*,*;S) \cong  \pi_2(X,L(u))$ to define
$$
\rho: \pi_2(L^{(1)},L^{(0)};*,*;S)
\to \C \setminus\{0\}
$$
as the composition
$$
\pi_2(L^{(1)},L^{(0)};*,*;S)
\to \pi_2(X,L(u))
\to \pi_1(L(u))
\overset{\rho}{\longrightarrow} \C \setminus\{0\}.
$$
There is an obvious compatibility relation of this
$\rho$ and other $\rho$'s and $\text{\rm Comp}$'s we defined
before through $\#$.
\par
Now let $\beta\in \pi_2(L^{(1)},L^{(0)};*,*;S)$, $\text{\bf p} \in
Map(\ell,\underline B)$. Let $h^{(j)}_i$ ($i = 1,\ldots,k_j$) and $h$
be differential forms on $L^{(j)}$ and on $L(u)$, respectively. We will define an element
\begin{equation}\label{defh}
\mathfrak h_{\beta;\ell;k_1,k_0}(SD(\text{\bf p});h^{(1)}_1,\ldots,h^{(1)}_{k_1};h;
h^{(0)}_1,\ldots,h^{(0)}_{k_0})
\in \Omega(L(u)) \widehat{\otimes} \Lambda
\end{equation}
by
\begin{equation}\label{deffrah}
\aligned
&\mathfrak h_{\beta;\ell;k_1,k_0}(SD(\text{\bf p});h^{(1)}_1,\ldots,h^{(1)}_{k_1};h;
h^{(0)}_1,\ldots,h^{(0)}_{k_0})\\
&= \frac{1}{\ell!}\rho(\partial\beta)
((ev_{+\infty})_!)
(ev^{(1) *}h^{(1)} \wedge ev_{-\infty}^*h
\wedge ev^{(0) *}h^{(0)}).
\endaligned
\end{equation}
Here
$(ev_{+\infty})_!$  is the integration along fiber of the map
\begin{equation}\label{evaluationinfty2}
ev_{+\infty}: {\mathcal M}_{k_1,k_0;\ell}(L^{(1)},L^{(0)};*,*;\beta;para;\text{\bf p})^{\mathfrak s}
\to L(u)
\end{equation}
of our moduli space which is perturbed by the continuous family $\mathfrak s$ of
perturbations given in Lemma \ref{gmulticonti2}.
More precisely we apply Definition \ref{cordefinitionss} to
$M_s = L(u)^{k_0+1+k_1}$, $M_t = L(u)$, $\mathcal M
= {\mathcal M}_{k_1,k_0;\ell}(L^{(1)},L^{(0)};*,*;\beta;para;\text{\bf p})$, and $ev_s = (ev^{(1)},ev_{-\infty},ev^{(0)})$,
$ev_t = ev_{+\infty}$. We then obtain  (\ref{evaluationinfty2}).
\par
The family of the maps $\mathfrak h_{\beta;\ell;k_1,k_0}$ induce a
homomorphism
$$
\aligned \mathfrak h_{\beta} : B((\Omega(L^{(1)})
\,\widehat{\otimes}\, \Lambda_0)[1]) &\otimes (\Omega(L(u))\,
\widehat{\otimes}\, \Lambda))[1] \\
&  \otimes
B((\Omega(L^{(0)})\,\widehat{\otimes}\,
\Lambda_0))[1]) \longrightarrow \Omega(L(u))[1] \widehat{\otimes}
\Lambda.
\endaligned
$$
Now we define
$$
\mathfrak h:  \Omega(L(u))\,\widehat{\otimes}\, \Lambda
\to \Omega(L(u))\,\widehat{\otimes}\,\Lambda
$$
by
\begin{equation}\label{hdeffinal}
\mathfrak h(h) = \sum_{\beta} T^{\omega \cap \beta/2\pi} \mathfrak
h_{\beta}(e^{\mathfrak b} ; e^{\psi_*(\mathfrak x_+)} \otimes h \otimes
e^{\mathfrak x_+}).
\end{equation}
\begin{lem}\label{chainhomotopy}
$\mathfrak h$ is a chain homotopy from the identity to $\mathfrak g\circ \mathfrak f$.
\end{lem}
\begin{proof}
Using Lemma \ref{stokes} we can prove that
$d\circ \mathfrak h + \mathfrak h \circ d$ is a sum of terms which are obtained from each of (1) - (6) of
Lemma \ref{2bdrywithpn6}
in the same way as (\ref{deffrah}),  (\ref{hdeffinal}), as follows.
\par
Using the fact $\mathfrak q_{\rho}(e^{\mathfrak b} ; e^{\mathfrak x_+})$ is
a harmonic form of degree $0$ in the same way as the proof of Lemma \ref{boundaryfloer2},
we can show that  the contributions of (1) and (2) vanish.
The contributions of (3) and (4) are
$$
(\mathfrak m_1^{\mathfrak b,\mathfrak x} - d) \circ \mathfrak h
$$
and
$$
\mathfrak h \circ (\mathfrak m_1^{\mathfrak b,\mathfrak x} - d)
$$
respectively. The contribution of (5) is $\mathfrak g\circ \mathfrak f$.
The contribution of (6) vanishes in the case when $\beta \ne 0$, because
of extra $\R$ symmetry. The case $\beta = 0$ gives rise to the identity.
\par
In sum, we use Stokes' formula to conclude
$$
\mathfrak h \circ \mathfrak m_1^{\mathfrak b,\mathfrak x} + \mathfrak m_1^{\mathfrak b,\mathfrak
x} \circ \mathfrak h = \mathfrak g\circ \mathfrak f - id.
$$
We use the composition formula in Appendix (Lemma \ref{compformula})
to prove the above formulae. The proof of Lemma \ref{chainhomotopy}
is now complete.
\end{proof}
We have thus proved that $\mathfrak g\circ \mathfrak f$ is chain homotopic to the identity.
We can prove $\mathfrak f\circ \mathfrak g$ is chain homotopic to identity in
the same way. 
The proof of Proposition \ref{homotopygyaku} is now complete.
\end{proof}
We now explain the point mentioned in Remark \ref{824}.
\par
We take a sequence $E_0 < E_1 < \cdots$  
with $E_i \to \infty$
and then use a system of (continuous family of) multisections 
of the moduli spaces with energy $< E_i$ to obtain 
$$
\mathfrak f_{E_i} : \Omega(L(u))\, \widehat{\otimes}\, \frac{\Lambda_0}{T^{E_i}\Lambda_0}
\to CF\left((L^{(1)},\rho),(L^{(0)},\rho); \frac{T^{-c}\Lambda_0}{T^{E_i-c}\Lambda_0}\right).
$$
Here $\frac{T^{-c}\Lambda_0}{T^{E_i-c}\Lambda_0}$ is a $\Lambda_0$ module. 
(It is not a ring but an abelian group.) The number $c$ is the energy loss  
(see Definition 5.2.1 \cite{fooo06} = Definition 21.1 \cite{fooo06pre}),  
which can be estimated by Hofer distance of $\psi$ from the identity map. (See 
Subsection 5.3.5 \cite{fooo06} = Section 22.5 \cite{fooo06pre}.) 
In particular, it is independent of $E_i$.
\par
We can prove that for $j> i$, the homomorphism $\mathfrak f_{E_j}$ induces 
$$
\Omega(L(u))\, \widehat{\otimes}\, \frac{\Lambda_0}{T^{E_i}\Lambda_0}
\to CF\left((L^{(1)},\rho),(L^{(0)},\rho); \frac{T^{-c}\Lambda_0}{T^{E_i-c}\Lambda_0}\right),
$$
which is chain homotopic to $\mathfrak f_{E_i}$.
Therefore we can take its projective limit and obtain 
$$
\mathfrak f_{\#} :
HF((L(u),\mathfrak b,\mathfrak x),(L(u),\mathfrak b,\mathfrak x);\Lambda_0)
\to
HF((L^{(1)},\mathfrak b,\psi_*(\mathfrak x)),(L^{(0)},\mathfrak b,\mathfrak x);T^{-c}\Lambda_0).
$$
\par
We next define 
$$
\mathfrak g_{E_i} : 
CF\left((L^{(1)},\rho),(L^{(0)},\rho); \frac{\Lambda_0}{T^{E_i}\Lambda_0}\right)
\to \Omega(L(u))\, \widehat{\otimes}\, \frac{T^{-c}\Lambda_0}{T^{E_i-c}\Lambda_0}
$$
in the same way as $\mathfrak f_{E_i}$ and take its projective limit to obtain
$$
\mathfrak g_{\#} :
HF((L^{(1)},\mathfrak b,\psi_*(\mathfrak x)),(L^{(0)},\mathfrak b,\mathfrak x);\Lambda_0)
\to
HF((L(u),\mathfrak b,\mathfrak x),(L(u),\mathfrak b,\mathfrak x);T^{-c}\Lambda_0).
$$
\par
We remark that $\mathfrak g_{E_i}$ induces a homomorphism 
$$
CF\left((L^{(1)},\rho),(L^{(0)},\rho); \frac{T^{-c}\Lambda_0}{T^{E_i-c}\Lambda_0}\right)
\to \Omega(L(u))\, \widehat{\otimes}\, \frac{T^{-2c}\Lambda_0}{T^{E_i-2c}\Lambda_0}.
$$
Therefore we have
$$
\mathfrak g_{E_i} \circ \mathfrak f_{E_i} : 
\Omega(L(u))\, \widehat{\otimes}\, \frac{\Lambda_0}{T^{E_i}\Lambda_0} 
\to \Omega(L(u))\, \widehat{\otimes}\, \frac{T^{-2c}\Lambda_0}{T^{E_i-2c}\Lambda_0}.
$$
Now we can construct $\mathfrak h_{E_i}$ which is a 
chain homotopy from $\mathfrak g_{E_i} \circ \mathfrak f_{E_i}$ to the map which is induced by 
the identity map.
Therefore $(\mathfrak f_{\#} \circ \mathfrak g_{\#}) \otimes_{\Lambda_0} \Lambda$ is the
identity map. In the same way  $(\mathfrak g_{\#} \circ \mathfrak f_{\#}) \otimes_{\Lambda_0} \Lambda$
is the identity map.
The proof of Proposition \ref{calcFloer} is complete. 
\par
Hence, by using the isomorphism (\ref{caniso}), we have
also completed the proof of Proposition \ref{prof:bal2} also.
\end{proof}
\begin{rem}
We gave the proof of the above proposition using de Rham cohomology.
In \cite{fooo06} we gave a proof based on singular cohomology.
Strictly speaking we only discussed in the case when $b \in
H^1(L(u);\Lambda_+)$ in \cite{fooo06}. But using Cho's idea of
shifting the constant term by non-unitary flat connection, the proof
of \cite{fooo06} can be easily generalized to the present situation
of $H^1(L(u);\Lambda_0)$. In fact Theorem \ref{HFdisplace} was
proved in Section 3.8 or 5.3 \cite{fooo06} (= Section 13 or 22 \cite{fooo06pre}) by proving a statement
similar to Proposition \ref{calcFloer} from which we can derive
Proposition \ref{prof:bal2}. (See also Remark \ref{37.190}.)
\par
The approach using de Rham cohomology is shorter but we cannot treat
the results with $\Q$-coefficients, at least at the time of writing
this article. Therefore we need to use singular homology version
for that purpose. It might be possible to develop the $\Q$ de Rham
theory for the purpose. We remark that by Lemma \ref{cbetawell}, $
\mathfrak{PO}^u(\mathfrak b;y_1,\ldots,y_n) $ is defined over the $\Q$
coefficients. To study quantum cohomology $QH(X;\Lambda_{0}(\Q))$
this de Rham version will be enough.
\end{rem}

\section{Domain of definition of potential function with bulk}
\label{sec:domain}
The purpose of this section is to prove Theorem
\ref{extendth}. Theorem \ref{extendth} is not used in the
other part of this paper except in Section \ref{sec:byLambda0}
but is used in \cite{fooo10}.
\begin{proof}[Proof of Theorem \ref{extendth}]
We recall that $D_1, \ldots, D_m$ are of complex codimension $1$ in $X$
and $D_{m+1},\ldots, D_{B}$ are of higher complex codimension. Let
$\text{\bf p} \in Map(\ell,\underline B)$. We put
$$
\vert \text{\bf p}\vert_{\text{\rm high}}
= \#\{ j \mid \text{\bf p}(j) > m\}.
$$
\begin{lem}\label{highestimate}
For any $E$ we have
$$
\sup \{\vert \text{\bf p}\vert_{\text{\rm high}} \mid
c(\beta;\text{\bf p}) \ne 0, \,\,
\beta \cap\omega < E\} < C(E),
$$
where $C(E)$ depends only on $E$ and $X$.
Here the supremum in the left hand side is taken over $\ell$ and 
$\text{\bf p} \in Map(\ell,\underline B)$.
\end{lem}
\begin{proof}
If $\vert \text{\bf p}\vert_{\text{\rm high}} = \mathcal N$ and
$c(\beta;\text{\bf p}) \ne 0$, then $2\mathcal N \leq \mu(\beta) $
by the dimension counting. The lemma then follows from Proposition
\ref{copy37.177} (5) and Gromov's compactness.
\end{proof}
We denote by $Map(\ell_+,\underline B \setminus \underline m)$ 
the set of the maps $\{1,\ldots,\ell_+\} \to \underline B \setminus
\underline m = \{m+1,\ldots,B\}$. We put
$$
M_+ = \bigcup_{\ell_+}Map(\ell_+,\underline B \setminus \underline m).
$$
For $\text{\bf p}_+ \in Map(\ell_+,\underline B \setminus \underline
m)$ and $\ell_1,\ldots,\ell_m$ we define $\text{\bf p} =
(\ell_1,\ldots,\ell_m;\text{\bf p}_+)$ by
\begin{equation}
\text{\bf p}(i)
=
\begin{cases}
j  &\text{if $\ell_1+\cdots+\ell_{j-1} < i \le \ell_1+\cdots+\ell_{j}$,}\\
\text{\bf p}_+(i -\sum_{j=1}^m \ell_j)
&\text{if $i > \sum_{j=1}^m \ell_j$}.
\end{cases}
\end{equation}
\begin{lem}\label{deg2divixopr}
If $\text{\bf p} = (\ell_1,\ldots,\ell_m;\text{\bf p}_+)$, then
$$
c(\beta;\text{\bf p})
= c(\beta;\text{\bf p}_+)\prod_{i=1}^m (\beta\cap D_i)^{\ell_i}.
$$
\end{lem}
\begin{proof}
By the dimensional reason
$$
\dim \mathcal M_{1,\vert\text{\bf p}_+\vert}(L(u);\beta;\text{\bf p}_+)
= n
$$
and $c(\beta;\text{\bf p}_+)$ is the degree of the map
\begin{equation}\label{ev00}
ev_0: \mathcal M_{1,\vert\text{\bf p}_+\vert}(L(u);\beta;\text{\bf p}_+)
\to L(u).
\end{equation}
Note that $\mathcal M_{1,\vert\text{\bf p}_+\vert}(L(u);\beta;\text{\bf
p}_+)$ after perturbation is a space with triangulation and the
weight in $\Q$, which is defined by the multiplicity and the order
of the isotropy group. So it has a fundamental cycle over $\Q$.
\par
We fix a regular value $p_0 \in L(u)$ of (\ref{ev00}). Let
$$
ev_0^{-1}(p_0) = \{\varphi_j \mid j=1,\ldots,K\}
$$
be its preimage. Each of its elements contributes to $c(\beta;\text{\bf p}_+)$ by $\epsilon_j \in \Q$ so that $\sum \epsilon_j =
c(\beta;\text{\bf p}_+)$.
\par
We remark that our counting problem to calculate $c(\beta;\text{\bf p})$ is well-defined in the sense of Lemma  \ref{cbetawell}.
Therefore we can perform the calculation in the homology level using 
Lemma \ref{37.184} (3) to
find that each of $\varphi_j$ contributes $\epsilon_j\prod_{i=1}^m
(\beta\cap D_i)^{\ell_i}$ to 
$c(\beta;\text{\bf p})$. The lemma
follows.
\end{proof}
Now we are ready to complete the proof of Theorem \ref{extendth}. By
Lemma \ref{deg2divixopr} and 
Formula \eqref{calcPO}, we find
\begin{equation}\label{94}
\aligned
&\mathfrak{PO}^u(w_1,\ldots,w_B;\mathfrak x) \\
&{}\quad= \sum_\beta \sum_{\text{\bf p}_+ \in M_+}
\sum_{\ell_1,\cdots,\ell_m} \left(
\frac{(\ell_1+\cdots+\ell_m+\vert\text{\bf p}_+\vert)!}
{\ell_1!\cdots\ell_m!\vert\text{\bf p}_+\vert!}\right)
\frac{w_{\text{\bf p}_+(1)}\cdots w_{\text{\bf p}_+(
\vert\text{\bf p}_+\vert)}}{(\ell_1+\cdots+\ell_m+\vert\text{\bf p}_+\vert)!}\\
&{}\qquad\qquad T^{\beta\cap \omega/2\pi}
c(\beta;\text{\bf p}_+)
\left(\prod_{i=1}^m (\beta\cap D_i)^{\ell_i}\right) w_1^{\ell_1}\cdots
w_m^{\ell_m}\exp(\partial \beta\cap
\mathfrak x) \\
&{}\quad = \sum_\beta \sum_{\text{\bf p}_+ \in M_+}
\frac{c(\beta;\text{\bf p}_+)}{\vert\text{\bf p}_+\vert !}
w_{\text{\bf p}_+(1)}\cdots w_{\text{\bf p}_+(
\vert\text{\bf p}_+\vert)}T^{\beta\cap \omega/2\pi}\\
&{}\quad\qquad\qquad\qquad
\mathfrak w_1^{\beta\cap D_1}\cdots
\mathfrak w_m^{\beta\cap D_m} (y_1(u))^{\partial \beta\cap \text{\bf e}_1}
\cdots (y_n(u))^{\partial \beta\cap \text{\bf e}_n}.
\endaligned
\end{equation}
We recall that we can write 
$\beta=\sum_{j=1}^m a_j \beta_j +\sum \alpha_i$ for some 
$a_j \in \Z_{\ge 0}$ as in Proposition \ref{copy37.177} (5). 
Then by noticing $\beta_j\cap \omega/2\pi=\ell_j(u)$ and using Lemma \ref{characterizationlamda0}, we find that
\begin{equation}\label{zproduct}
T^{\beta\cap\omega/2\pi}
(y_1(u))^{\partial \beta\cap \text{\bf e}_1}
\cdots (y_n(u))^{\partial \beta\cap \text{\bf e}_n} = T^{\sum\alpha_i\cap\omega/2\pi}\prod_{j=1}^m
(z_j)^{a_j}.
\end{equation}
By Lemma \ref{lem:zjupsiu} we have
$\mathfrak v_T^{u}(z_j^{a_j})=a_j\ell_j(u) \ge 0$ for any 
$u \in \text{Int }P$, and 
by Theorem \ref{weakpotential} we have 
$\mathfrak v_T^u(T^{\sum_{i} \alpha_i \cap \omega/2\pi}) 
= +\infty$ if 
there are infinitely many non-zero terms in 
$\sum_{i} \alpha_i \cap \omega/2\pi$. 
Therefore 
by Lemma \ref{highestimate} the series 
(\ref{94}) converges on $\mathfrak
v_T^{u}$-adic topology for any $u \in \operatorname{Int}\,P$.
\begin{rem}\label{transsukosi}
In the second equality in (\ref{94}), beside the identity
$$
\exp(\partial \beta \cap \mathfrak x) = (y_1(u))^{\partial \beta\cap \text{\bf e}_1}
\cdots (y_n(u))^{\partial \beta\cap \text{\bf e}_n},
$$
we also use the definition of $\mathfrak w_i$
\begin{equation}\label{defoffrakw}
\mathfrak w_i := \exp w_i= \sum_{k=0}^{\infty} \frac{w_i^k}{k!}.
\end{equation}
(See the line right above \eqref{stconvserring}.)
Generally, any $x \in \Lambda_0$ can be 
uniquely written as
$x=\overline{x} + x_+$ where 
$\overline{x} \in \C$ and $x_+ \in \Lambda_+$. Then 
we have 
$$
e^x = e^{\overline{x}} e^{x_+} = 
\left(\sum_{k=0}^{\infty} \frac{\overline{x}^k}{k!} \right)
\left( \sum_{k=0}^{\infty} \frac{x_+^k}{k!}
\right),
$$
where the first factor $\sum_{k=0}^{\infty} \overline{x}^k /k!$ converges with respect to the usual Archimedean topology of $\C$ and 
the second factor $\sum_{k=0}^{\infty} x_+^k /k!$
converges with respect to the adic  
non-Archimedean topology. 
In this sense,  
if we replace the formal variable $w_i$ by a 
number
$c_i \in \C$ and 
$\mathfrak w_i$ by $\mathfrak c_i = e^{c_i} \in \C$, 
then
the second equality (\ref{94}) still holds. 
The convergence
in the left hand side of
$$
\sum_{k=0}^{\infty} \frac{c_i^k}{k!} = \mathfrak c_i
$$
is with respect to the usual Archimedean 
topology of $\C$.
\end{rem}
Now we examine the dependence of this sum on $u$'s. Firstly, through
the isomorphism $\psi_u: H^*(T^n;\Z) \to H^*(L(u);\Z)$, we may
regard $\beta$ or $\beta_j$ are independent of $u$ and so are the
coefficients $a_j$'s. Secondly by the structure theorem, Proposition
\ref{copy37.177}, the moduli spaces associated to a given $\beta$
are all isomorphic and so can be canonically identified when $u \in
\text{Int}P$ varies. Thirdly
Lemma \ref{lem:zjupsiu} shows that $z_j$ in the formula (\ref{zproduct}) are
independent of $u \in \text{Int} P$. Therefore the composition
$\mathfrak{PO}^u \circ \psi_u$ is independent of $u$'s. This proves
Theorem \ref{extendth} (1). We denote the common function by $\mathfrak{PO}$.
\par
Now if we regard $\mathfrak{PO}$ as a
a function defined on $\mathcal A(\Lambda_+) \times H^1(T^n;\Lambda_0)$,
we can easily derive from the expression \eqref{94} and 
the formula \eqref{zproduct} that $\mathfrak{PO}$ lies in
$\Lambda^P_0\langle\!\langle\mathfrak w,\mathfrak w^{-1},w,y,y^{-1}\rangle\!\rangle$
by using Lemma \ref{characterizationlamda0}.

The proofs of Theorem \ref{extendth} are now complete.
\end{proof}
We recall that $X$ is nef if and only if every holomorphic
sphere $w: S^2 \to X$ satisfies $w_*[S^2] \cap c_1(X) \ge 0$.
In the nef case we can prove the following statement which is
somewhat similar to Proposition \ref{POcalcFano}.
\begin{prop}
If $X$ is nef and $\mathfrak b$ is as in $(\ref{formfrakb})$, then we have
\begin{equation}\label{POFanoformulanef}
\mathfrak{PO}^u(\mathfrak b;y)= \sum_{l=1}^{K}\sum_{j=1}^{a(l)}
T^{S_l}(\exp(\mathfrak b_{l,j}) + c_{l,j}(\mathfrak b)) y^{\vec v_{i(l,j)}} +
\sum_{i=\mathcal K+1}^m T^{\ell_i(u)}(1 + c_i(\mathfrak b)) y^{\vec v_i},
\end{equation}
where $c_i(\mathfrak b), c_{l,j}(\mathfrak b) \in \Lambda_+$.
\end{prop}
\begin{proof}
Let $\beta \in H_2(X,L(u);\Z)$ with $\mu(\beta) = 2$.
We assume $\mathcal M_{1;\ell}^{\text{\rm main}}(L(u),\beta)$ is nonempty.
Let
$$
\beta = \sum_{i=1}^m k_i\beta_i + \sum_j \alpha_j
$$
be as in Proposition \ref{copy37.177} (5). Since $\alpha_j \cap
c_1(X) \ge 0$ by assumption, it follows from the condition
$\mu(\beta) = 2$ that there exists unique $i$ such that $k_i =1$ and
other $k_i$'s are zero. Moreover  $\alpha_j \cap c_1(X) = 0$. Hence if
$\beta$ is not $\beta_i$, then we have
$$
\beta = \beta_i + \sum_j \alpha_j.
$$
This $\beta$ contributes
$$
c T^{\sum_j \alpha_j\cap[\omega]/2\pi} T^{\ell_i(u)} y^{\vec v_i}
$$
to $\mathfrak{PO}^u(\mathfrak b;y)$.
The rest of the proof is the same as the proof of Proposition \ref{POcalcFano}.
\end{proof}
\par
\section{Euler vector field}
\label{sec:euler} 
The formula (\ref{94}) derived in the previous
section yields an interesting consequence which is related to the
Euler vector field on a Frobenius manifold and to our potential
function. In \cite{fooo10}, we further
discuss the Frobenius manifold structure on the quantum cohomology
and on the Jacobian ring of our potential function and their
relationship. (See Remark \ref{anounce}.)
\par
For $i=1,\cdots,B$, let $d_i$ be the degree of $D_i \in \mathcal A$.
(That is twice of the real codimension of the corresponding faces of
$P$.) In case $d_i =2$ (that is $i\le m$) 
we observe that
$$
\mu_{L(u)}(\beta) = 
\sum_{i=1}^m 2 (\beta \cap D_i)
$$
for $\beta \in H_2(X,L(u);\Z)$.
Here $\mu_{L(u)}$ is the
Maslov index.
\begin{defn}
We define the {\it Euler vector field} $\mathfrak E$ on $\mathcal A$
by
$$
\mathfrak E = \sum_{i=m+1}^B \left( 1 - \frac{d_i}{2}\right)w_i \frac{\partial}
{\partial w_i} + \sum_{i=1}^m  \frac{\partial}
{\partial \mathfrak w_i}.
$$
\end{defn}
\begin{thm}\label{eulerthem} The directional derivative
$\mathfrak{PO}^u$ along the vector field $\mathfrak E$ satisfies
$$
\mathfrak E(\mathfrak{PO}^u) = \mathfrak{PO}^u.
$$
\end{thm}
\begin{proof}
The proof is similar to the proof of a similar identity for the case
of the Gromov-Witten potentials. (See \cite{dub} for example.) Let
$$\aligned
\mathfrak E_1 &= \sum_{i=m+1}^B \left( 1 - \frac{d_i}{2}\right)w_i \frac{\partial}
{\partial w_i}, \\
\mathfrak E_2 &=  \sum_{i=1}^m  \frac{\partial}
{\partial \mathfrak w_i}, \\
\mathfrak{PO}^u_{\beta,1} &= \sum_{\text{\bf p}_+ \in M_+}
\frac{c(\beta;\text{\bf p}_+)}{\vert\text{\bf p}_+\vert !}w_{\text{\bf p}_+(1)}\cdots w_{\text{\bf p}_+(
\vert\text{\bf p}_+\vert)}, \\
\mathfrak{PO}^u_{\beta,2} &= \mathfrak w_1^{\beta\cap D_1}\cdots
\mathfrak w_m^{\beta\cap D_m}.
\endaligned$$
Since $\dim \mathcal M_{1,\vert\text{\bf p}_+\vert}(L(u),\beta;\text{\bf p}_+) = n$,
it follows that
$$
n -2 + \mu_{L(u)}(\beta) + \sum_{i} (2 - \deg \text{\bf p}_+(i)) = n.
$$
Therefore
$$
\mathfrak E_1(\mathfrak{PO}^u_{\beta,1}) = \left(1 - \frac{\mu_{L(u)}(\beta)}{2}\right)
\mathfrak{PO}^u_{\beta,1}.
$$
On the other hand, we have
$$
\mathfrak E_2(\mathfrak{PO}^u_{\beta,2}) = \frac{\mu_{L(u)}(\beta)}{2}\,\mathfrak{PO}^u_{\beta,2}
$$
by definition. Theorem \ref{eulerthem} now follows from (\ref{94}).
\end{proof}
\begin{rem}\label{anounce}
In our recent paper \cite{fooo10}, we prove the ring isomorphism
\begin{equation}\label{Miriso}
\Phi:  (H(X;\Lambda_0),\cup^{\mathfrak b}) \cong \frac{\Lambda_0^P\langle\!\langle y,y^{-1}\rangle\!\rangle}
{\left( y_i\frac{\partial \mathfrak{PO}_{\mathfrak b}}{\partial  y_i}: i=1,\ldots,n\right)},
\end{equation}
for arbitrary compact toric manifold (which is not necessarily Fano).
Here the product $\cup^{\mathfrak b}$ in the left hand side is defined by the formula
$$
\langle \mathfrak a_1 \cup^{\mathfrak b} \mathfrak a_2, \mathfrak a_3 \rangle_{\text{\rm PD}} =
\sum_{\alpha\in H_2(X;\Z)} \sum_{\ell=0}^{\infty}
\frac{T^{\alpha\cap \omega/2\pi}}{\ell!} GW_{\alpha,\ell+3}(\mathfrak
a_1,\mathfrak a_2,\mathfrak a_3,\mathfrak b^{\otimes \ell})
$$
where
$$
GW_{\alpha,m}(\mathfrak c_1,\ldots,\mathfrak c_m) = \int_{\mathcal M_m(\alpha)}
ev^*(\mathfrak c_1 \times \cdots \times \mathfrak c_m),
$$
$\mathcal M_m(\alpha)$ is the moduli space of the stable maps of genus 0
with $m$ marked points in homology class $\alpha$, and $ev:
\mathcal M_m(\alpha) \to X^{m}$ is the evaluation map. ($\langle
,\rangle_{\text{\rm PD}}$ denotes the Poincar\'e duality pairing.)
\par
The isomorphism (\ref{Miriso}) is defined as follows. We choose a lift
$$
\bigoplus_{d\ne 0} H^d(X;\Lambda_0) \cong \mathcal H(\Lambda_0) \subset \mathcal A(\Lambda_0).
$$
Using its basis $\text{\bf f}_a$ we write an element of $\mathcal H(\Lambda_0)$ as
$\sum_a w_a \text{\bf f}_a$. Then (\ref{Miriso}) sends $\text{\bf f}_a$ to
$$
\left[\left(\frac{\partial}{\partial w_a} \mathfrak{PO}\right)(\mathfrak b;y)\right].
$$
We can prove that this map is a ring isomorphism by an argument
which elaborates the discussion outlined in Remark 6.15
\cite{fooo08}. We have worked it out in detail in  \cite{fooo10}.
\par
We also prove in \cite{fooo10} that if $\mathfrak{PO}_{\mathfrak b}$ has
only nondegenerate critical point, then (\ref{Miriso}) sends
Poincar\'e duality to the residue pairing. Here the residue pairing is
defined, in the case when $X$ is nef and $\deg \mathfrak b = 2$, as follows: (See \cite{fooo10}
for the general case.) 
By nondegeneracy assumption we have a ring
isomorphism
\begin{equation}
\frac{\Lambda_0^P\langle\!\langle y,y^{-1}\rangle\!\rangle}
{\left( y_i\frac{\partial \mathfrak{PO}_{\mathfrak b}}{\partial  y_i}: i=1,\ldots,n\right)}
\otimes_{\Lambda_0} \Lambda
\cong \prod_{p\in \text{\rm Crit}(\mathfrak{PO}_{\mathfrak b})} \Lambda.
\end{equation}
(See Proposition 7.10 \cite{fooo08}. It is generalized to the non-Fano case
in \cite{fooo10} Proposition 2.15.) Here $\text{\rm Crit}(\mathfrak{PO}_{\mathfrak b})$  is
the set of critical points of $\mathfrak{PO}_{\mathfrak b}$. Let $1_p$ be the
unit $\in \Lambda$ in the factor corresponding to $p$. We then put
$$
\langle 1_p,1_q \rangle_{\text{res}} =
\begin{cases}
0   &\text{if $p \ne q$,} \\
(\det \text{\rm Hess}_p \mathfrak{PO}^u_{\mathfrak b})^{-1} &\text{if $p=q$.}
\end{cases}
$$
Here
$$
\text{\rm Hess}_p \mathfrak{PO}^u_{\mathfrak b} = \left( \frac{\partial^2 \mathfrak{PO}^u_{\mathfrak b}}
{\partial x_i\partial x_j}\right) (\mathfrak x)
$$
is the Hessian matrix at $\mathfrak x = (\mathfrak x_1,\ldots,\mathfrak x_n)$ with
$e^{\mathfrak x_i} = \mathfrak y_i$, $e^{x_i} = y_i$ and 
$$(T^{u_1}\mathfrak
y_1,\ldots,T^{u_n}\mathfrak y_n) = p.$$
Then we have:
\begin{equation}\label{ringiso}
\langle \mathfrak c, \mathfrak d\rangle_{\text{\rm PD}}
= \langle \Phi(\mathfrak c), \Phi(\mathfrak d) \rangle_{\text{\rm res}}.
\end{equation}
\par
The proof of \eqref{ringiso},
which we give in \cite{fooo10}, uses the moduli space of
pseudo-holomorphic annuli bordered to our Lagrangian fiber $L(u)$.
\par
In the mean time, here we illustrate the identity \eqref{ringiso}
for the simple case $X = \C P^1$, $\mathfrak b = \text{\bf 0}$.
(See \cite{taka}.) Its
moment polytope is $[0,1]$. The potential function is:
$$
\mathfrak{PO}_{\text{\bf 0}}^u(y) = T^{u}y + T^{1-u}y^{-1}.
$$
The critical points are given at $u= 1/2$ and $y = \pm 1$. We denote
them by $p_{+}, p_{-}$ respectively. We have
$$
\text{\rm Hess}_{p_{+}} \mathfrak{PO}^{1/2}_{\text{\bf 0}} = 2T^{1/2},
\quad \text{\rm Hess}_{p_{-}} \mathfrak{PO}^{1/2}_{\text{\bf 0}} = -2T^{1/2}.
$$
(Note we here take $x = \log y$ as a variable.)
Therefore
$$
\langle 1_{p_{+}},1_{p_{+}} \rangle_{\text{\rm res}} = T^{-1/2}/2, \quad
\langle 1_{p_{-}},1_{p_{-}} \rangle_{\text{\rm res}} = -T^{-1/2}/2, \quad
\langle 1_{p_{+}},1_{p_{-}}\rangle_{\text{\rm res}} = 0.
$$
We consider $PD[pt] \in H^2(\C P^1)$ and identify it with
$[\pi^{-1}(0)]$. Then the isomorphism (\ref{Miriso}) sends $PD[pt]$
to $T^{u}y \mod \left( y\frac{\partial \mathfrak{PO}^u_0}{\partial
y}\right)$. At $u = 1/2$, the latter becomes $T^{1/2}(1_{p_+} -
1_{p_-})$ in the Jacobian ring, which can be easily seen from the
identity
$$
T^{1/2}y = \frac{1}{2} \left( (1+y) - (1-y) \right)T^{1/2}.
$$
On the other hand, $PD[\C P^1] \in H^0(\C P^1)$ is the unit and so
becomes $1_{p_+} + 1_{p_-}$. We have
$$
\langle T^{1/2}(1_{p_+} - 1_{p_-}), 1_{p_+} + 1_{p_-}\rangle_{\text{\rm res}} = 1.
$$
This is consistent with the corresponding pairing
$$
\langle PD[pt],PD[\C P^1]\rangle_{\text{\rm
PD}} = (PD[pt] \cup PD[\C P^1]) \cap [\C P^1]
= 1
$$
in the quantum cohomology side.
\par
We recall that collection of a product structure on the tangent
space, residue pairing, Euler vector field, and the unit consists of
the data which determine Saito's flat structure (that is, the
structure of Frobenius manifold) \cite{Sai83}.
\end{rem}
\par
\section{Deformation by $\mathfrak b \in \mathcal A(\Lambda_0)$}
\label{sec:byLambda0}
\par
In Sections \ref {sec:Ellim} and \ref{sec:HFbulk}, we used the bulk
deformation of Lagrangian Floer cohomology by the divisor cycles
$\mathfrak b \in \mathcal A(\Lambda_+)$. Actually using the result of Section \ref{sec:domain}, most of the argument there can be
generalized to the case where $\mathfrak b \in \mathcal A(\Lambda_0)$ by
a minor modification. In this section we discuss this  and some new
phenomena appearing in the deformation by $\mathfrak b \in \mathcal
A(\Lambda_0)$.
\par
In this section we consider the case $R=\C$.  (See
Remark \ref{Rrationality}, however.) We write $\Lambda_0$ etc. in
place of $\Lambda_0(\C)$ etc.. 
\par
We first remark that the potential
function 
$$\mathfrak{PO}^u_{\mathfrak b}(y_1,\ldots,y_n) = \mathfrak{PO}^u(\mathfrak
b;y_1,\ldots,y_n)$$ 
itself 
can be defined for $\mathfrak b \in \mathcal
A(\Lambda_0)$ by the formula 
$$
\mathfrak{PO}^u_{\mathfrak b}(y_1,\ldots,y_n) 
= \sum_{\beta;\ell,k} T^{\beta\cap \omega/2\pi}
\mathfrak q^{{\text{\rm can}}}_{\beta;\ell,k}(\mathfrak b^{\otimes\ell};b,\ldots,b),
$$
where $b = \sum x_i \text{\bf e}_i$, $y_i = e^{x_i}$.
Here the right hand side converges by (\ref{94}).
The convergence means one with respect to 
the topology of combination 
of the adic non-Archimedean topology on $\Lambda_+$ and 
the usual Archimedean topology on $\C$ described in Remark \ref{transsukosi}.
\par
But the definition of the
leading term equation (\ref{LTE}), Definition \ref{LOEdef} need
some minor modification which is in order. 
We put
$$
\mathfrak b = \sum \mathfrak b_a D_a \in \mathcal A(\Lambda_0)
$$
and consider its zero order term
$$
\mathfrak b_a \equiv \overline{\mathfrak b}_a \mod \Lambda_+
$$
where $\overline{\mathfrak b}_a \in \C$. 
We put
\begin{equation}\label{bulkLOE}
(\mathfrak{PO}^u_{\mathfrak b})_l = \sum_{r=1}^{a(l)}
\exp(\overline{\mathfrak b}_{i(l,r)}) y^{\vec v_{l,r}} \in
\C[y_{1,1},\ldots,y_{l,d(l)}^{-1}].
\end{equation}
We define the leading term equation for
\begin{equation}\label{criequ}
y_{l,s}\frac{\partial \mathfrak{PO}^u_{\mathfrak b}}{\partial y_{l,s}} = 0
\end{equation}
for $\mathfrak b\in \mathcal A(\Lambda_0)$ in the same way as Definition
\ref{LOEdef} by using (\ref{bulkLOE}) in place of (\ref{POsital}). 
Namely, the leading term equation is the system of equations 
\begin{equation}\label{LOELambda0}
\frac{\partial (\mathfrak{PO}^u_{\mathfrak b})_l}{\partial y_{l,s}} = 
\frac{\partial}{\partial y_{l,s}} 
\left( 
\sum_{r=1}^{a(l)}
\exp(\overline{\mathfrak b}_{i(l,r)}) y^{\vec v_{l,r}} \right) = 
0 
\end{equation}
with $y_{l,s} \in \C\setminus \{ 0\}$ 
for $l=1,\dots ,K, 
~ s=1,\dots , d(l)$. 
\par
We remark that only $\overline{\mathfrak b}_a$, $a=1,\ldots,m$ appear
in (\ref{bulkLOE}). In other words, coefficients of the cohomology
classes $D_a$ of degree $>2$ do not affect the leading term
equation.
\begin{lem}\label{genrelLTE}
Lemma $\ref{relLTE}$ holds also for $\mathfrak b \in \mathcal
A(\Lambda_0)$.
\end{lem}
\begin{proof}
The formula (\ref{94}) implies that the coefficient of $y^{\vec
v_{l,r}}$ $(r = 1,\ldots,a(l))$ in $\mathfrak{PO}^u_{\mathfrak b}$ is
$T^{S_l}\exp(\overline{\mathfrak b}_{i(l,r)})$. The rest of the proof is
the same as the proof of Lemma \ref{relLTE}.
\end{proof}
The leading term equation is of the form
$$
0 = \sum_{r=1}^{a(l)} \exp(\overline{\mathfrak b}_{i(l,r)}) y^{\vec v_{l,r}} \vec v_{l,r}.
$$
We note that $\exp : \C \to \C \setminus \{0\}$ 
is surjective.
\begin{defn}\label{GLOE}
A system of polynomial equations
$$
0 = \sum_{r=1}^{a(l)} C_{i(l,r)} y^{\vec v_{l,r}} \vec v_{l,r}
$$
with $C_{i(l,r)}\in \C \setminus\{0\}$, $l=1,\cdots K$ is called a
{\it generalized leading term equation}.
\end{defn}
Now Theorem \ref{ellim} is generalized as follows.
\begin{prop}\label{gellim}
The following three conditions for $u$ are equivalent to each other.
\begin{enumerate}
\item There exists a generalized leading term equation of $\mathfrak {PO}_0^u$,
which has a solution
$y_{l,j} \in \C \setminus \{0\}$ 
$(l=1,\dots , K, s=1,\dots ,d(l))$.
\item There exists $\mathfrak b \in \mathcal A(\Lambda_0)$ such that $\mathfrak{PO}^u_{\mathfrak b}$
has a critical point on $(\Lambda_{0} \setminus \Lambda_+)^n$. 
\item There exists $\mathfrak b \in \mathcal A(\Lambda_0)$ such that $y_{l,s} \in \C \setminus \{0\}$ 
$(l=1,\dots , K, s=1,\dots ,d(l))$ in the item (1) above is a critical point of $\mathfrak{PO}^u_{\mathfrak b}$. 
\end{enumerate}
\end{prop}
\begin{proof} (3) $\Rightarrow$ (2) is obvious. 
(2) $\Rightarrow$ (1) follows from Lemma \ref{genrelLTE}.
Let us assume that the generalized leading term equation with $C_{i(l,r)}$ as a coefficient has a solution $\mathfrak y_{l,s} \in \C \setminus \{0\}$.
We put $\overline{\mathfrak b}_{i(l,r)} = \log C_{i(l,r)}$. Then we can
add higher order term in the same way as the proof of Theorem \ref{ellim}
to obtain $\mathfrak b$ such that $\mathfrak y_{l,s} $ is a solution of (\ref{criequ}). Thus 
(1) $\Rightarrow$ (3) follows. 
\end{proof}
\begin{prop}\label{lambda0General}
Theorem $\ref{homologynonzero}$ and Proposition $\ref{prof:bal2}$ hold for $\mathfrak b \in
\mathcal A(\Lambda_0)$.
\end{prop}
\begin{proof}
The proof goes in the way similar to those of Theorem $\ref{homologynonzero}$ and Proposition $\ref{prof:bal2}$
using (\ref{94}).
We however need to modify the definition $\mathfrak m_k^{\mathfrak b,b}$ 
so that it converges for $\mathfrak b \in  \mathcal A(\Lambda_0)$ as follows.
\par
We remark that (\ref{94}) and Lemma \ref{deg2divixopr} are stated only in the case of moduli spaces 
$\mathcal M_{1,\vert\text{\bf p}\vert}(L(u);\beta;\text{\bf p})$ of dimension $n$, 
that is the moduli space used to define $\mathfrak{PO}^u(w_1,\ldots,w_B;\mathfrak x)$.
This is because it is cumbersome to appropriately generalize the proofs of Lemma \ref{deg2divixopr}
for other dimensions.
The simplest way to go around this trouble is to make the formula 
similar to (\ref{94}) as a `definition' of modified $\mathfrak q$ operator, 
and proceed as follows.
\par
Let $\mathfrak q_{\beta;\ell,k}^{{\text{\rm can}}}$ be the restriction of (\ref{37.188}) 
to $T^n$ invariant forms.
We divide
$$
\mathfrak b = \mathfrak b_0 + \mathfrak b_{\text{\rm high}}
$$
where $\mathfrak b_0 \in \mathcal A^2(\Lambda_0)$ and $\mathfrak b_{\text{\rm high}} \in 
\bigoplus_{k>1} \mathcal A^{2k}(\Lambda_0)$.  
From the calculation of (\ref{94}) we can observe that 
$$
\mathfrak{PO}_{\mathfrak b}^u(y_1,\dots , y_n) = 
e^{\beta \cap \mathfrak b_0}
\mathfrak{PO}_{\mathfrak b_{\text{\rm high}}}^u(y_1,\dots , y_n)
$$ 
and   
$$
\mathfrak m_k^{\mathfrak b,\mathfrak x}(x_1,\dots ,x_n)=e^{\beta \cap \mathfrak b_0}
\mathfrak m_k^{\mathfrak b_{\text{\rm high}},\mathfrak x}
(x_1,\dots ,x_n)
$$
for $\mathfrak b_0 \equiv 0 \mod \Lambda_+$.
Now we consider  
$\mathfrak b = \mathfrak b_0 + \mathfrak b_{\text{\rm high}}$ 
where $\mathfrak b_0$ is not necessarily assumed as  
$\mathfrak b_0 \equiv 0 \mod \Lambda_+$. 
Then, motivated by the observation above, 
we {\it define}  
\begin{equation}
\mathfrak m_k^{\mathfrak b,\mathfrak x,\prime}(x_1,\ldots,x_k)
= \sum_{\ell,\beta}e^{\beta\cap \mathfrak b_0} \rho(\partial \beta)
\mathfrak q^{{\text{\rm can}}}_{\beta;\ell,*}(\mathfrak b_{\text{\rm high}}^{\otimes\ell};
e^{\mathfrak x_+} x_1
e^{\mathfrak x_+} \cdots e^{\mathfrak x_+} x_k e^{\mathfrak x_+}) T^{\beta\cap \omega/2\pi}
\end{equation}
where we write $\mathfrak x = \mathfrak x_0 + \mathfrak x_+$  and define $\rho$ using $\mathfrak x_0$ as in 
Definition \ref{rhoyx}.
\par
We can prove $A_{\infty}$ relation for $\mathfrak m_k^{\mathfrak b,\mathfrak x,\prime}$ by 
using the fact that $\beta \mapsto e^{\beta \cap \mathfrak b_0}$ 
is a homomorphism from $H_2(X,L;\Z)$ to $\Lambda_0 \setminus \Lambda_+$.
\par
We write $\mathfrak b_0 = \sum_{a=1}^m w_a D_a \in 
\mathcal {A}^2(\Lambda_0)$ with 
$w_a \in \Lambda_0$ and 
$w_a =\overline{w}_a + w_{a+} \in \C \oplus \Lambda_+ = \Lambda_0$. 
As in the calculation of the second equality in (\ref{94}), we have 
$$
e^{\beta \cap \mathfrak b_0} = 
\prod_{a=1}^m 
\left(\sum_{j=0}^{\infty} \frac{(\beta \cap D_a)^j}{j!} \right)
\left(\sum_{j=0}^{\infty} \frac{\overline{w}_a^j}{j!} \right)
\left(\sum_{j=0}^{\infty} \frac{w_{a+}^j}{j!}\right) 
 = 
\prod_{a=1}^m 
e^{\beta \cap D_a} e^{\overline{w}_a}e^{w_{a+}}.
$$
Here as we describe in 
Remark \ref{transsukosi},  
$\left(\sum \frac{(\beta \cap D_a)^j}{j!} \right)
\left(\sum \frac{\overline{w}_a^j}{j!} \right)$ converges to 
$e^{\beta \cap D_a} e^{\overline{w}_a}$ 
in the usual Archimedean topology 
and $\sum \frac{w_{a+}^j}{j!}$ converges to $e^{w_{a+}}=\mathfrak {w}_{a+}$ 
(see (\ref{defoffrakw})) in the adic 
non-Archimedean topology. 
Therefore we obtain  
$$
e^{\beta \cap \mathfrak b_0}
\mathfrak q_{\beta;\ell,*}^{{\text{\rm can}}}
(\mathfrak b_{\text{high}}^{\otimes\ell};e^{\mathfrak x_+} x_1
e^{\mathfrak x_+} \cdots e^{\mathfrak x_+} x_k e^{\mathfrak x_+})
= 
\mathfrak q_{\beta;\ell,*}^{{\text{\rm can}}}
(\mathfrak b^{\otimes\ell};e^{\mathfrak x_+} x_1
e^{\mathfrak x_+} \cdots e^{\mathfrak x_+} x_k e^{\mathfrak x_+}).
$$
In particular, we have 
$$
\mathfrak m_k^{\mathfrak b,\mathfrak x,\prime}(x_1,\ldots,x_k) 
= \sum_{\beta,\ell} T^{\beta\cap\omega/2\pi}\rho(\partial\beta)\mathfrak q_{\beta;\ell,*}^{{\text{\rm can}}}
(\mathfrak b^{\otimes\ell};e^{\mathfrak x_+} x_1
e^{\mathfrak x_+} \cdots e^{\mathfrak x_+} x_k e^{\mathfrak x_+})
$$
for $x_i \in H^1(L(u);\Lambda_+)$. 
Hence we have 
$$
\mathfrak{PO}_{\mathfrak b}^u(\mathfrak x+x) = \sum_{k=0}^{\infty}\mathfrak m_k^{\mathfrak b,\mathfrak x,\prime}(x,\ldots,x).
$$
\par
Moreover $\mathfrak m_k^{\mathfrak b,\mathfrak x,\prime}$ converges for any $\mathfrak b \in \mathcal A(\Lambda_0)$.
The proof of this convergence is similar to the proof of Theorem \ref{weakpotential} 
given in Section \ref{sec:CalPot}.
\par
We can use this operator $\mathfrak m_k^{\mathfrak b,\mathfrak x,\prime}$ in place of $\mathfrak m_k^{\mathfrak b,\mathfrak x}$.
It is then straightforward to see that the proof Theorem $\ref{homologynonzero}$ and Proposition $\ref{prof:bal2}$
go through after minor modification for $\mathfrak b \in \mathcal A(\Lambda_0)$.
\end{proof}
\begin{rem}\label{Rrationality}
Let $R$ be a field such that $\Q \subset R \subset \C$.
Even if we assume
$\mathfrak b = \sum \mathfrak b_a D_a \in \mathcal A(\Lambda_0(R))$,
it does not imply
\begin{equation}\label{bulkLOErat}
(\mathfrak{PO}^u_{\mathfrak b})_l  \in
R[y_{1,1},\ldots,y_{l,d(l)}^{-1}].
\end{equation}
In fact, $\exp(\overline{\mathfrak b}_{i(r,s)})$ may not be an element of $R$.
(This point is related to Remark \ref{transsukosi}.)
An appropriate condition for (\ref{bulkLOErat}) to hold is
$$
\exp(\mathfrak b_i) \in \Lambda_0(R)
$$
for $i=1,\ldots,m$.
\end{rem}
\begin{exm}\label{onepointbu}
We put
$$
P = \{ (u_1,u_2) \mid 0\le u_1,u_2, \,\, u_1+u_2 \le 1, \,\,
u_2 \le 2/3 \},
$$
which is a moment polytope of monotone one point blow up of $\C P^2$.
We consider $u = (1/3,1/3)$. Then $L(u)$ is a monotone
Lagrangian submanifold. 
Now we put
$$
D_2 = \pi^{-1}(\{ (u_1,u_2) \in P\mid u_2 = 0\}).
$$
Let
$
\mathfrak b_c = (\log c)[D_2], 
$
where $c \in \C\setminus \{0\}$. Proposition \ref{POcalcFano} implies
$$
\mathfrak{PO}_{\mathfrak b_c}^{u}(y_1,y_2) =
T^{1/3}(y_1+cy_2 + y_2^{-1} + (y_1y_2)^{-1}).
$$
Thus the critical point is given by
$$
1 - y_1^{-2}y_2^{-1} = 0 = c - y_2^{-2} - y_1^{-1}y_2^{-2}.
$$
The first equation gives $y_2=y_1^{-2}$. Hence the second equation becomes
\begin{equation}\label{4jiequ}
c - y_1^4 - y_1^3 = 0.
\end{equation}
(\ref{4jiequ}) has a nonzero multiple root $y_1 = -3/4$ if
$c = -27/256$.
\par
Namely if $
\mathfrak b = (\log (-27/256))[D_2]
$,
then $\mathfrak{PO}^{u}_{\mathfrak b}$ has a degenerate critical
point of type $A_2$.
\end{exm}
\begin{exm}\label{twopointbu}
We again consider the example of two points blow up in Section \ref{sec:Exam}. 
Namely its moment polytope is (\ref{momentpoly})
with $\beta = \frac{1-\alpha}{2}$. We consider the point $u= (\beta,\beta)$.  
We put $ D_2 = \pi^{-1}(\{ (u_1,u_2) \in P\mid
u_2 = 0\}) $, and consider $ \mathfrak b_c = (\log c)[D_2]. $ 
We have
$$
\mathfrak{PO}_{\mathfrak b_c}^{u}(y_1,y_2)
= T^{\beta}(cy_2+y_2^{-1}+y_1+y_1y_2) + T^{1-\beta}y_1^{-1}y_2^{-1}.
$$
The (generalized) leading term equation is
$$
c-y_2^{-2}+y_1 = 0 = 1+y_2.
$$
It has a nonzero solution $(1-c,-1)$ if $c\ne 1$.
Hence there exists $b$ such that
$$
HF((L(u),(\mathfrak b_c ,b)),(L(u),(\mathfrak b_c ,b));\Lambda)
\ne 0
$$
if and only if $c\ne 1$.
\par
If we deform only by $\mathfrak b \in \Lambda_+$,  then $c=1$. Namely
there is no such $b$ with nontrivial Floer cohomology. We remark
that $L(u)$ is bulk-balanced in the sense of Definition
\ref{balanced} since it is a limit of balanced fibers.
\par
The authors do not know an example of $L(u)$ that carries a pair
$(\mathfrak b,b)$ with $\mathfrak b \in \mathcal A(\Lambda_0)$, $b \in
H^1(L;\Lambda_0)$ for which we have
$$
HF((L(u),(\mathfrak b ,b)),(L(u),(\mathfrak b ,b));\Lambda) \ne 0,
$$
but which is not bulk-balanced in the sense of Definition
\ref{balanced}.
\end{exm}
\par
\section{Appendix: Continuous family of multisections}
\label{sec:integration}
\par
In this section we review the techniques of using a continuous
family of multisections and integration along the fiber on their
zero sets so that smooth correspondence by spaces with Kuranishi
structure induces a map between de Rham complex.
We used it in Section \ref{sec:HFbulk}.
\par
This technique is not new and is known to various people. In fact, 
\cite{Rua99} and Section 16 \cite{fukaya;abel} use a similar technique
and Section 7.5 \cite{fooo06} (= Section 33 \cite{fooo06pre}), \cite{fukaya:loop},
\cite{fukaya;operad}  contain almost the same argument as we
describe below. We include the details here for reader's convenience.
\par
Let $\mathcal M$ be a space with Kuranishi structure and let $ev_s:
\mathcal M \to M_s$, $ev_t: \mathcal M \to M_t$ be strongly
continuous smooth maps. (See Definition 6.6 \cite{FO} and the
description below.) (Here $s$ and $t$ stand for source and target,
respectively.) We assume our smooth manifolds $M_s, M_t$ are compact
and oriented without boundary. We also assume $\mathcal M$ has a
tangent bundle and is oriented in the sense of Kuranishi structure.
(See Definition A1.14 \cite{fooo06} =  Definition A1.14 \cite{fooo06pre} and the description below.)
\begin{rem}
We may relax the orientability assumption above by using local
coefficients in the same way as in Section A2 \cite{fooo06} (= Section A2 \cite{fooo06pre}). We do not
discuss it here since we do not need this generalization in this
paper.
\end{rem}
We include the case when $\mathcal M$ has a boundary or corner.
We assume that $ev_t$ is weakly submersive. (See A1.13 \cite{fooo06}
 (=  A1.13 \cite{fooo06pre}) and the description below.)
In this situation we will construct the map
\begin{equation}\label{smoothcorrespondence}
(\mathcal M;ev_s,ev_t)_*: \Omega^kM_s \to \Omega^{k+\dim M_t - \dim \mathcal M}M_t.
\end{equation}
We call  (\ref{smoothcorrespondence}), the
{\it smooth correspondence map associated to}  $(\mathcal M;ev_s,ev_t)$.
\par
The space ${\mathcal M}$
is covered by a finite number of Kuranishi charts $$(V_{\alpha},E_{\alpha},
\Gamma_{\alpha},\psi_{\alpha},s_{\alpha}), 
\quad \alpha \in \mathfrak A.
$$
They satisfy \cite{fooo08} Condition B.1.
\par
We assume that
$\{(V_{\alpha},E_{\alpha},\Gamma_{\alpha},
\psi_{\alpha},s_{\alpha})\mid \alpha \in \mathfrak A\}$ is a good coordinate system 
in the sense of Definition 6.1  \cite{FO} or Lemma A1.11 \cite{fooo06} (= 
 Lemma A1.11 \cite{fooo06pre}).
This means the following:
The set $\mathfrak A$ has a partial order $<$, where
either $\alpha_1 \le \alpha_2$ or $\alpha_2 \le \alpha_1$ holds
for $\alpha_1, \alpha_2 \in \mathfrak A$
if
$$
\psi_{\alpha_1}(s_{\alpha_1}^{-1}(0)/\Gamma_{\alpha_1}) \cap
\psi_{\alpha_2}(s_{\alpha_2}^{-1}(0)/\Gamma_{\alpha_2}) \ne \emptyset.
$$
Let $\alpha_1, \alpha_2 \in \mathfrak A$ and $\alpha_1 \le
\alpha_2$. Then there exist a $\Gamma_{\alpha_1}$-invariant open
subset $V_{\alpha_2,\alpha_1} \subset V_{\alpha_1}$, a smooth
embedding
$$
\varphi_{\alpha_2,\alpha_1}:  V_{\alpha_2,\alpha_1}
\to V_{\alpha_2}
$$
and a bundle map
$$
\widehat{\varphi}_{\alpha_2,\alpha_1}:  E_{\alpha_1}\vert_{V_{\alpha_2,\alpha_1}}
\to E_{\alpha_2}
$$
which covers  $\varphi_{\alpha_2,\alpha_1}$.
Moreover there exists an injective homomorphism
$$
\widehat{\widehat{\varphi}}_{\alpha_2,\alpha_1}
: \Gamma_{\alpha_1} \to \Gamma_{\alpha_2}.
$$
We require that they satisfy \cite{fooo08} Conditions B.2, B.3.
\par
A strongly continuous smooth map 
$ev_t: \mathcal M \to M_t$ is a
family of $\Gamma_{\alpha}$ invariant smooth maps
\begin{equation}\label{evalpha}
ev_{t;\alpha}: V_{\alpha} \to M_t
\end{equation}
which induces
\begin{equation}
\overline{ev}_{t;\alpha}: V_{\alpha}/\Gamma_{\alpha} \to M_t
\nonumber\end{equation}
such that
$$
\overline{ev}_{t;\alpha_2} \circ \overline{\varphi}_{\alpha_2,\alpha_1}
= \overline{ev}_{t;\alpha_1}
$$
on $V_{\alpha_2,\alpha_1}/\Gamma_{\alpha}$. (Note $\Gamma_{\alpha}$
action on $M_t$ is trivial.) $ev_s: \mathcal M \to M_s$ consists of
a similar family, $ ev_{s;\alpha}: V_{\alpha} \to M_s $.
\par
Our assumption that $ev_t$ is weakly submersive
means that each of ${ev}_{t;\alpha}$ in (\ref{evalpha}) is
a submersion.
\par
We next review on the multisections. (See Section 3 \cite{FO} for detail.) Let
$(V_{\alpha},E_{\alpha},\Gamma_{\alpha},\psi_{\alpha},s_{\alpha})$
be a Kuranishi chart of $\mathcal M$. For $x \in V_{\alpha}$ we
consider the fiber $E_{\alpha,x}$ of the bundle $E_{\alpha}$ at $x$.
We take its $l$ copies and consider the direct product $
E_{\alpha,x}^{l} $. We take the quotient thereof by the action of
symmetric group of order $l!$ and let $\mathcal S^l(E_{\alpha,x})$
be the quotient space. There exists a map
$$
tm_m: \mathcal S^l(E_{\alpha,x})
\to \mathcal S^{lm}(E_{\alpha,x}),
$$
which sends $[a_1,\ldots,a_l]$ to
$$
[\,\underbrace {a_1,\ldots,a_1}_{\text{$m$ copies}},\ldots,
\underbrace {a_l,\ldots,a_l}_{\text{$m$ copies}}].
$$
A smooth {\it multisection} $s$ of the orbibundle
$$
E_{\alpha} \to V_{\alpha}
$$
consists of an open covering
$$
\bigcup_iU_i = V_{\alpha}
$$
and $s_i$ which maps $x\in U_i$ to $s_i(x) \in \mathcal
S^{l_i}(E_{\alpha,x})$. They are required to satisfy \cite{fooo08} Condition B.12.
\par
We identify two multisections $(\{U_i\},\{s_i\},\{l_i\})$,  $(\{U'_i\},\{s'_i\},\{l'_i\})$
if
$$
tm_{l_j}(s_i(x)) = tm_{l'_i}(s'_j(x)) \in \mathcal S^{l_il'_j}(E_{\alpha,\gamma x})
$$
on $U_i \cap U'_j$.
We say $s_{i,j}$ to be a {\it branch} of $s_i$.
\par
We next discuss a continuous family of multisections and their
transversality. Let $W_{\alpha}$  be a finite dimensional smooth
oriented manifold and consider the pull-back bundle
$$
\pi_\alpha^*E_\alpha \to W_\alpha \times V_\alpha
$$
under the projection $\pi_\alpha: W_\alpha \times V_\alpha \to
V_\alpha$. The action of $\Gamma_{\alpha}$ on $W_{\alpha}$ is, by
definition, trivial.
\begin{defn}
\begin{enumerate}
\item
A {\it $W_{\alpha}$-parameterized family $\mathfrak s_{\alpha}$ of
multisections} is by definition a multisection of
$\pi_\alpha^*E_\alpha$.
\item
We fix a metric of our bundle $E_{\alpha}$. We say that $\mathfrak
s_{\alpha}$ is {\it $\epsilon$-close} to $s_{\alpha}$ in $C^0$ topology if
the following holds. Let $(w,x) \in W_{\alpha} \times V_{\alpha}$.
Then for each branch $\mathfrak s_{\alpha,i,j}$ of $\mathfrak s_{\alpha}$ we
have
$$
\vert \mathfrak s_{\alpha,i,j}(w,\ldots) - s_{\alpha}(\ldots)\vert_{C^0} < \epsilon
$$
in a neighborhood of $x$.
\item
$\mathfrak s_{\alpha}$ is said to be {\it transversal to $0$} if each branch
$\mathfrak s_{\alpha,i,j}$ of $\mathfrak s_{\alpha}$ is transversal to $0$.
\item
Let $f_{\alpha}: V_{\alpha} \to M$ be a
$\Gamma_{\alpha}$-equivariant smooth map. We assume that $\mathfrak
s_{\alpha}$ is transversal to $0$. We then say that
$f_{\alpha}\vert_{\mathfrak s_{\alpha}^{-1}(0)}$ is a {\it submersion} if the
following holds: Let $(w,x) \in W_{\alpha} \times V_{\alpha}$. Then
for each branch $\mathfrak s_{\alpha,i,j}$ of $\mathfrak s_{\alpha}$ the
restriction of
$$
f_{\alpha} \circ \pi_{\alpha}: W_{\alpha} \times V_{\alpha} \to M
$$
to
\begin{equation}\label{fras-1(0)}
\{ (w,x) \mid \mathfrak s_{\alpha,i,j}(w,x) = 0\}
\end{equation}
is a submersion. We remark that (\ref{fras-1(0)}) is a smooth manifold by our assumption.
\end{enumerate}
\end{defn}
\begin{rem}
In case $\mathcal M$ has a boundary or a corner, so does
(\ref{fras-1(0)}). In this case we require that the restriction of
$f_{\alpha}$ to each of the stratum of (\ref{fras-1(0)}) is a
submersion.
\end{rem}
\begin{lem}\label{existencesubmersion}
We assume that $f_{\alpha}: V_{\alpha} \to M$ is a submersion. Then
there exists $W_{\alpha}$ such that for any  $\epsilon$ there exists
a $W_{\alpha}$-parameterized family $\mathfrak s_{\alpha}$ of
multisections which is $\epsilon$-close to $s_{\alpha}$, transversal
to $0$ and such that $f_{\alpha}\vert_{\mathfrak s_{\alpha}^{-1}(0)}$ is
a submersion.
\par
If $\mathfrak s_{\alpha}$ is already given and satisfies the required condition
on a neighborhood of a
$\Gamma_{\alpha}$ invariant compact set $K_{\alpha}
\subset V_{\alpha}$, then we may
extend it to the whole $V_{\alpha}$ without changing it on $K_{\alpha}$.
\end{lem}
In the course of the proof of Lemma \ref{existencesubmersion}, we
need to shrink $V_{\alpha}$ slightly. We do not mention it
explicitly.
\begin{proof}
We may choose $W_{\alpha}$ to be a vector space of sufficiently large dimension
so that there exists a surjective bundle map
\begin{equation}\label{sur}
\text{\rm Sur}: W_{\alpha} \times V_{\alpha}  \to E_{\alpha}.
\end{equation}
We remark that (\ref{sur}) is not necessarily
$\Gamma_{\alpha}$-equivariant. We put
$$
\mathfrak s^{(1)}_\alpha(w,x) = \text{\rm Sur}(w,x) + s_{\alpha}(x)
$$
and
$$
\mathfrak s^{(2)}_\alpha(w,x) = [\gamma_1\mathfrak
s^{(1)}_\alpha(w,x),\ldots,\gamma_g\mathfrak s^{(1)}_\alpha(w,x)],
$$
where $\Gamma_{\alpha} = \{\gamma_1,\ldots,\gamma_g\}$. $\mathfrak
s^{(2)}$ defines a multisection on $W_\alpha \times V_\alpha$ which
is transversal to $0$ by construction. Moreover since $(\mathfrak
s^{(2)}_\alpha)^{-1}(0) \to V_{\alpha}$ is a submersion it follows
from assumption that $f_{\alpha}\vert_{(\mathfrak
s^{(2)}_\alpha)^{-1}(0)}$ is a submersion. By replacing $W_{\alpha}$
to a small neighborhood of $0$, we can choose $\mathfrak s^{(2)}_\alpha$
which is sufficiently close to $s_{\alpha}$.
\par
The last part of the lemma can be proved by using an appropriate
partition of unity in the same way as in Section 3 \cite{FO}.
\end{proof}
Now let $\theta_{\alpha}$ be a smooth differential form of compact
support on $V_{\alpha}$. We assume that $\theta_{\alpha}$ is
$\Gamma_{\alpha}$-invariant. Let $f_{\alpha}: V_{\alpha} \to M$ be
a $\Gamma_{\alpha}$ equivariant submersion. (The $\Gamma_{\alpha}$
action on $M$ is trivial.) Let $\mathfrak s_{\alpha}$ satisfy the
conclusion of Lemma \ref{existencesubmersion}. We put a smooth
measure $\omega_{\alpha}$ on $W_\alpha$ of compact support with
total mass $1$. By fixing an orientation on $W_\alpha$, we regard
$\omega_{\alpha}$ as a differential form of top degree. We have
\begin{equation}\label{totalmass1}
\int_{W_{\alpha}} \omega_{\alpha} = 1.
\end{equation}
We next define {\it integration along the fiber}
$$
((V_{\alpha},E_{\alpha},\Gamma_{\alpha},\psi_{\alpha},s_{\alpha}),
(W_{\alpha},\omega_{\alpha}),\mathfrak s_{\alpha},f_{\alpha})_* (\theta_{\alpha})
\in \Omega^{\deg \theta_{\alpha} + \dim M - \dim \mathcal M}(M).
$$
Let $(U_i,\mathfrak s_{\alpha,i})$ be a representative of $\mathfrak
s_{\alpha}$. Namely $\{U_i\mid i\in I\}$ is an open covering of
$W_{\alpha} \times V_{\alpha}$ and $s_{\alpha}$ is represented by
$\mathfrak s_{\alpha,i}$ on $U_i$. By the definition of the
multisection, $U_i$ is $\Gamma_{\alpha}$-invariant. We may shrink
$U_i$, if necessary, so that there exists a lifting $\tilde{\mathfrak
s}_{\alpha,i} = (\tilde{\mathfrak s}_{\alpha,i,1},\ldots,\tilde{\mathfrak
s}_{\alpha,i,l_i})$ as in \cite{fooo08} (B.10).
\par
Let $\{\chi_i \mid i \in I\}$ be a partition of unity subordinate to
the covering $\{U_i\mid i\in I\}$. By replacing $\chi_i$ with its
average over $\Gamma_{\alpha}$ we may assume $\chi_i$ is
$\Gamma_{\alpha}$-invariant.
We put
\begin{equation}
\tilde{\mathfrak s}_{\alpha,i,j}^{-1}(0)
=  \{(w,x) \in U_i \mid \tilde{\mathfrak s}_{\alpha,i,j}(w,x) = 0\}.
\end{equation}
By assumption $\tilde{\mathfrak s}_{\alpha,i,j}^{-1}(0)$ is a smooth manifold and
\begin{equation}\label{restrictsubmersion}
f_{\alpha} \circ \pi_{\alpha}\vert_{\tilde{\mathfrak s}_{\alpha,i,j}^{-1}(0)}:
\tilde{\mathfrak s}_{\alpha,i,j}^{-1}(0) \to M
\end{equation}
is a submersion.
\begin{defn}
We define
\begin{equation}\label{intfibermainformula}
\aligned
&((V_{\alpha},E_{\alpha},\Gamma_{\alpha},\psi_{\alpha},s_{\alpha}),
(W_{\alpha},\omega_{\alpha}),\mathfrak s_{\alpha},f_{\alpha})_* (\theta_{\alpha})\\
&=
\frac{1}{\#\Gamma_{\alpha}}\sum_{i=1}^I\sum_{j=1}^{l_i} \frac{1}{l_i}
(f_{\alpha} \circ \pi_{\alpha}\vert_{\tilde{\mathfrak s}_{\alpha,i,j}^{-1}(0)})_!((\chi_i
\pi_{\alpha}^*\theta_{\alpha} \wedge \omega_{\alpha})\vert_{\tilde{\mathfrak s}_{\alpha,i,j}^{-1}(0)}).
\endaligned
\end{equation}
Here $(f_{\alpha} \circ \pi_{\alpha}\vert_{\tilde{\mathfrak s}_{\alpha,i,j}^{-1}(0)})_!$ is the
integration along the fiber of the smooth submersion (\ref{restrictsubmersion}).
\end{defn}
\begin{lem}
The right hand side of $(\ref{intfibermainformula})$ depends only on
$(V_{\alpha},E_{\alpha},\Gamma_{\alpha},\psi_{\alpha},s_{\alpha})$,
$(W_{\alpha},\omega_{\alpha})$, $\mathfrak s_{\alpha}$, $f_{\alpha}$,
and $\theta_{\alpha}$ but independent of the following choices:
\begin{enumerate}
\item The choice of representatives $(\{U_i\},\mathfrak s_{\alpha,i})$ of $\mathfrak s_{\alpha}$.
\item The lifting $\tilde{\mathfrak s}_{\alpha,i}$.
\item
The partition of unity $\chi_i$.
\end{enumerate}
\end{lem}
\begin{proof}
The proof is straightforward generalization of the proof of
well-definedness of integration on manifold, which can be found in
text books of manifold theory, and is left to the reader.
\end{proof}
\par
So far we have been working on one Kuranishi chart
$(V_{\alpha},E_{\alpha},\Gamma_{\alpha},\psi_{\alpha},s_{\alpha})$.
We next describe the compatibility conditions among the
$W_\alpha$-parameterized families of multisections for various
$\alpha$. During the construction we need to shrink $V_{\alpha}$ a
bit several times. We will not mention  explicitly this point
henceforth.
\begin{rem}
The discussion below is a $W_{\alpha}$ parametrized version 
of \cite{fooo08} Section C.
\end{rem}
\par
Let $\alpha_1 < \alpha_2$.
For $\alpha_1 < \alpha_2$, we take the normal
bundle $N_{V_{\alpha_2,\alpha_1}}V_{\alpha_2}$ of
$\varphi_{\alpha_2,\alpha_1}(V_{\alpha_2,\alpha_1})$ in
$V_{\alpha_2}$.
We use an appropriate $\Gamma_{\alpha_2}$ invariant Riemannian
metric on $V_{\alpha_2}$ to define
the exponential map
\begin{equation}\label{normalexp}
\text{Exp}_{\alpha_2,\alpha_1} :
B_{\e}N_{V_{\alpha_2,\alpha_1}}V_{\alpha_2} \to V_{\alpha_2}.
\end{equation}
(Here $B_{\e}N_{V_{\alpha_2,\alpha_1}}V_{\alpha_2}$ is the $\epsilon$ neighborhood of the
zero section of $N_{V_{\alpha_2,\alpha_1}}V_{\alpha_2}$.)
\par
Using $\text{Exp}_{\alpha_2,\alpha_1}$, we identify
$B_{\e}N_{V_{\alpha_2,\alpha_1}}V_{\alpha_2}/{\Gamma_{\alpha_1}}
$  to an open subset of $V_{\alpha_1}/{\Gamma_{\alpha_1}}$
and denote it by $U_{\epsilon}({V_{\alpha_2,\alpha_1}}/{\Gamma_{\alpha_1}})$.
\par
Using the projection
$$
\text{\rm Pr}_{V_{\alpha_2,\alpha_1}} :
U_{\epsilon}({V_{\alpha_2,\alpha_1}}/{\Gamma_{\alpha_1}}) \to {V_{\alpha_2,\alpha_1}}/{\Gamma_{\alpha_1}},
$$
we extend the orbibundle $E_{\alpha_1}$ to $U_{\epsilon}({V_{\alpha_2,\alpha_1}}/{\Gamma_{\alpha_1}})$.
Also we extend the embedding $E_{\alpha_1} \to {\widehat{\varphi}_{\alpha_2,\alpha_1}^*E_{\alpha_2}}$,
(which is induced by $\widehat{\varphi}_{\alpha_2,\alpha_1}$) to
$U_{\epsilon}({V_{\alpha_2,\alpha_1}}/{\Gamma_{\alpha_1}})$.
\par
We fix a $\Gamma_{\alpha}$-invariant inner product of the bundles
$E_{\alpha}$. We then have a bundle isomorphism
\begin{equation}\label{bdliso}
E_{\alpha_2} \cong E_{\alpha_1} \oplus \frac{\widehat{\varphi}_{\alpha_2,\alpha_1}^*E_{\alpha_2}}{E_{\alpha_1}}
\end{equation}
on $U_\e(V_{\alpha_2,\alpha_1}/\Gamma_{\alpha_1})$. We can use
\cite{fooo08} Condition B.3  to modify
$\text{Exp}_{\alpha_2,\alpha_1}$ in (\ref{normalexp}) so that 
\cite{fooo08} Condition C.3 is satisfied.
\par
Let $W_{\alpha_1}$ be a finite dimensional manifold and $\mathfrak
s_{\alpha_1}$ a multisection of $\pi_{\alpha_1}^*E_{\alpha_1}$ on
$W_{\alpha_1} \times V_{\alpha_1}$. We put $W_{\alpha_2} =
W_{\alpha_1} \times W'$, where $W'$ will be defined later.
\begin{defn}\label{compdefini}
A multisection $\mathfrak s_{\alpha_2}$ of $W_{\alpha_2} \times V_{\alpha_2}$
is said to be {\it compatible} with $\mathfrak s_{\alpha_1}$ if
the following holds for each
$y = \text{Exp}_{\alpha_2,\alpha_1}(\tilde y) \in U_{\epsilon}({V_{\alpha_2,\alpha_1}}/{\Gamma_{\alpha_1}})$.
\begin{equation}\label{compatibilityforalpha}
\mathfrak s_{\alpha_2}((w,w'),y)
= \mathfrak s_{\alpha_1}(w,\text{\rm Pr}(\tilde y)) \oplus ds_{\alpha_2}(\tilde y \mod TV_{\alpha_1}).
\end{equation}
\end{defn}
We remark that on the right hand side of  
(\ref{compatibilityforalpha}), 
$\mathfrak s_{\alpha_1}(w,\text{\rm Pr}(\tilde y))$ is a
multisection of $\pi_{\alpha_1}^*E_{\alpha_1}$ and
$ds_{\alpha_2}(\tilde y \mod TV_{\alpha_1})$ is a (single valued)
section. Therefore using (\ref{bdliso}) the right hand side of
(\ref{compatibilityforalpha}) is an element of $\mathcal
S^{l_i}(E_{\alpha_2})_x$ ($x = \text{\rm Pr}(\tilde y)$), and hence
is regarded as  a multisection of $\pi_{\alpha_2}^*E_{\alpha_2}$. In
other words, we omit $\widehat{\varphi}_{\alpha_2,\alpha_1}$ in
(\ref{compatibilityforalpha}).
\par
\cite{fooo08} Condition C.3 implies that the original Kuranishi
map $s_{\alpha}$ satisfies the compatibility condition
(\ref{compatibilityforalpha}). We use this and (the proof of) Lemma
\ref{existencesubmersion} and prove the following. Let $ev_t:
\mathcal M \to M_t$ be a weakly submersive strongly smooth map. We
choose a good coordinate system
$(V_{\alpha},E_{\alpha},\Gamma_{\alpha},\psi_{\alpha},s_{\alpha})$
and let $ev_{t,\alpha}: V_{\alpha} \to M_t$ be a local
representative of $ev_t$.
\begin{lem}\label{existsfras}
We have $W_{\alpha}$ such that
for each $\epsilon$ there exists $\mathfrak s_{\alpha}$ which is a $W_{\alpha}$-parameterized
family of multisections with the following properties.
\begin{enumerate}
\item $\mathfrak s_{\alpha}$ is transversal to $0$.
\par
\item $ev_{t,\alpha}\vert_{\mathfrak s_{\alpha}^{-1}(0)}$ is a submersion.
\item $\mathfrak s_{\alpha}$ is $\epsilon$-close to $s_{\alpha}$.
\item $\mathfrak s_{\alpha_2}$ is compatible with $\mathfrak s_{\alpha_1}$
for each $\alpha_1 < \alpha_2$.
\end{enumerate}
If $\{\mathfrak s_{\alpha}\}$ is already defined and satisfies $(1)$ -
$(4)$ on a neighborhood of a compact set $K \subset \mathcal M$,
then we may choose $\mathfrak s_{\alpha}$ without changing it on $K$.
\end{lem}
\begin{proof}
The proof is by induction on $\alpha$. (We remark that $\mathfrak A$
(the totality of $\alpha$'s) is partially ordered.) For minimal
$\alpha$ we use Lemma \ref{existencesubmersion} to prove existence
of $\mathfrak s_{\alpha}$. If we have constructed $\mathfrak s_{\alpha'}$
for every $\alpha'$ smaller than $\alpha$, then we use
(\ref{compatibilityforalpha}) to define $\mathfrak s_{\alpha}$ on a
neighborhood of the images of $V_{\alpha,\alpha'}$ for various
$\alpha' <\alpha$. They coincide on the overlapped part by the
induction hypothesis and \cite{fooo08} Condition B.2. \cite{fooo08} Condition
C.3 then implies that this is still 
$\epsilon$-close to $s_{\alpha}$. Therefore we can use Lemma
\ref{existencesubmersion} (the relative version) to extend it and
obtain  $\mathfrak s_{\alpha}$. (We choose $W'$ at this step.)
\par
The proof of the last statement is similar.
\end{proof}
\par
We choose measures $\omega_{\alpha}$ on $W_\alpha$ such that the
measure $\omega_{\alpha_2}$ is a direct product measure
$\omega_{\alpha_1} \times \omega'$ on $W_\alpha \times W'$ if
$\alpha_1 < \alpha_2$.
\par
We next choose a partition of unity $\chi_{\alpha}$ subordinate to
our Kuranishi charts. To define the notion of partition of unity, we
need some notation. For $\alpha_1 < \alpha_2$, we take the normal
bundle $N_{V_{\alpha_2,\alpha_1}}V_{\alpha_2}$ of
$\varphi_{\alpha_2,\alpha_1}(V_{\alpha_2,\alpha_1})$ in
$V_{\alpha_2}$. Let $\text{Pr}_{V_{\alpha_2,\alpha_1}}:
N_{V_{\alpha_2,\alpha_1}}V_{\alpha_2} \to V_{\alpha_2,\alpha_1}$ be
the projection. We fix a $\Gamma_{\alpha_1}$-invariant positive
definite metric of $N_{V_{\alpha_2,\alpha_1}}V_{\alpha_2}$ and let
$\rho_{\alpha_2,\alpha_1}: N_{V_{\alpha_2,\alpha_1}}V_{\alpha_2} \to
[0,\infty)$ be the norm with respect to this metric. 
(Actually we modify it so that it becomes compatible.)
We fix a
sufficiently small $\delta$ and let $\chi^{\delta}: \R \to [0,1]$
be a smooth function such that
$$
\chi^{\delta}(t)
=
\begin{cases}
0  &t \ge \delta \\
1  &t \le \delta/2.
\end{cases}
$$

Let $U_{\delta}(V_{\alpha_2,\alpha_1}/\Gamma_{\alpha_1})$
be the image of the exponential map. Namely
$$
U_{\delta}(V_{\alpha_2,\alpha_1}/\Gamma_{\alpha_1}) = \{\text{Exp}(v) \mid v \in  N_{V_{\alpha_2,\alpha_1}}V_{\alpha_2}/
\Gamma_{\alpha_1} \mid \rho_{\alpha_2,\alpha_1}(v) < \delta\}.
$$
We push out our function $\rho_{\alpha_2,\alpha_1}$ to
$U_{\delta}(V_{\alpha_2,\alpha_1}/\Gamma_{\alpha_1})$  and
denote it by the same symbol.
It is called a {\it tubular distance function}.
Following \cite{Math73} we assume the following compatibility condition for
various tubular neighborhoods and tubular distance functions. 
\par
Let $\alpha_1 < \alpha_2 < \alpha_3$. We assume the 
following equalities hold on the domain where both sides are defined.
\begin{subequations}
\begin{eqnarray}
\text{Pr}_{V_{\alpha_2,\alpha_1}} \circ \text{Pr}_{V_{\alpha_3,\alpha_2}} 
&=& \text{Pr}_{V_{\alpha_3,\alpha_1}} \\
\rho_{\alpha_2,\alpha_1} \circ \text{Pr}_{V_{\alpha_3,\alpha_2}}
&=& \rho_{\alpha_3,\alpha_1}.
\end{eqnarray}
\end{subequations}
 The existence of such system is proved in \cite{Math73} 
 in more difficult situation than ours. (See also Section 35.2 \cite{fooo06pre}.)
\par
Let $x \in V_{\alpha}$.
We put
$$\aligned
\mathfrak A_{x,+} &= \{ \alpha_+  \mid x \in V_{\alpha_+,\alpha}, \,  \alpha_+ > \alpha\} \\
\mathfrak A_{x,-} &= \{ \alpha_-  \mid [x \mod\Gamma_{\alpha}] \in
U_{\delta}(V_{\alpha,\alpha_-}/\Gamma_{\alpha_-}),\, \alpha_- <
\alpha\}.
\endaligned$$
For $\alpha_- \in \mathfrak A_{x,-}$ we take $x_{\alpha_-}$ such that
$\text{Exp}(x_{\alpha_-}) = x$.
\begin{defn}\label{partitionofunity}
A system $\{\chi_{\alpha} \mid \alpha \in \mathfrak A\}$ of
$\Gamma_{\alpha}$-equivariant smooth functions $\chi_{\alpha}:
V_{\alpha} \to [0,1]$ of compact support is said to be a partition
of unity subordinate to our Kuranishi chart if:
$$
\chi_{\alpha}(x) + \sum_{\alpha_- \in \mathfrak A_{x,-}}
\chi^{\delta}(\rho_{\alpha,\alpha_-}(x))
\chi_{\alpha_-}(\text{Pr}_{V_{\alpha,\alpha_-}}
(x_{\alpha -}))
+ \sum_{\alpha_+ \in \mathfrak A_{x,+}} \chi_{\alpha_+}(\varphi_{\alpha_+,\alpha}(x))
= 1.
$$
\end{defn}
In \cite{fooo08} Lemma C.6 we proved the existence of partition of 
unity 
subordinate to our Kuranishi chart.
\par
Now we consider the situation we start with.
Namely we have two strongly continuous smooth maps
$$
ev_{s}: \mathcal M \to M_s, \qquad ev_{t}: \mathcal M \to M_t
$$
and $ev_{t}$ is weakly submersive. Let $h$ be a differential
form on $M_s$. We choose
$((V_{\alpha},E_{\alpha},\Gamma_{\alpha},\psi_{\alpha},s_{\alpha}),
(W_{\alpha},\omega_{\alpha}),\mathfrak s_{\alpha})$ which satisfies
(1) - (4) of Lemma \ref{existsfras}. We also choose a
partition of unity $\chi_{\alpha}$ subordinate to our Kuranishi chart.
We put
\begin{equation}\label{thetaalphadef}
\theta_{\alpha} = \chi_{\alpha}(ev_{s} \circ \pi_{\alpha})^* h
\end{equation}
which is a differential form on $W_{\alpha} \times V_{\alpha}$.
\begin{defn}\label{cordefinitionss}
We define
\begin{equation}\label{correspondfinalformula}
(\mathcal M;ev_s,ev_t)_* (h)
= \sum_{\alpha} ((V_{\alpha},\Gamma_{\alpha},E_{\alpha},\psi_{\alpha},s_{\alpha}),
(W_{\alpha},\omega_{\alpha}),\mathfrak s_{\alpha},ev_{t,\alpha})_* (\theta_{\alpha}).
\end{equation}
This is a smooth differential form on $M_t$.
\end{defn}
\begin{rem}\label{detasdenote}
\begin{enumerate}
\item
Actually the right hand side of (\ref{correspondfinalformula}) {\it depends}
on the choice of $((V_{\alpha},E_{\alpha},\Gamma_{\alpha},\psi_{\alpha},s_{\alpha}),
(W_{\alpha},\omega_{\alpha}),\mathfrak s_{\alpha})$.
We write $\mathfrak s$ to demonstrate this choice and write
$(\mathcal M;ev_s,ev_t,\mathfrak s)_* (h)$.
\item The right hand side of (\ref{correspondfinalformula}) is independent of the choice of
partition of unity. The proof is similar to the
well-definedness of integration on manifolds.
\end{enumerate}
\end{rem}
In case $\mathcal M$ has a boundary $\partial \mathcal M$, the choice
$((V_{\alpha},E_{\alpha},\Gamma_{\alpha},\psi_{\alpha},s_{\alpha}),
(W_{\alpha},\omega_{\alpha}),\mathfrak s_{\alpha})$ on $\mathcal M$ induces one for
$\partial \mathcal M$. We then have the following:
\begin{lem}[Stokes' theorem]\label{stokes}
We have
\begin{equation}\label{stoformula}
d((\mathcal M;ev_s,ev_t,\mathfrak s)_* (h))
= (\mathcal M;ev_s,ev_t,\mathfrak s)_* (dh) +
(\partial\mathcal M;ev_s,ev_t,\mathfrak s)_* (h).
\end{equation}
\end{lem}
We will discuss the sign at the end of this section.
\begin{proof}
Using the partition of unity $\chi_{\alpha}$ it suffices to consider
the case when $\mathcal M$ has only one Kuranishi chart
$V_{\alpha}$. We use the open covering $U_i$ of $V_{\alpha}$ and the
partition of unity again to see that we need only to study on one
$U_i$. In that case (\ref{stoformula}) is immediate from the usual
Stokes' formula.
\end{proof}
We consider the following situation. We assume $\mathcal M$ is a
space with Kuranishi structure with corners. Let $\partial_{c}
\mathcal M$, $c=1,\cdots,C$ be a decomposition of the boundary
$\partial \mathcal M$ into components. The intersection
$\partial_{c} \mathcal M \cap \partial_{c'} \mathcal M$ is a
codimension 2 stratum of $\mathcal M$ if it is nonempty. We denote
it by $\partial_{cc'} \mathcal M$. (Actually there may be a case
where there is a self intersection of  $\partial_{c}\mathcal M$ with
itself. If it occurs there is a codimension 2 stratum of $\mathcal
M$ corresponding to the self intersection points. We write it as
$\partial_{cc} \mathcal M$.) $\partial_{c} \mathcal M$ is regarded
as a space with Kuranishi structure which we denote by the same
symbol. (This is slightly imprecise in case there is a self
intersection. Since the way to handle it is rather obvious we do not
discuss it here.) The boundary of $\partial_{c} \mathcal M$  is the
union of $\partial_{cc'} \mathcal M$ for various $c'$. (Actually we
include the case $c'=c$. In that case we take two copies of
$\partial_{cc}M$, which become components of the boundary of
$\partial_{c} \mathcal M$.)
\par
Now we have the following:
\begin{lem}\label{cornerextend}
Suppose that there exist data $\mathfrak s_c$ as in Remark $\ref{detasdenote}$
$(1)$ on each of $\partial_{c} \mathcal M$
and the
restriction of $\mathfrak s_c$ to $\partial_{cc'} \mathcal M$ coincides
with the restriction of  $\mathfrak s_{c'}$ to $\partial_{cc'} \mathcal
M$. We also assume a similar compatibility at the self intersection
$\partial_{cc} \mathcal M$.
\par
Then there exists a datum $\mathfrak s$ on $\mathcal M$ whose
restriction to $\partial_{c} \mathcal M$ is $\mathfrak s_c$ for each
$c$.
\end{lem}
\begin{proof}
Using the compatibility condition we assumed, we can define $\mathfrak s$ in a neighborhood of
the union $\partial_c \mathcal M$ over $c$.
We can then extend it by using Lemma \ref{existsfras}.
\end{proof}
\par
We next discuss composition of smooth correspondences.
We consider the following situation.
Let
$$
ev_{s;st}: \mathcal M_{st} \to M_s, \qquad ev_{t;st}: \mathcal M_{st} \to M_t
$$
be as before such that $ev_{t;st}$ is weakly submersive.
Let
$$
ev_{r;rs}: \mathcal M_{rs} \to M_r, \qquad ev_{s;rs}: \mathcal M_{rs} \to M_s
$$
be a similar diagram such that $ev_{s;rs}$ is weakly submersive.
We use the fact that $ev_{s;rs}$ is weakly submersive to define the fiber product
$$
\mathcal M_{rs}\, {}_{ev_{s;rs}}\times_{ev_{s;st}} \mathcal M_{st}
$$
as a space with Kuranishi structure.
We write it as $\mathcal M_{rt}$. We have a diagram of strongly continuous
smooth maps
$$
ev_{r;rt}: \mathcal M_{rt} \to M_r, \qquad ev_{t;rt}: \mathcal M_{rt} \to M_t.
$$
It is easy to see that $ev_{t;rt}$ is weakly submersive.
\par
We next make choices $\mathfrak s^{st}$, $\mathfrak s^{rs}$ for $\mathcal M_{st}$ and
$\mathcal M_{rs}$. It is easy to see that it determines a
choice $\mathfrak s^{rt}$ for $\mathcal M_{rt}$.
\par
Now we have:
\begin{lem}[Composition formula]\label{compformula}
We have the following formula for each differential form $h$ on $M_r$.
\begin{equation}
\aligned
&(\mathcal M_{rt};ev_{r;rt},ev_{t;,rt},\mathfrak s^{rt})_* (h)\\
&= ((\mathcal M_{st};ev_{s;st},ev_{t;,st},\mathfrak s^{st})_*
\circ (\mathcal M_{rs};ev_{r;rs},ev_{s;,rs},\mathfrak s^{rs})_*)(h).
\endaligned
\end{equation}
\end{lem}
\begin{proof}
Using a partition of unity, it suffices to study locally on $\mathcal
M_{rs}$, $\mathcal M_{st}$. In that case it suffices to consider the
case of usual manifold, which is well-known.
\end{proof}
We finally discuss the signs in Lemmas \ref{stokes} and
\ref{compformula}. It is rather cumbersome to fix appropriate sign
convention and show those lemmata with sign. So, instead, we use the
trick of Subsection 8.10.3 \cite{fooo06} (= Section 53.3 \cite{fooo06pre}) (see also Section 13
\cite{fukaya;operad}) to reduce the orientation problem to the case
which is already discussed in Chapter 8 \cite{fooo06} (= Chapter 9 \cite{fooo06pre}), as follows.
\par
For generic $w \in W_{\alpha}$, the space $\mathfrak
s_{\alpha,i,j}^{-1}(0) \cap (\{w\} \times U_i)$ is a smooth
manifold. Hence the right hand side of (\ref{intfibermainformula})
can be regarded as an average of the correspondence by $\mathfrak
s_{\alpha,i,j}^{-1}(0) \cap (\{w\} \times U_i)$ over $w$. We can
also represent the smooth form $h$ by an appropriate average (with
respect to certain smooth measure) of a family of currents realized
by smooth singular chains. So, as far as sign concerns, it suffices
to consider a current realized by a smooth singular chain. Then the
right hand side of (\ref{intfibermainformula}) turn out to be a
current realized by a smooth singular chain which is obtained from a
smooth singular chain on $M_s$ by a transversal smooth
correspondence. In fact, we may assume that all the fiber products
appearing here are transversal, since it suffices to discuss the
sign at the generic point where the transversality holds. Thus the
problem reduces to find a sign convention (and orientation) for
correspondence of the singular chains by a smooth manifold. In the
situation of our application, such sign convention (singular
homology version) was determined and analyzed in detail in 
Chapter 8 \cite{fooo06} (= Chapter 9 \cite{fooo06pre}). Especially the existence of an appropriate
orientation that is consistent with the sign appearing in
$A_{\infty}$ formulae etc. was proved there. Therefore we can prove
that there is a sign (orientation) convention which induces all the
formulae we need with sign, in our de Rham version, as well. See
Subsection 8.10.3 \cite{fooo06} (= Section 53.3 \cite{fooo06pre}) or Section 13 \cite{fukaya;operad} for
detail.

\bibliographystyle{amsnumber}

\end{document}